\documentclass[11pt,a4paper]{amsart}
\usepackage[margin=1in]{geometry}
\usepackage{amsmath,amssymb,amsthm,mathtools}
\usepackage[T1]{fontenc}
\usepackage{lmodern}
\usepackage[colorlinks=true,linkcolor=blue,urlcolor=blue,citecolor=blue]{hyperref}
\usepackage{enumitem}
\usepackage{graphicx}
\usepackage{tikz}
\usepackage{needspace}
\usepackage{xcolor}
\usepackage{comment}
\newenvironment{mathematicalcorrection}{\begingroup\color{red}}{\endgroup}
\newcommand{\mathfix}[1]{{\color{red}#1}}
\newcommand{\R}{\mathbb R}
\newcommand{\eps}{\varepsilon}
\newcommand{\Cyl}{C_{3,3}\times\R}
\DeclareMathOperator{\Sing}{Sing}

\newtheorem{theorem}{Theorem}[section]
\newtheorem{lemma}[theorem]{Lemma}
\newtheorem{proposition}[theorem]{Proposition}
\newtheorem{corollary}[theorem]{Corollary}
\newtheorem{conjecture}{Follow-up question}
\theoremstyle{definition}
\newtheorem{definition}[theorem]{Definition}
\theoremstyle{remark}
\newtheorem{remark}[theorem]{Remark}
\newtheorem*{remark*}{Remark}

\makeatletter
\renewcommand{\tocsection}[3]{%
  \indentlabel{\@ifnotempty{#2}{\bfseries\ignorespaces\makebox[\@ifempty{#1}{25pt}{75pt}][l]{#1 #2\quad}}}\bfseries#3}
\renewcommand{\tocsubsection}[3]{%
  \indentlabel{\@ifnotempty{#2}{\ignorespaces\makebox[30pt][l]{#1 #2\quad}}}#3}
\newcommand\@dotsep{4}
\def\@tocline#1#2#3#4#5#6#7{\relax
  \ifnum #1>\c@tocdepth \else
    \par \addpenalty\@secpenalty\addvspace{#2}%
    \begingroup \hyphenpenalty\@M
    \@ifempty{#4}{\@tempdima\csname r@tocindent\number#1\endcsname\relax}{\@tempdima#4\relax}%
    \parindent\z@ \leftskip#3\relax \advance\leftskip\@tempdima\relax
    \rightskip\@pnumwidth plus1em \parfillskip-\@pnumwidth
    #5\leavevmode\hskip-\@tempdima{#6}\nobreak
    \leaders\hbox{$\m@th\mkern \@dotsep mu\hbox{.}\mkern \@dotsep mu$}\hfill
    \nobreak
    \hbox to\@pnumwidth{\@tocpagenum{\ifnum#1=1\bfseries\fi#7}}\par
    \nobreak
    \endgroup
  \fi}
\def\l@section{\@tocline{1}{0pt}{1pc}{25pt}{}}
\def\l@subsection{\@tocline{2}{0pt}{2pc}{30pt}{}}
\makeatother

\title[Conical to cylindrical singularities]{On Cylindrical singularities of minimal hypersurfaces }
\author[A. Halavati]{Aria Halavati}
\address{Department of Decision Sciences and BIDSA, Bocconi University, Milano, Italy}
\email{aria.halavati@unibocconi.it}

\author[L. Spolaor]{Luca Spolaor}
\address{Department of Mathematics, UC San Diego, AP\&M, La Jolla, California, USA}
\email{lspolaor@ucsd.edu}

\begin{document}
\begin{abstract}
In a neighborhood of the origin in $\R^9$, we construct a minimal
hypersurface with respect to a smooth conformally
Euclidean metric, whose singular set consists of the origin and a
sequence of isolated singular points converging to it. At each of
the isolated singular points the hypersurface has the unique tangent
cone $C_{4,3}$, the cone over
$S^4(\sqrt{4/7})\times S^3(\sqrt{3/7})$, whereas at the origin its
unique tangent cone is the cone $C_{3,3}\times\R$. The metric is Euclidean near each isolated singular
point and agrees with the Euclidean metric to infinite order at the
origin. A key step, of independent interest, is the construction of
an entire stationary hypersurface in Euclidean $\R^9$ with exactly
one singular point at the origin, at which its unique tangent cone is $C_{4,3}$,
and whose unique tangent cone at infinity is $C_{3,3}\times\R$.
\end{abstract}
\maketitle

\medskip
\begingroup\small
\tableofcontents
\endgroup

\section{Introduction}\label{sec:introduction}
For an area-minimizing hypersurface of dimension $n$, Federer's
dimension-reduction theorem bounds the Hausdorff dimension of the
interior singular set by $n-7$ \cite{Fed}. Such hypersurfaces are
therefore smooth in ambient dimensions at most seven. Singularities
first occur in ambient dimension eight, as exemplified by the Simons
cone, proved to be area-minimizing by Bombieri--De Giorgi--Giusti
\cite{BDGG}. Its Euclidean products give higher-dimensional singular
sets, whose possible geometry becomes increasingly complicated as
the dimension grows. In general, the strongest available structural
conclusions are rectifiability and quantitative measure bounds:
for $n\ge7$, Simon \cite{Sim95} established countable $(n-7)$-rectifiability of
the singular set, and Naber--Valtorta \cite{NV} proved local
finiteness and quantitative bounds for its $(n-7)$-dimensional
Hausdorff measure. A finer local geometric description remains
unavailable in general.

The geometry of a minimal hypersurface near a singular point is
described to first order by its tangent cones. If one tangent cone
has multiplicity one and is regular, meaning that it is smooth away
from its vertex, the situation is well understood. Building on the
integrable case treated by Allard--Almgren \cite{AA}, Simon's
asymptotic \L ojasiewicz theory \cite{Sim83,Sim85} proves uniqueness of the tangent
cone and represents the hypersurface as a normal graph over it in
a punctured neighborhood of the singular point: that is in a neighborhood of the singular point the minimal hypersurface is homeomorphic to its tangent model. See also Adams--Simon \cite{AS} for convergence rates
and Engelstein--Spolaor--Velichkov \cite{ESV} for the corresponding
regularity theory for almost area-minimizing currents.

Much less is understood at points with cylindrical tangent cones,
that is, cones of the form $C\times\R^\ell$, where $C$ has an
isolated vertex singularity and $\ell\ge1$. The singular set
$\{0\}\times\R^\ell$ of such a cone, called spine, is not isolated, and
convergence to the cone away from the spine does not by itself
describe the hypersurface near the spine itself.
For quadratic cylindrical cones with a one-dimensional spine, a major
breakthrough was Sz\'ekelyhidi's uniqueness theorem for the Simons
case $C_{3,3}\times\R$ \cite{Sz}, beyond the integrable cases
covered by Simon \cite{Sim94}. Firester--Tsiamis--Wang
\cite{FTWUniqueness} subsequently treated $C_{2,4}\times\R$.
Together these works establish uniqueness for multiplicity-one
tangent cones $C_{p,q}\times\R$ of area-minimizing hypersurfaces.

Uniqueness of the tangent cone leaves open the structure of the
singular set near the point. In a striking construction, Simon
\cite{SimPrescribed} proved the existence of a stable minimal hypersurface in a smooth non euclidean metric in $\R^9$ whose singular set is an arbitrary closed
subset of the singular axis of the Simons cone times $\R$. In Euclidean space, Sz\'ekelyhidi \cite{SzCyl} constructed hypersurfaces with an isolated
singular point whose unique tangent cone is nevertheless cylindrical;
his construction gives area-minimizing examples when the
cross-section is strictly minimizing and strictly stable.
Thus the nearby singular set need not resemble the spine of its
tangent cone. Beyond such constructions, general structural results
such as \cite{Sim93}, and classifications under additional symmetry
\cite{SzCyl}, a detailed local description remains largely open,
even for quadratic cylindrical tangent cones.

Motivated by the question of better understanding the behavior of area minimizing hypersurfaces close to a cylindrical tangent cone, in this note we prove the following theorem. In what follows we work in $\R^9=\R_x^4\times\R_z^4\times\R_y$ and, for integers $p,q\geq 1$ we use the convention
\[
 C_{p,q}=\{(x,z)\in\R^{p+1}\times\R^{q+1}:
                  q|x|^2=p|z|^2\}.
\]
The two cones relevant to this paper are
\[
 V=C_{3,3}\times\R=\{|z|=|x|\},\qquad
 C_{4,3}=\{4|z|^2=3(|x|^2+y^2)\}.
\]
The first has a singular axis; the second is singular only at its
vertex.

\begin{theorem}[Accumulating singularities]
\label{thm:main-accumulation}
There exist an open neighborhood $\mathcal U$ of zero in $\R^9$,
a conformally Euclidean $C^\infty$ metric $g$ on $\mathcal U$,
and an embedded multiplicity-one hypersurface $\Sigma$,
stationary with respect to $g$, such that, for some $k_0$,
\[
 \Sing\Sigma=\{0\}\cup\{p_k:k\ge k_0\},
 \qquad p_k=(0,0,2^{-k}),
\]
and
\[
 \operatorname{Tan}(\Sigma,0)=\{C_{3,3}\times\R\},
 \qquad
 \operatorname{Tan}(\Sigma,p_k)=\{C_{4,3}\}.
\]
All these tangent cones have multiplicity one.

The metric is Euclidean near each $p_k$ and agrees with the
Euclidean metric to infinite order at zero:
\[
 D^j(g-g_{\rm euc})(0)=0
 \qquad\text{for every integer }j\ge0.
\]
More precisely, for every pair of nonnegative integers $m,N$,
there is a constant $C_{m,N}$ such that
\[
 \|g-g_{\rm euc}\|_{C^m(B_\rho(0))}
 \le C_{m,N}\rho^N
\]
for all sufficiently small $\rho>0$.
In particular, the metric can be made arbitrarily close to
Euclidean in any prescribed finite collection of $C^m$ norms
by shrinking $\mathcal U$.
\end{theorem}

Here and below, $\operatorname{Tan}(\Sigma,p)$ denotes the set
of tangent cones at $p$ of the varifold associated with $\Sigma$. To
our knowledge, Theorem~\ref{thm:main-accumulation} provides the first
example of a minimal hypersurface, with respect to a smooth metric, in
which singular points whose tangent cones have an isolated
singularity accumulate at a point with a cylindrical tangent cone. 

\begin{remark*}
We do not address here whether $\Sigma$ is stable, or
area-minimizing, with respect to $g$. In view of the recent work of
Liu \cite{Liu2026} on Simon's examples, we expect that this question
can be settled by similar methods, but we do not pursue it here.
\end{remark*}

As is well known (see, for instance, the work of
Edelen--Spolaor \cite{ESp}, which builds on Simon--Solomon
\cite{SS}), local regularity questions near a cone are closely tied
to Liouville-type theorems, that is, to the classification of complete
minimal hypersurfaces with a prescribed tangent cone at infinity. An
entire model interpolating between the two cones of
Theorem~\ref{thm:main-accumulation} is therefore an essentially
necessary ingredient of the construction; its existence is the
content of the following result.

\begin{theorem}[Existence of the connector]\label{thm:main-connector}
There exists an entire Lipschitz function
$U:\R^4_x\times\R_y\to[0,\infty)$ whose symmetric graph
$M=\{|z|=U(x,y)\}$ is a complete multiplicity-one stationary
eight-dimensional hypersurface in Euclidean $\R^9$, with
\[
 \Sing M=\{0\},\qquad
 \operatorname{Tan}(M,0)=\{C_{4,3}\},\qquad
 \lim_{R\to\infty}R^{-1}M=C_{3,3}\times \R.
\]
The last limit holds in the sense of varifolds, so that the
tangent cone at infinity is unique, and the convergence is smooth on
compact subsets of $\R^9\setminus\{x=z=0\}$. The function $U$ is $O(4)_x$-invariant,
even in $y$, smooth away from zero, and $U_y>0$ for $y>0$.
\end{theorem}

We call $M$ a \emph{connector}, since it connects a
singularity modeled on $C_{4,3}$ at the origin with the cylindrical
cone $C_{3,3}\times\R$ at infinity. Besides its role in the proof of
Theorem~\ref{thm:main-accumulation}, its existence is of interest in
connection with the Bernstein-type problem for minimal hypersurfaces
asymptotic to $C_{3,3}\times\R$ studied in \cite{ES}.
A special case of \cite[Theorem~1.1]{ES}, together with its stationary-varifold
extension in Remark~1.2, asserts that if a nonzero stationary integral
eight-varifold in all of $\R^9$ lies entirely
on one side of $V=C_{3,3}\times \R$ and has density at infinity less than
$2\Theta(V)$, then it is a leaf product $H(\lambda)\times\R$, where $H(\lambda)$ is a leaf of the Simons foliation. The connector in
Theorem~\ref{thm:main-connector} satisfies the density condition,
$\Theta_M(\infty)=\Theta(V)$, but meets both sides of $V$. Indeed, the two sides are distinguished by
$q(x,z,y)=|z|^2-|x|^2$, and on its tangent cone
\[
 C_{4,3}:\qquad
 q=\frac34(|x|^2+y^2)-|x|^2
   =\frac{3y^2-|x|^2}{4}.
\]
This is negative at $y=0$, $x\ne0$, and positive at $x=0$,
$y\ne0$. Containment in either closed side would persist under
dilation and passage to a tangent varifold. Since $C_{4,3}$ is
contained in neither side, $M$ meets both open sides in every
neighborhood of zero. Thus the one-sidedness hypothesis of
\cite[Theorem~1.1 and Remark~1.2]{ES} fails.

Write $U(x,y)=u_\infty(t,y)$, where $t=|x|$, for the radial
profile of the connector. At a fixed height $y\ne0$, its orbit at $x=0$
is the sphere $\{|z|=U(0,y)\}\cong S^3$. We call its radius
$a_\infty(y)=U(0,y)$ the \emph{smoothing radius}: it sets
the scale on which the transverse cone vertex is rounded off.
The \emph{smoothing region} is the part of the slice where
$|x|$ is comparable to or smaller than this radius.
Figure~\ref{fig:connector-slice-evolution} illustrates these slices
and the growth of their smoothing radii away from the singular origin.
\begin{figure}[p]
\centering
\begingroup
\definecolor{connectorblue}{HTML}{157D96}
\definecolor{connectororange}{HTML}{CF713F}
\definecolor{connectorink}{HTML}{223442}
\definecolor{connectorgray}{HTML}{909BA2}
\definecolor{connectorpale}{HTML}{D8DDE0}
\newcommand{\connectorsliceaxes}[1]{%
  \foreach \tick in {0,2,4,6}{%
    \draw[connectorpale!65,line width=.25pt] (\tick,0)--(\tick,6.5);
    \draw[connectorpale!65,line width=.25pt] (0,\tick)--(6,\tick);
    \draw[connectorgray,line width=.4pt] (\tick,0)--(\tick,-.1);
    \draw[connectorgray,line width=.4pt] (0,\tick)--(-.1,\tick);
    \node[below,inner sep=2pt] at (\tick,0) {$\tick$};
    \node[left,inner sep=2pt] at (0,\tick) {$\tick$};
  }
  \draw[connectorgray,line width=.5pt] (0,6.5)--(0,0)--(6,0);
  \node at (3,7) {$y=#1$};
  \node at (3,-.9) {$|x|$};
  \node[rotate=90] at (-1.1,3.25) {$|z|$};
}
\resizebox{\textwidth}{!}{%
\begin{tikzpicture}[x=.59cm,y=.59cm,font=\sffamily\footnotesize,
                   text=connectorink,line join=round]
\path[use as bounding box] (0,0) rectangle (26,30);
\begin{scope}[shift={(3.5,22.5)}]
\connectorsliceaxes{0}
\def\slicebelow{
(0.000000,0.000000) (0.050000,0.043140) (0.100000,0.086031) (0.150000,0.128764) (0.200000,0.171412)
(0.250000,0.214032) (0.300000,0.256669) (0.350000,0.299357) (0.400000,0.342125) (0.450000,0.384992)
(0.500000,0.427975) (0.550000,0.471086) (0.600000,0.514335) (0.650000,0.557726) (0.700000,0.601264)
(0.750000,0.644952) (0.800000,0.688791) (0.850000,0.732780) (0.900000,0.776919) (0.950000,0.821205)
(1.000000,0.865638) (1.050000,0.910213) (1.100000,0.954930) (1.150000,0.999783) (1.200000,1.044770)
(1.250000,1.089888) (1.300000,1.135133) (1.350000,1.180501) (1.400000,1.225989) (1.450000,1.271594)
(1.500000,1.317311) (1.550000,1.363138) (1.600000,1.409072) (1.650000,1.455108) (1.700000,1.501244)
(1.750000,1.547476) (1.800000,1.593802) (1.850000,1.640218) (1.900000,1.686722) (1.950000,1.733311)
(2.000000,1.779982) (2.050000,1.826733) (2.100000,1.873560) (2.150000,1.920462) (2.200000,1.967437)
(2.250000,2.014481) (2.300000,2.061593) (2.350000,2.108771) (2.400000,2.156012) (2.450000,2.203315)
(2.500000,2.250678) (2.550000,2.298098) (2.600000,2.345575) (2.650000,2.393107) (2.700000,2.440691)
(2.750000,2.488326) (2.800000,2.536011) (2.850000,2.583745) (2.900000,2.631525) (2.950000,2.679351)
(3.000000,2.727221) (3.050000,2.775134) (3.100000,2.823089) (3.150000,2.871084) (3.200000,2.919119)
(3.250000,2.967192) (3.300000,3.015302) (3.350000,3.063448) (3.400000,3.111630) (3.450000,3.159846)
(3.500000,3.208095) (3.550000,3.256377) (3.600000,3.304690) (3.650000,3.353034) (3.700000,3.401408)
(3.750000,3.449811) (3.800000,3.498242) (3.850000,3.546702) (3.900000,3.595188) (3.950000,3.643700)
(4.000000,3.692239) (4.050000,3.740802) (4.100000,3.789389) (4.150000,3.838001) (4.200000,3.886635)
(4.250000,3.935293) (4.300000,3.983972) (4.350000,4.032673) (4.400000,4.081395) (4.450000,4.130137)
(4.500000,4.178900) (4.550000,4.227682) (4.600000,4.276484) (4.650000,4.325304) (4.700000,4.374142)
(4.750000,4.422999) (4.800000,4.471873) (4.850000,4.520764) (4.900000,4.569671) (4.950000,4.618596)
(5.000000,4.667536) (5.050000,4.716491) (5.100000,4.765462) (5.150000,4.814448) (5.200000,4.863449)
(5.250000,4.912464) (5.300000,4.961493) (5.350000,5.010536) (5.400000,5.059592) (5.450000,5.108662)
(5.500000,5.157744) (5.550000,5.206839) (5.600000,5.255947) (5.650000,5.305067) (5.700000,5.354198)
(5.750000,5.403341) (5.800000,5.452496) (5.850000,5.501662) (5.900000,5.550839) (5.950000,5.600027)
(6.000000,5.649225)}
\begin{scope}
\clip (0,0) rectangle (6,6.5);
\draw[connectorgray,line width=.6pt,dashed] (0,0)--(6,6);
\draw[connectorgray,line width=.7pt,densely dotted] plot coordinates {
(0.000000,0.000000) (0.050000,0.043301) (0.100000,0.086603) (0.150000,0.129904) (0.200000,0.173205)
(0.250000,0.216506) (0.300000,0.259808) (0.350000,0.303109) (0.400000,0.346410) (0.450000,0.389711)
(0.500000,0.433013) (0.550000,0.476314) (0.600000,0.519615) (0.650000,0.562917) (0.700000,0.606218)
(0.750000,0.649519) (0.800000,0.692820) (0.850000,0.736122) (0.900000,0.779423) (0.950000,0.822724)
(1.000000,0.866025) (1.050000,0.909327) (1.100000,0.952628) (1.150000,0.995929) (1.200000,1.039230)
(1.250000,1.082532) (1.300000,1.125833) (1.350000,1.169134) (1.400000,1.212436) (1.450000,1.255737)
(1.500000,1.299038) (1.550000,1.342339) (1.600000,1.385641) (1.650000,1.428942) (1.700000,1.472243)
(1.750000,1.515544) (1.800000,1.558846) (1.850000,1.602147) (1.900000,1.645448) (1.950000,1.688750)
(2.000000,1.732051) (2.050000,1.775352) (2.100000,1.818653) (2.150000,1.861955) (2.200000,1.905256)
(2.250000,1.948557) (2.300000,1.991858) (2.350000,2.035160) (2.400000,2.078461) (2.450000,2.121762)
(2.500000,2.165064) (2.550000,2.208365) (2.600000,2.251666) (2.650000,2.294967) (2.700000,2.338269)
(2.750000,2.381570) (2.800000,2.424871) (2.850000,2.468172) (2.900000,2.511474) (2.950000,2.554775)
(3.000000,2.598076) (3.050000,2.641377) (3.100000,2.684679) (3.150000,2.727980) (3.200000,2.771281)
(3.250000,2.814583) (3.300000,2.857884) (3.350000,2.901185) (3.400000,2.944486) (3.450000,2.987788)
(3.500000,3.031089) (3.550000,3.074390) (3.600000,3.117691) (3.650000,3.160993) (3.700000,3.204294)
(3.750000,3.247595) (3.800000,3.290897) (3.850000,3.334198) (3.900000,3.377499) (3.950000,3.420800)
(4.000000,3.464102) (4.050000,3.507403) (4.100000,3.550704) (4.150000,3.594005) (4.200000,3.637307)
(4.250000,3.680608) (4.300000,3.723909) (4.350000,3.767211) (4.400000,3.810512) (4.450000,3.853813)
(4.500000,3.897114) (4.550000,3.940416) (4.600000,3.983717) (4.650000,4.027018) (4.700000,4.070319)
(4.750000,4.113621) (4.800000,4.156922) (4.850000,4.200223) (4.900000,4.243524) (4.950000,4.286826)
(5.000000,4.330127) (5.050000,4.373428) (5.100000,4.416730) (5.150000,4.460031) (5.200000,4.503332)
(5.250000,4.546633) (5.300000,4.589935) (5.350000,4.633236) (5.400000,4.676537) (5.450000,4.719838)
(5.500000,4.763140) (5.550000,4.806441) (5.600000,4.849742) (5.650000,4.893044) (5.700000,4.936345)
(5.750000,4.979646) (5.800000,5.022947) (5.850000,5.066249) (5.900000,5.109550) (5.950000,5.152851)
(6.000000,5.196152)};
\fill[connectororange,opacity=.085] plot coordinates {\slicebelow} --(6,6)--cycle;
\draw[connectororange,line width=1pt] plot coordinates {\slicebelow};
\end{scope}
\fill[connectorink] (0,0) circle[radius=1.6pt];
\end{scope}
\begin{scope}[shift={(16.5,22.5)}]
\connectorsliceaxes{0.25}
\def\slicebelow{
(0.378054,0.378054) (0.400000,0.393813) (0.450000,0.430713) (0.500000,0.468766) (0.550000,0.507752)
(0.600000,0.547510) (0.650000,0.587918) (0.700000,0.628883) (0.750000,0.670332) (0.800000,0.712210)
(0.850000,0.754469) (0.900000,0.797073) (0.950000,0.839990) (1.000000,0.883195) (1.050000,0.926665)
(1.100000,0.970381) (1.150000,1.014327) (1.200000,1.058487) (1.250000,1.102849) (1.300000,1.147401)
(1.350000,1.192132) (1.400000,1.237033) (1.450000,1.282095) (1.500000,1.327310) (1.550000,1.372671)
(1.600000,1.418170) (1.650000,1.463802) (1.700000,1.509561) (1.750000,1.555441) (1.800000,1.601436)
(1.850000,1.647543) (1.900000,1.693755) (1.950000,1.740070) (2.000000,1.786483) (2.050000,1.832991)
(2.100000,1.879589) (2.150000,1.926274) (2.200000,1.973042) (2.250000,2.019892) (2.300000,2.066819)
(2.350000,2.113822) (2.400000,2.160897) (2.450000,2.208041) (2.500000,2.255253) (2.550000,2.302531)
(2.600000,2.349871) (2.650000,2.397271) (2.700000,2.444731) (2.750000,2.492247) (2.800000,2.539818)
(2.850000,2.587443) (2.900000,2.635119) (2.950000,2.682845) (3.000000,2.730619) (3.050000,2.778440)
(3.100000,2.826307) (3.150000,2.874217) (3.200000,2.922170) (3.250000,2.970165) (3.300000,3.018200)
(3.350000,3.066274) (3.400000,3.114385) (3.450000,3.162534) (3.500000,3.210718) (3.550000,3.258938)
(3.600000,3.307190) (3.650000,3.355476) (3.700000,3.403794) (3.750000,3.452143) (3.800000,3.500522)
(3.850000,3.548931) (3.900000,3.597368) (3.950000,3.645833) (4.000000,3.694325) (4.050000,3.742844)
(4.100000,3.791389) (4.150000,3.839958) (4.200000,3.888553) (4.250000,3.937171) (4.300000,3.985812)
(4.350000,4.034476) (4.400000,4.083162) (4.450000,4.131870) (4.500000,4.180599) (4.550000,4.229349)
(4.600000,4.278119) (4.650000,4.326908) (4.700000,4.375716) (4.750000,4.424544) (4.800000,4.473389)
(4.850000,4.522253) (4.900000,4.571133) (4.950000,4.620031) (5.000000,4.668946) (5.050000,4.717877)
(5.100000,4.766824) (5.150000,4.815786) (5.200000,4.864764) (5.250000,4.913757) (5.300000,4.962764)
(5.350000,5.011786) (5.400000,5.060822) (5.450000,5.109871) (5.500000,5.158934) (5.550000,5.208010)
(5.600000,5.257098) (5.650000,5.306200) (5.700000,5.355314) (5.750000,5.404439) (5.800000,5.453577)
(5.850000,5.502726) (5.900000,5.551887) (5.950000,5.601059) (6.000000,5.650242)}
\def\sliceabove{
(0.000000,0.208508) (0.050000,0.212447) (0.100000,0.223946) (0.150000,0.242120) (0.200000,0.265762)
(0.250000,0.293660) (0.300000,0.324789) (0.350000,0.358366) (0.378054,0.378054)}
\begin{scope}
\clip (0,0) rectangle (6,6.5);
\draw[connectorgray,line width=.6pt,dashed] (0,0)--(6,6);
\draw[connectorgray,line width=.7pt,densely dotted] plot coordinates {
(0.000000,0.216506) (0.050000,0.220794) (0.100000,0.233184) (0.150000,0.252488) (0.200000,0.277263)
(0.250000,0.306186) (0.300000,0.338194) (0.350000,0.372492) (0.400000,0.408503) (0.450000,0.445814)
(0.500000,0.484123) (0.550000,0.523211) (0.600000,0.562917) (0.650000,0.603117) (0.700000,0.643720)
(0.750000,0.684653) (0.800000,0.725862) (0.850000,0.767300) (0.900000,0.808934) (0.950000,0.850735)
(1.000000,0.892679) (1.050000,0.934746) (1.100000,0.976921) (1.150000,1.019191) (1.200000,1.061544)
(1.250000,1.103970) (1.300000,1.146462) (1.350000,1.189012) (1.400000,1.231615) (1.450000,1.274264)
(1.500000,1.316957) (1.550000,1.359687) (1.600000,1.402453) (1.650000,1.445251) (1.700000,1.488078)
(1.750000,1.530931) (1.800000,1.573809) (1.850000,1.616710) (1.900000,1.659631) (1.950000,1.702572)
(2.000000,1.745530) (2.050000,1.788505) (2.100000,1.831495) (2.150000,1.874500) (2.200000,1.917518)
(2.250000,1.960548) (2.300000,2.003591) (2.350000,2.046644) (2.400000,2.089707) (2.450000,2.132780)
(2.500000,2.175862) (2.550000,2.218952) (2.600000,2.262051) (2.650000,2.305157) (2.700000,2.348271)
(2.750000,2.391391) (2.800000,2.434517) (2.850000,2.477650) (2.900000,2.520789) (2.950000,2.563933)
(3.000000,2.607082) (3.050000,2.650236) (3.100000,2.693395) (3.150000,2.736558) (3.200000,2.779726)
(3.250000,2.822897) (3.300000,2.866073) (3.350000,2.909252) (3.400000,2.952435) (3.450000,2.995622)
(3.500000,3.038811) (3.550000,3.082004) (3.600000,3.125200) (3.650000,3.168399) (3.700000,3.211600)
(3.750000,3.254804) (3.800000,3.298011) (3.850000,3.341220) (3.900000,3.384431) (3.950000,3.427645)
(4.000000,3.470861) (4.050000,3.514079) (4.100000,3.557299) (4.150000,3.600521) (4.200000,3.643745)
(4.250000,3.686970) (4.300000,3.730198) (4.350000,3.773427) (4.400000,3.816658) (4.450000,3.859890)
(4.500000,3.903124) (4.550000,3.946359) (4.600000,3.989596) (4.650000,4.032834) (4.700000,4.076073)
(4.750000,4.119314) (4.800000,4.162556) (4.850000,4.205800) (4.900000,4.249044) (4.950000,4.292290)
(5.000000,4.335536) (5.050000,4.378784) (5.100000,4.422033) (5.150000,4.465283) (5.200000,4.508534)
(5.250000,4.551785) (5.300000,4.595038) (5.350000,4.638292) (5.400000,4.681546) (5.450000,4.724802)
(5.500000,4.768058) (5.550000,4.811315) (5.600000,4.854573) (5.650000,4.897831) (5.700000,4.941090)
(5.750000,4.984351) (5.800000,5.027611) (5.850000,5.070873) (5.900000,5.114135) (5.950000,5.157398)
(6.000000,5.200661)};
\fill[connectorblue,opacity=.085] plot coordinates {\sliceabove} --(0,0)--cycle;
\fill[connectororange,opacity=.085] plot coordinates {\slicebelow} --(6,6)--cycle;
\draw[connectorblue,line width=1pt] plot coordinates {\sliceabove};
\draw[connectororange,line width=1pt] plot coordinates {\slicebelow};
\end{scope}
\draw[connectorblue,fill=white,line width=.65pt] (0,0.208508) circle[radius=1.8pt];
\fill[connectorink] (0.378054,0.378054) circle[radius=1.45pt];
\end{scope}
\begin{scope}[shift={(3.5,13.25)}]
\connectorsliceaxes{0.5}
\def\slicebelow{
(0.715257,0.715257) (0.750000,0.740483) (0.800000,0.777578) (0.850000,0.815514) (0.900000,0.854201)
(0.950000,0.893560) (1.000000,0.933526) (1.050000,0.974040) (1.100000,1.015052) (1.150000,1.056518)
(1.200000,1.098399) (1.250000,1.140662) (1.300000,1.183277) (1.350000,1.226216) (1.400000,1.269457)
(1.450000,1.312977) (1.500000,1.356759) (1.550000,1.400784) (1.600000,1.445037) (1.650000,1.489504)
(1.700000,1.534172) (1.750000,1.579030) (1.800000,1.624065) (1.850000,1.669270) (1.900000,1.714634)
(1.950000,1.760149) (2.000000,1.805807) (2.050000,1.851601) (2.100000,1.897525) (2.150000,1.943572)
(2.200000,1.989736) (2.250000,2.036013) (2.300000,2.082396) (2.350000,2.128881) (2.400000,2.175465)
(2.450000,2.222141) (2.500000,2.268908) (2.550000,2.315760) (2.600000,2.362694) (2.650000,2.409708)
(2.700000,2.456798) (2.750000,2.503961) (2.800000,2.551194) (2.850000,2.598495) (2.900000,2.645861)
(2.950000,2.693289) (3.000000,2.740778) (3.050000,2.788326) (3.100000,2.835930) (3.150000,2.883588)
(3.200000,2.931298) (3.250000,2.979060) (3.300000,3.026870) (3.350000,3.074728) (3.400000,3.122632)
(3.450000,3.170580) (3.500000,3.218571) (3.550000,3.266604) (3.600000,3.314677) (3.650000,3.362789)
(3.700000,3.410939) (3.750000,3.459126) (3.800000,3.507348) (3.850000,3.555606) (3.900000,3.603897)
(3.950000,3.652220) (4.000000,3.700575) (4.050000,3.748961) (4.100000,3.797377) (4.150000,3.845823)
(4.200000,3.894296) (4.250000,3.942797) (4.300000,3.991325) (4.350000,4.039879) (4.400000,4.088459)
(4.450000,4.137063) (4.500000,4.185691) (4.550000,4.234343) (4.600000,4.283018) (4.650000,4.331715)
(4.700000,4.380433) (4.750000,4.429173) (4.800000,4.477934) (4.850000,4.526715) (4.900000,4.575516)
(4.950000,4.624335) (5.000000,4.673174) (5.050000,4.722031) (5.100000,4.770905) (5.150000,4.819798)
(5.200000,4.868707) (5.250000,4.917633) (5.300000,4.966575) (5.350000,5.015533) (5.400000,5.064507)
(5.450000,5.113496) (5.500000,5.162500) (5.550000,5.211518) (5.600000,5.260551) (5.650000,5.309597)
(5.700000,5.358658) (5.750000,5.407731) (5.800000,5.456818) (5.850000,5.505917) (5.900000,5.555029)
(5.950000,5.604154) (6.000000,5.653290)}
\def\sliceabove{
(0.000000,0.403553) (0.050000,0.405425) (0.100000,0.411026) (0.150000,0.420297) (0.200000,0.433122)
(0.250000,0.449326) (0.300000,0.468675) (0.350000,0.490898) (0.400000,0.515707) (0.450000,0.542817)
(0.500000,0.571959) (0.550000,0.602886) (0.600000,0.635381) (0.650000,0.669252) (0.700000,0.704334)
(0.715257,0.715257)}
\begin{scope}
\clip (0,0) rectangle (6,6.5);
\draw[connectorgray,line width=.6pt,dashed] (0,0)--(6,6);
\draw[connectorgray,line width=.7pt,densely dotted] plot coordinates {
(0.000000,0.433013) (0.050000,0.435172) (0.100000,0.441588) (0.150000,0.452079) (0.200000,0.466369)
(0.250000,0.484123) (0.300000,0.504975) (0.350000,0.528559) (0.400000,0.554527) (0.450000,0.582559)
(0.500000,0.612372) (0.550000,0.643720) (0.600000,0.676387) (0.650000,0.710194) (0.700000,0.744983)
(0.750000,0.780625) (0.800000,0.817007) (0.850000,0.854035) (0.900000,0.891628) (0.950000,0.929718)
(1.000000,0.968246) (1.050000,1.007162) (1.100000,1.046422) (1.150000,1.085990) (1.200000,1.125833)
(1.250000,1.165922) (1.300000,1.206234) (1.350000,1.246746) (1.400000,1.287439) (1.450000,1.328298)
(1.500000,1.369306) (1.550000,1.410452) (1.600000,1.451723) (1.650000,1.493109) (1.700000,1.534601)
(1.750000,1.576190) (1.800000,1.617869) (1.850000,1.659631) (1.900000,1.701470) (1.950000,1.743380)
(2.000000,1.785357) (2.050000,1.827396) (2.100000,1.869492) (2.150000,1.911642) (2.200000,1.953842)
(2.250000,1.996090) (2.300000,2.038382) (2.350000,2.080715) (2.400000,2.123087) (2.450000,2.165496)
(2.500000,2.207940) (2.550000,2.250417) (2.600000,2.292924) (2.650000,2.335460) (2.700000,2.378024)
(2.750000,2.420615) (2.800000,2.463230) (2.850000,2.505868) (2.900000,2.548529) (2.950000,2.591211)
(3.000000,2.633913) (3.050000,2.676635) (3.100000,2.719375) (3.150000,2.762132) (3.200000,2.804906)
(3.250000,2.847696) (3.300000,2.890502) (3.350000,2.933321) (3.400000,2.976155) (3.450000,3.019002)
(3.500000,3.061862) (3.550000,3.104734) (3.600000,3.147618) (3.650000,3.190513) (3.700000,3.233419)
(3.750000,3.276336) (3.800000,3.319262) (3.850000,3.362198) (3.900000,3.405143) (3.950000,3.448097)
(4.000000,3.491060) (4.050000,3.534031) (4.100000,3.577010) (4.150000,3.619997) (4.200000,3.662991)
(4.250000,3.705992) (4.300000,3.749000) (4.350000,3.792015) (4.400000,3.835036) (4.450000,3.878063)
(4.500000,3.921097) (4.550000,3.964136) (4.600000,4.007181) (4.650000,4.050231) (4.700000,4.093287)
(4.750000,4.136348) (4.800000,4.179414) (4.850000,4.222484) (4.900000,4.265560) (4.950000,4.308640)
(5.000000,4.351724) (5.050000,4.394812) (5.100000,4.437905) (5.150000,4.481002) (5.200000,4.524102)
(5.250000,4.567206) (5.300000,4.610315) (5.350000,4.653426) (5.400000,4.696541) (5.450000,4.739660)
(5.500000,4.782782) (5.550000,4.825907) (5.600000,4.869035) (5.650000,4.912166) (5.700000,4.955300)
(5.750000,4.998437) (5.800000,5.041577) (5.850000,5.084720) (5.900000,5.127865) (5.950000,5.171013)
(6.000000,5.214163)};
\fill[connectorblue,opacity=.085] plot coordinates {\sliceabove} --(0,0)--cycle;
\fill[connectororange,opacity=.085] plot coordinates {\slicebelow} --(6,6)--cycle;
\draw[connectorblue,line width=1pt] plot coordinates {\sliceabove};
\draw[connectororange,line width=1pt] plot coordinates {\slicebelow};
\end{scope}
\draw[connectorblue,fill=white,line width=.65pt] (0,0.403553) circle[radius=1.8pt];
\fill[connectorink] (0.715257,0.715257) circle[radius=1.45pt];
\end{scope}
\begin{scope}[shift={(16.5,13.25)}]
\connectorsliceaxes{1}
\def\slicebelow{
(1.357597,1.357597) (1.400000,1.389722) (1.450000,1.428146) (1.500000,1.467123) (1.550000,1.506615)
(1.600000,1.546590) (1.650000,1.587017) (1.700000,1.627868) (1.750000,1.669115) (1.800000,1.710735)
(1.850000,1.752705) (1.900000,1.795005) (1.950000,1.837614) (2.000000,1.880516) (2.050000,1.923693)
(2.100000,1.967129) (2.150000,2.010812) (2.200000,2.054726) (2.250000,2.098861) (2.300000,2.143204)
(2.350000,2.187744) (2.400000,2.232471) (2.450000,2.277376) (2.500000,2.322451) (2.550000,2.367686)
(2.600000,2.413074) (2.650000,2.458609) (2.700000,2.504282) (2.750000,2.550088) (2.800000,2.596020)
(2.850000,2.642074) (2.900000,2.688243) (2.950000,2.734523) (3.000000,2.780908) (3.050000,2.827395)
(3.100000,2.873980) (3.150000,2.920657) (3.200000,2.967424) (3.250000,3.014277) (3.300000,3.061212)
(3.350000,3.108227) (3.400000,3.155318) (3.450000,3.202482) (3.500000,3.249717) (3.550000,3.297019)
(3.600000,3.344387) (3.650000,3.391819) (3.700000,3.439311) (3.750000,3.486861) (3.800000,3.534469)
(3.850000,3.582131) (3.900000,3.629846) (3.950000,3.677612) (4.000000,3.725427) (4.050000,3.773290)
(4.100000,3.821200) (4.150000,3.869154) (4.200000,3.917151) (4.250000,3.965190) (4.300000,4.013270)
(4.350000,4.061389) (4.400000,4.109546) (4.450000,4.157740) (4.500000,4.205970) (4.550000,4.254235)
(4.600000,4.302534) (4.650000,4.350866) (4.700000,4.399229) (4.750000,4.447623) (4.800000,4.496048)
(4.850000,4.544501) (4.900000,4.592983) (4.950000,4.641493) (5.000000,4.690030) (5.050000,4.738593)
(5.100000,4.787181) (5.150000,4.835794) (5.200000,4.884431) (5.250000,4.933092) (5.300000,4.981775)
(5.350000,5.030481) (5.400000,5.079209) (5.450000,5.127958) (5.500000,5.176728) (5.550000,5.225518)
(5.600000,5.274327) (5.650000,5.323156) (5.700000,5.372004) (5.750000,5.420869) (5.800000,5.469753)
(5.850000,5.518654) (5.900000,5.567572) (5.950000,5.616507) (6.000000,5.665458)}
\def\sliceabove{
(0.000000,0.763576) (0.050000,0.764456) (0.100000,0.767100) (0.150000,0.771518) (0.200000,0.777723)
(0.250000,0.785732) (0.300000,0.795555) (0.350000,0.807196) (0.400000,0.820652) (0.450000,0.835904)
(0.500000,0.852925) (0.550000,0.871673) (0.600000,0.892096) (0.650000,0.914135) (0.700000,0.937722)
(0.750000,0.962784) (0.800000,0.989247) (0.850000,1.017032) (0.900000,1.046063) (0.950000,1.076265)
(1.000000,1.107564) (1.050000,1.139889) (1.100000,1.173174) (1.150000,1.207354) (1.200000,1.242370)
(1.250000,1.278165) (1.300000,1.314688) (1.350000,1.351888) (1.357597,1.357597)}
\begin{scope}
\clip (0,0) rectangle (6,6.5);
\draw[connectorgray,line width=.6pt,dashed] (0,0)--(6,6);
\draw[connectorgray,line width=.7pt,densely dotted] plot coordinates {
(0.000000,0.866025) (0.050000,0.867107) (0.100000,0.870345) (0.150000,0.875714) (0.200000,0.883176)
(0.250000,0.892679) (0.300000,0.904157) (0.350000,0.917537) (0.400000,0.932738) (0.450000,0.949671)
(0.500000,0.968246) (0.550000,0.988370) (0.600000,1.009950) (0.650000,1.032896) (0.700000,1.057119)
(0.750000,1.082532) (0.800000,1.109054) (0.850000,1.136607) (0.900000,1.165118) (0.950000,1.194519)
(1.000000,1.224745) (1.050000,1.255737) (1.100000,1.287439) (1.150000,1.319801) (1.200000,1.352775)
(1.250000,1.386317) (1.300000,1.420387) (1.350000,1.454948) (1.400000,1.489966) (1.450000,1.525410)
(1.500000,1.561249) (1.550000,1.597459) (1.600000,1.634013) (1.650000,1.670890) (1.700000,1.708069)
(1.750000,1.745530) (1.800000,1.783255) (1.850000,1.821229) (1.900000,1.859435) (1.950000,1.897861)
(2.000000,1.936492) (2.050000,1.975316) (2.100000,2.014324) (2.150000,2.053503) (2.200000,2.092845)
(2.250000,2.132340) (2.300000,2.171981) (2.350000,2.211758) (2.400000,2.251666) (2.450000,2.291697)
(2.500000,2.331845) (2.550000,2.372103) (2.600000,2.412468) (2.650000,2.452932) (2.700000,2.493492)
(2.750000,2.534142) (2.800000,2.574879) (2.850000,2.615698) (2.900000,2.656596) (2.950000,2.697568)
(3.000000,2.738613) (3.050000,2.779726) (3.100000,2.820904) (3.150000,2.862145) (3.200000,2.903446)
(3.250000,2.944805) (3.300000,2.986218) (3.350000,3.027685) (3.400000,3.069202) (3.450000,3.110768)
(3.500000,3.152380) (3.550000,3.194037) (3.600000,3.235738) (3.650000,3.277480) (3.700000,3.319262)
(3.750000,3.361082) (3.800000,3.402940) (3.850000,3.444833) (3.900000,3.486761) (3.950000,3.528721)
(4.000000,3.570714) (4.050000,3.612738) (4.100000,3.654791) (4.150000,3.696874) (4.200000,3.738984)
(4.250000,3.781121) (4.300000,3.823284) (4.350000,3.865472) (4.400000,3.907685) (4.450000,3.949921)
(4.500000,3.992180) (4.550000,4.034461) (4.600000,4.076763) (4.650000,4.119087) (4.700000,4.161430)
(4.750000,4.203793) (4.800000,4.246175) (4.850000,4.288575) (4.900000,4.330993) (4.950000,4.373428)
(5.000000,4.415880) (5.050000,4.458349) (5.100000,4.500833) (5.150000,4.543333) (5.200000,4.585848)
(5.250000,4.628377) (5.300000,4.670921) (5.350000,4.713478) (5.400000,4.756049) (5.450000,4.798633)
(5.500000,4.841229) (5.550000,4.883838) (5.600000,4.926459) (5.650000,4.969092) (5.700000,5.011736)
(5.750000,5.054392) (5.800000,5.097058) (5.850000,5.139735) (5.900000,5.182422) (5.950000,5.225120)
(6.000000,5.267827)};
\fill[connectorblue,opacity=.085] plot coordinates {\sliceabove} --(0,0)--cycle;
\fill[connectororange,opacity=.085] plot coordinates {\slicebelow} --(6,6)--cycle;
\draw[connectorblue,line width=1pt] plot coordinates {\sliceabove};
\draw[connectororange,line width=1pt] plot coordinates {\slicebelow};
\end{scope}
\draw[connectorblue,fill=white,line width=.65pt] (0,0.763576) circle[radius=1.8pt];
\fill[connectorink] (1.357597,1.357597) circle[radius=1.45pt];
\end{scope}
\begin{scope}[shift={(3.5,4.0)}]
\connectorsliceaxes{2}
\def\slicebelow{
(2.606254,2.606254) (2.650000,2.641755) (2.700000,2.682626) (2.750000,2.723798) (2.800000,2.765260)
(2.850000,2.806999) (2.900000,2.849003) (2.950000,2.891262) (3.000000,2.933765) (3.050000,2.976501)
(3.100000,3.019462) (3.150000,3.062638) (3.200000,3.106020) (3.250000,3.149601) (3.300000,3.193372)
(3.350000,3.237326) (3.400000,3.281456) (3.450000,3.325755) (3.500000,3.370215) (3.550000,3.414832)
(3.600000,3.459599) (3.650000,3.504511) (3.700000,3.549562) (3.750000,3.594746) (3.800000,3.640060)
(3.850000,3.685498) (3.900000,3.731055) (3.950000,3.776728) (4.000000,3.822513) (4.050000,3.868405)
(4.100000,3.914400) (4.150000,3.960496) (4.200000,4.006688) (4.250000,4.052973) (4.300000,4.099349)
(4.350000,4.145812) (4.400000,4.192358) (4.450000,4.238987) (4.500000,4.285694) (4.550000,4.332477)
(4.600000,4.379334) (4.650000,4.426262) (4.700000,4.473259) (4.750000,4.520323) (4.800000,4.567452)
(4.850000,4.614644) (4.900000,4.661897) (4.950000,4.709209) (5.000000,4.756578) (5.050000,4.804002)
(5.100000,4.851481) (5.150000,4.899011) (5.200000,4.946593) (5.250000,4.994224) (5.300000,5.041902)
(5.350000,5.089628) (5.400000,5.137398) (5.450000,5.185212) (5.500000,5.233069) (5.550000,5.280968)
(5.600000,5.328906) (5.650000,5.376884) (5.700000,5.424901) (5.750000,5.472954) (5.800000,5.521044)
(5.850000,5.569168) (5.900000,5.617328) (5.950000,5.665520) (6.000000,5.713745)}
\def\sliceabove{
(0.000000,1.402195) (0.050000,1.402617) (0.100000,1.403885) (0.150000,1.406003) (0.200000,1.408980)
(0.250000,1.412828) (0.300000,1.417561) (0.350000,1.423193) (0.400000,1.429743) (0.450000,1.437229)
(0.500000,1.445667) (0.550000,1.455077) (0.600000,1.465473) (0.650000,1.476870) (0.700000,1.489282)
(0.750000,1.502717) (0.800000,1.517183) (0.850000,1.532683) (0.900000,1.549219) (0.950000,1.566788)
(1.000000,1.585385) (1.050000,1.605000) (1.100000,1.625622) (1.150000,1.647237) (1.200000,1.669828)
(1.250000,1.693376) (1.300000,1.717861) (1.350000,1.743259) (1.400000,1.769547) (1.450000,1.796701)
(1.500000,1.824693) (1.550000,1.853498) (1.600000,1.883089) (1.650000,1.913439) (1.700000,1.944521)
(1.750000,1.976308) (1.800000,2.008774) (1.850000,2.041893) (1.900000,2.075638) (1.950000,2.109986)
(2.000000,2.144911) (2.050000,2.180390) (2.100000,2.216399) (2.150000,2.252918) (2.200000,2.289924)
(2.250000,2.327396) (2.300000,2.365315) (2.350000,2.403662) (2.400000,2.442418) (2.450000,2.481566)
(2.500000,2.521090) (2.550000,2.560973) (2.600000,2.601199) (2.606254,2.606254)}
\begin{scope}
\clip (0,0) rectangle (6,6.5);
\draw[connectorgray,line width=.6pt,dashed] (0,0)--(6,6);
\draw[connectorgray,line width=.7pt,densely dotted] plot coordinates {
(0.000000,1.732051) (0.050000,1.732592) (0.100000,1.734215) (0.150000,1.736915) (0.200000,1.740690)
(0.250000,1.745530) (0.300000,1.751428) (0.350000,1.758373) (0.400000,1.766352) (0.450000,1.775352)
(0.500000,1.785357) (0.550000,1.796350) (0.600000,1.808314) (0.650000,1.821229) (0.700000,1.835075)
(0.750000,1.849831) (0.800000,1.865476) (0.850000,1.881987) (0.900000,1.899342) (0.950000,1.917518)
(1.000000,1.936492) (1.050000,1.956240) (1.100000,1.976740) (1.150000,1.997968) (1.200000,2.019901)
(1.250000,2.042517) (1.300000,2.065793) (1.350000,2.089707) (1.400000,2.114237) (1.450000,2.139363)
(1.500000,2.165064) (1.550000,2.191318) (1.600000,2.218107) (1.650000,2.245412) (1.700000,2.273214)
(1.750000,2.301494) (1.800000,2.330236) (1.850000,2.359423) (1.900000,2.389037) (1.950000,2.419065)
(2.000000,2.449490) (2.050000,2.480297) (2.100000,2.511474) (2.150000,2.543005) (2.200000,2.574879)
(2.250000,2.607082) (2.300000,2.639602) (2.350000,2.672429) (2.400000,2.705550) (2.450000,2.738955)
(2.500000,2.772634) (2.550000,2.806577) (2.600000,2.840775) (2.650000,2.875217) (2.700000,2.909897)
(2.750000,2.944805) (2.800000,2.979933) (2.850000,3.015274) (2.900000,3.050820) (2.950000,3.086564)
(3.000000,3.122499) (3.050000,3.158619) (3.100000,3.194918) (3.150000,3.231389) (3.200000,3.268027)
(3.250000,3.304826) (3.300000,3.341781) (3.350000,3.378887) (3.400000,3.416138) (3.450000,3.453531)
(3.500000,3.491060) (3.550000,3.528721) (3.600000,3.566511) (3.650000,3.604424) (3.700000,3.642458)
(3.750000,3.680608) (3.800000,3.718871) (3.850000,3.757243) (3.900000,3.795721) (3.950000,3.834302)
(4.000000,3.872983) (4.050000,3.911761) (4.100000,3.950633) (4.150000,3.989596) (4.200000,4.028647)
(4.250000,4.067785) (4.300000,4.107006) (4.350000,4.146309) (4.400000,4.185690) (4.450000,4.225148)
(4.500000,4.264681) (4.550000,4.304286) (4.600000,4.343961) (4.650000,4.383706) (4.700000,4.423517)
(4.750000,4.463393) (4.800000,4.503332) (4.850000,4.543333) (4.900000,4.583394) (4.950000,4.623513)
(5.000000,4.663690) (5.050000,4.703921) (5.100000,4.744207) (5.150000,4.784545) (5.200000,4.824935)
(5.250000,4.865375) (5.300000,4.905864) (5.350000,4.946400) (5.400000,4.986983) (5.450000,5.027611)
(5.500000,5.068284) (5.550000,5.108999) (5.600000,5.149757) (5.650000,5.190556) (5.700000,5.231396)
(5.750000,5.272274) (5.800000,5.313191) (5.850000,5.354146) (5.900000,5.395137) (5.950000,5.436164)
(6.000000,5.477226)};
\fill[connectorblue,opacity=.085] plot coordinates {\sliceabove} --(0,0)--cycle;
\fill[connectororange,opacity=.085] plot coordinates {\slicebelow} --(6,6)--cycle;
\draw[connectorblue,line width=1pt] plot coordinates {\sliceabove};
\draw[connectororange,line width=1pt] plot coordinates {\slicebelow};
\end{scope}
\draw[connectorblue,fill=white,line width=.65pt] (0,1.402195) circle[radius=1.8pt];
\fill[connectorink] (2.606254,2.606254) circle[radius=1.45pt];
\end{scope}
\begin{scope}[shift={(16.5,4.0)}]
\connectorsliceaxes{4}
\def\slicebelow{
(5.071911,5.071911) (5.100000,5.096403) (5.150000,5.140106) (5.200000,5.183942) (5.250000,5.227908)
(5.300000,5.272000) (5.350000,5.316216) (5.400000,5.360551) (5.450000,5.405003) (5.500000,5.449568)
(5.550000,5.494243) (5.600000,5.539027) (5.650000,5.583914) (5.700000,5.628904) (5.750000,5.673994)
(5.800000,5.719180) (5.850000,5.764461) (5.900000,5.809834) (5.950000,5.855296) (6.000000,5.900846)}
\def\sliceabove{
(0.000000,2.490883) (0.050000,2.491095) (0.100000,2.491734) (0.150000,2.492799) (0.200000,2.494295)
(0.250000,2.496223) (0.300000,2.498589) (0.350000,2.501399) (0.400000,2.504657) (0.450000,2.508371)
(0.500000,2.512549) (0.550000,2.517199) (0.600000,2.522330) (0.650000,2.527951) (0.700000,2.534072)
(0.750000,2.540702) (0.800000,2.547852) (0.850000,2.555532) (0.900000,2.563751) (0.950000,2.572521)
(1.000000,2.581851) (1.050000,2.591750) (1.100000,2.602228) (1.150000,2.613293) (1.200000,2.624954)
(1.250000,2.637218) (1.300000,2.650092) (1.350000,2.663582) (1.400000,2.677695) (1.450000,2.692433)
(1.500000,2.707803) (1.550000,2.723806) (1.600000,2.740444) (1.650000,2.757719) (1.700000,2.775632)
(1.750000,2.794182) (1.800000,2.813368) (1.850000,2.833188) (1.900000,2.853638) (1.950000,2.874716)
(2.000000,2.896417) (2.050000,2.918736) (2.100000,2.941667) (2.150000,2.965204) (2.200000,2.989340)
(2.250000,3.014068) (2.300000,3.039380) (2.350000,3.065267) (2.400000,3.091720) (2.450000,3.118731)
(2.500000,3.146290) (2.550000,3.174388) (2.600000,3.203014) (2.650000,3.232158) (2.700000,3.261811)
(2.750000,3.291961) (2.800000,3.322598) (2.850000,3.353712) (2.900000,3.385292) (2.950000,3.417327)
(3.000000,3.449808) (3.050000,3.482723) (3.100000,3.516062) (3.150000,3.549815) (3.200000,3.583971)
(3.250000,3.618520) (3.300000,3.653453) (3.350000,3.688759) (3.400000,3.724429) (3.450000,3.760453)
(3.500000,3.796822) (3.550000,3.833526) (3.600000,3.870557) (3.650000,3.907906) (3.700000,3.945564)
(3.750000,3.983523) (3.800000,4.021774) (3.850000,4.060310) (3.900000,4.099122) (3.950000,4.138203)
(4.000000,4.177545) (4.050000,4.217142) (4.100000,4.256985) (4.150000,4.297069) (4.200000,4.337386)
(4.250000,4.377930) (4.300000,4.418695) (4.350000,4.459674) (4.400000,4.500861) (4.450000,4.542251)
(4.500000,4.583838) (4.550000,4.625617) (4.600000,4.667582) (4.650000,4.709727) (4.700000,4.752049)
(4.750000,4.794542) (4.800000,4.837201) (4.850000,4.880022) (4.900000,4.923000) (4.950000,4.966131)
(5.000000,5.009412) (5.050000,5.052837) (5.071911,5.071911)}
\begin{scope}
\clip (0,0) rectangle (6,6.5);
\draw[connectorgray,line width=.6pt,dashed] (0,0)--(6,6);
\draw[connectorgray,line width=.7pt,densely dotted] plot coordinates {
(0.000000,3.464102) (0.050000,3.464372) (0.100000,3.465184) (0.150000,3.466536) (0.200000,3.468429)
(0.250000,3.470861) (0.300000,3.473831) (0.350000,3.477337) (0.400000,3.481379) (0.450000,3.485954)
(0.500000,3.491060) (0.550000,3.496695) (0.600000,3.502856) (0.650000,3.509541) (0.700000,3.516746)
(0.750000,3.524468) (0.800000,3.532704) (0.850000,3.541451) (0.900000,3.550704) (0.950000,3.560460)
(1.000000,3.570714) (1.050000,3.581463) (1.100000,3.592701) (1.150000,3.604424) (1.200000,3.616628)
(1.250000,3.629308) (1.300000,3.642458) (1.350000,3.656074) (1.400000,3.670150) (1.450000,3.684681)
(1.500000,3.699662) (1.550000,3.715087) (1.600000,3.730952) (1.650000,3.747249) (1.700000,3.763974)
(1.750000,3.781121) (1.800000,3.798684) (1.850000,3.816658) (1.900000,3.835036) (1.950000,3.853813)
(2.000000,3.872983) (2.050000,3.892541) (2.100000,3.912480) (2.150000,3.932795) (2.200000,3.953479)
(2.250000,3.974528) (2.300000,3.995935) (2.350000,4.017695) (2.400000,4.039802) (2.450000,4.062250)
(2.500000,4.085034) (2.550000,4.108147) (2.600000,4.131586) (2.650000,4.155343) (2.700000,4.179414)
(2.750000,4.203793) (2.800000,4.228475) (2.850000,4.253454) (2.900000,4.278726) (2.950000,4.304286)
(3.000000,4.330127) (3.050000,4.356246) (3.100000,4.382636) (3.150000,4.409294) (3.200000,4.436215)
(3.250000,4.463393) (3.300000,4.490824) (3.350000,4.518504) (3.400000,4.546427) (3.450000,4.574590)
(3.500000,4.602988) (3.550000,4.631617) (3.600000,4.660472) (3.650000,4.689550) (3.700000,4.718845)
(3.750000,4.748355) (3.800000,4.778075) (3.850000,4.808001) (3.900000,4.838130) (3.950000,4.868457)
(4.000000,4.898979) (4.050000,4.929693) (4.100000,4.960595) (4.150000,4.991681) (4.200000,5.022947)
(4.250000,5.054392) (4.300000,5.086010) (4.350000,5.117800) (4.400000,5.149757) (4.450000,5.181879)
(4.500000,5.214163) (4.550000,5.246606) (4.600000,5.279204) (4.650000,5.311956) (4.700000,5.344857)
(4.750000,5.377906) (4.800000,5.411100) (4.850000,5.444435) (4.900000,5.477910) (4.950000,5.511522)
(5.000000,5.545268) (5.050000,5.579146) (5.100000,5.613154) (5.150000,5.647289) (5.200000,5.681549)
(5.250000,5.715932) (5.300000,5.750435) (5.350000,5.785056) (5.400000,5.819794) (5.450000,5.854646)
(5.500000,5.889609) (5.550000,5.924684) (5.600000,5.959866) (5.650000,5.995154) (5.700000,6.030547)
(5.750000,6.066043) (5.800000,6.101639) (5.850000,6.137335) (5.900000,6.173127) (5.950000,6.209016)
(6.000000,6.244998)};
\fill[connectorblue,opacity=.085] plot coordinates {\sliceabove} --(0,0)--cycle;
\fill[connectororange,opacity=.085] plot coordinates {\slicebelow} --(6,6)--cycle;
\draw[connectorblue,line width=1pt] plot coordinates {\sliceabove};
\draw[connectororange,line width=1pt] plot coordinates {\slicebelow};
\end{scope}
\draw[connectorblue,fill=white,line width=.65pt] (0,2.490883) circle[radius=1.8pt];
\fill[connectorink] (5.071911,5.071911) circle[radius=1.45pt];
\end{scope}

\draw[connectorblue,line width=1pt] (5,1.8)--(6.2,1.8);
\node[anchor=west] at (6.4,1.8) {Profile: $|z|>|x|$};
\draw[connectororange,line width=1pt] (5,.8)--(6.2,.8);
\node[anchor=west] at (6.4,.8) {Profile: $|z|<|x|$};
\draw[connectorgray,line width=.6pt,dashed] (14,1.8)--(15.2,1.8);
\node[anchor=west] at (15.4,1.8) {Cylinder $C_{3,3}\times\R$};
\draw[connectorgray,line width=.7pt,densely dotted] (14,.8)--(15.2,.8);
\node[anchor=west] at (15.4,.8) {Core cone $C_{4,3}$};
\end{tikzpicture}%
}
\endgroup
\caption{Transverse slices at $y=0,\frac14,\frac12,1,2,4$, read from
left to right and top to bottom, with common scales. The dashed and
dotted curves are the corresponding sections of $V=C_{3,3}\times\R$
and $C_{4,3}$. Blue and orange distinguish the regions $|z|>|x|$ and
$|z|<|x|$. Open circles mark the smooth $S^3$ orbits at $x=0$;
filled circles mark crossings with $V$. The origin in the $y=0$
slice is singular. Negative $y$ gives the same slices by symmetry.
The colored curves use an illustrative algebraic profile, not a
numerical solution of the minimal surface equation.}
\label{fig:connector-slice-evolution}
\end{figure}
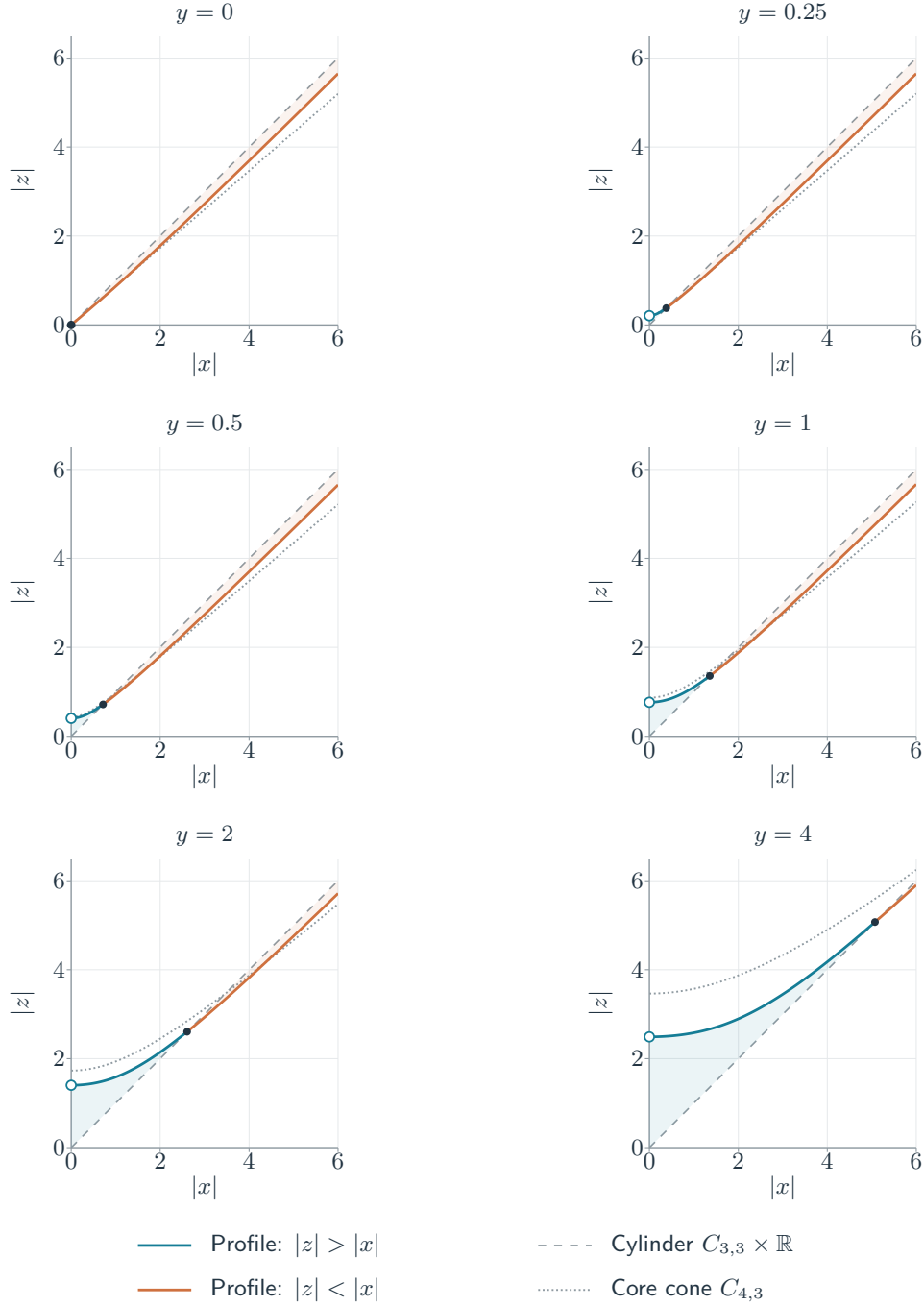

\subsection{Relation to earlier constructions and rigidity results}
Hardt--Simon \cite{HS} associate smooth area-minimizing leaves to an
area-minimizing hypercone with isolated singularity; their dilations
foliate each complementary component. These leaves describe how a
conical singularity can be smoothed on either side. Caffarelli--Hardt--Simon
\cite{CHS} constructed local nonconical minimal hypersurfaces with
prescribed isolated conical singularities. Smale \cite{Smale1989}
subsequently constructed minimal hypersurfaces with many isolated
conical singularities. Simon--Solomon \cite{SS}
proved a global rigidity theorem for minimizing hypersurfaces
asymptotic to suitable quadratic cones. Together with the local
regularity theorem of Edelen--Spolaor \cite{ESp} discussed above,
these results concern
cones with smooth links. The asymptotic cone in
Theorem~\ref{thm:main-connector} instead, has a singular link, and the
connector meets both sides of that cone.

The cylindrical analysis requires control near the axis without
assuming that the singular set fills it. Sz\'ekelyhidi's barriers
and nonconcentration estimates \cite{Sz} address this difficulty
and are essential ingredients here. The connector has a different
geometry from the local examples in \cite{SzCyl}: its tangent cone
at the singular point is $C_{4,3}$, whereas its cone at infinity is $V$.

One-sidedness imposes a further restriction on entire hypersurfaces.
Simon \cite{SimLiouville} proved a cylindrical Liouville theorem
under strict stability and a finite-density hypothesis.
Edelen--Sz\'ekelyhidi \cite{ES} obtained the corresponding rigidity
for one-sided hypersurfaces under a density bound, with an extension
to stationary integral varifolds. The calculation above verifies
directly why that theorem does not exclude the present connector:
the density bound is satisfied, but the tangent cone forces the
surface to cross $V$ (See the orange and blue sides in Figure~\ref{fig:connector-slice-evolution}).

The prescribed-singular-set construction of Simon
\cite{SimPrescribed} already allows accumulation in smooth metrics.
In a recent work, Liu \cite{Liu2026} proves area-minimizing
properties for these examples; his theorem covers the Simons-cone cross-section
when the construction parameters are sufficiently small. The
new feature in Theorem~\ref{thm:main-accumulation} is the change of
tangent-cone type: the isolated singular points have the tangent cone
$C_{4,3}$, which has an isolated singularity, while their
accumulation point has the cylindrical tangent cone $V$. The metric is conformally Euclidean,
is exactly Euclidean near each inserted singularity, and has the
full Euclidean jet at the accumulation point, that is the origin. Since $g$ is not Euclidean on any neighborhood of the origin, the last property implies
that $g$ is not real analytic near the origin. In fact we conjecture that such an example cannot be constructed with an analytic metric (in particular the euclidean metric).
For related strong unique continuation results in the Euclidean
setting and a discussion of real-analytic ambient metrics, see
\cite[Theorem~1.2 and the subsequent discussion]{SzCyl}.

\subsection{The construction: main ingredients and use of AI}
In this subsection we describe the main ideas of the proof, what AI contributed and where they come from. More
detailed explanations of each main step are given in the corresponding
sections. We first construct the entire Euclidean connector, then
establish the estimates needed to insert copies of it at infinitely
many scales and obtain a smooth ambient metric.

\smallskip\noindent\emph{Constructing the connector.}
In Section~\ref{sec:finite-domain-data} we begin with Dirichlet problems on a fixed rounded cylinder
$\Omega\subset\R^4_x\times\R_y$ (see Figure~\ref{fig:finite-domain}). On this domain we will use Leray--Schauder degree argument from Fouladgar--Simon \cite{FS}, so that given $a>0$ we can find  boundary data $g_{\eps,b_{\eps,a}}$ and solutions $u_{\eps,a}$ of the symmetric minimal surface equation (SME) such that the minimum of $u_{\eps,a}$ is $a$ for every $\eps>0$. The parameter $b_{\eps,a}$ is chosen to guarantee such condition by representing a dip in the boundary data.  Keeping $\eps$ fixed, we let $a\downarrow0$ and pass to a subsequence;
this gives a singular solution and a limiting parameter $b_\eps$.
Only afterward do we let $\eps\downarrow0$, obtaining convergence
to $V=C_{3,3}\times\R$. The analysis also determines the limiting
choice of $b$. Using leaf comparison and moving planes, we also show axial monotonicity of solutions throughout the construction. The boundary data are chosen to give more than convergence to $V$:
the leading normalized perturbation near the origin is a positive
multiple of the degree-zero Jacobi field
\[
 J_2(r,y)=\frac{y^2}{r^2}-\frac13,
 \qquad r=(|x|^2+|z|^2)^{1/2}.
\]
Section~\ref{sec:classification} uses this field to obtain a strict
density drop below $\Theta(V)$. A separate density comparison excludes
zeros away from the axis. The strict drop then excludes an interval
of axial zeros, and the cone classification identifies the unique
remaining singularity as $C_{4,3}$. Thus each member of the family
has the desired local singularity, although the family converges
at fixed scale to the cylindrical cone.

Section~\ref{sec:connector} extracts an entire surface from these
two facts. We choose a radius at which the density ratio has a fixed
value strictly between $\Theta(C_{4,3})$ and $\Theta(V)$, and rescale
that radius to one. The chosen radii tend to zero, so the rescaled
domains exhaust the whole space. The density normalization prevents
the limit from being either cone. Compactness and cone classification
then give the connector with tangent cone $C_{4,3}$ at zero and
unique blow-down $V$.

\smallskip\noindent\emph{Quantitative control of the end.}
For gluing, convergence to $V$ away from its spine is not sufficient:
we need uniform control of higher order derivatives through the smoothing regions. Section~\ref{sec:end}
first proves quadratic growth of the signed Hardt--Simon phase and
positivity of its leading coefficient. Sz\'ekelyhidi's barriers and
nonconcentration estimates \cite{Sz} transfer this information to
the spine and determine the smoothing radius
$a_\infty(Y)=u_\infty(0,Y)\sim\kappa_H|Y|^{2/3}$.
We then differentiate the equation for $u_\infty$ directly. Normalizing
by the positive dilation Jacobi field of a comparison leaf leads
to a uniformly controlled Muckenhoupt $A_2$ equation. The weighted
estimates of \cite{FKS,FJK} give bounds for all axial derivatives,
including those of $a_\infty$. These bounds control the forcing in
the equation for the difference between the connector and the leaf
with the same smoothing radius. Appendix~\ref{app:connector-derivatives} supplies
the detailed derivative and remainder estimates.

\smallskip\noindent\emph{Gluing the singularities.}
Section~\ref{sec:accumulation} starts from $V$, keeping the Simons-cone
slice at $y=0$. At other heights we replace a small neighborhood of
the spine by a positive Hardt--Simon leaf and cut back to the cone
outside a narrow tube around the spine. Into this background we insert translated and
dilated connectors in disjoint balls whose centers approach zero.
The background radius function $a(y)$ is chosen to agree with the
smoothing radius of the corresponding dilated connector on the
axial intervals belonging to the near-axis portions of each
attachment annulus.

Matching the axis values alone does not suffice. Direct interpolation
leaves a mean-curvature error that tends to zero but varies across
the much smaller smoothing scale; the resulting conformal correction
cannot have bounded second derivatives. We therefore improve the
background before interpolating. Start with
$V_0[a](t,y)=a(y)h(t/a(y))$, where $t=|x|$.
If $a$ were constant, this would be an exact minimal hypersurface.
When $a$ varies with $y$, each transverse slice is still a foliation
leaf, but the full hypersurface has a mean-curvature error: its
leading terms involve $a''$ and $(a')^2/a$.

To cancel these terms, fix $y$, linearize the radial equation about
the leaf, and solve this ordinary differential equation with the
negative of the leading error as forcing. The correction has zero
value and zero radial derivative at $t=0$, so it preserves the
smoothing radius $a(y)$ and smoothness at the axis. Its dependence
on $y$ enters through the factors $a(a')^2$ and $a^2a''$; the
coefficients of these factors are fixed functions of $t/a(y)$,
determined by that radial equation.

Adding this first correction gives $V_1[a]$. We then compute its
remaining error, cancel the next terms by another radial solve,
and continue to a finite approximation $V_J[a]$; see
Figure~\ref{fig:radial-correction-scheme}. Section~\ref{sec:accumulation}
proves two estimates for this procedure: the corrected background
has a smaller mean-curvature residual, and, wherever $a(y)$ agrees
with the connector's axis value, it approximates the connector to
the corresponding higher order. Thus both errors introduced by
interpolation are controlled. At each stage we use a finite sum
with a remainder estimate; convergence as $J\to\infty$ at a fixed
point is not required.

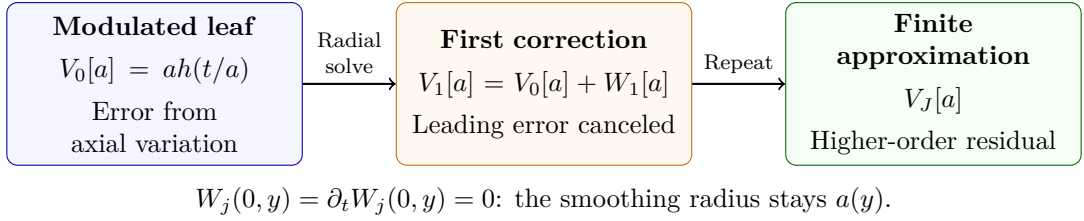
\begin{figure}[htbp]
\centering
\begin{tikzpicture}[font=\small]
  \node[draw=blue!65!black,fill=blue!4,rounded corners,
        text width=3.65cm,minimum height=2.15cm,align=center]
    (leaf) at (0,0)
    {\textbf{Modulated leaf}\\[4pt]
     $V_0[a]=a h(t/a)$\\[4pt]
     Error from\\axial variation};
  \node[draw=orange!80!black,fill=orange!5,rounded corners,
        text width=3.65cm,minimum height=2.15cm,align=center]
    (first) at (5.15,0)
    {\textbf{First correction}\\[4pt]
     $V_1[a]=V_0[a]+W_1[a]$\\[4pt]
     Leading error canceled};
  \node[draw=green!40!black,fill=green!4,rounded corners,
        text width=3.65cm,minimum height=2.15cm,align=center]
    (finite) at (10.3,0)
    {\textbf{Finite}\\\textbf{approximation}\\[4pt]
     $V_J[a]$\\[4pt]
     Higher-order residual};
  \draw[->,thick] (leaf.east)--(first.west)
    node[midway,above,font=\scriptsize,align=center] {Radial\\solve};
  \draw[->,thick] (first.east)--(finite.west)
    node[midway,above,font=\scriptsize] {Repeat};
  \node at (5.15,-1.55)
    {$W_j(0,y)=\partial_tW_j(0,y)=0$: the smoothing radius stays $a(y)$.};
\end{tikzpicture}
\caption{Successive radial corrections. At each step, solve the
linearized radial equation to cancel the leading remaining error,
then estimate the new residual. When the axis data agree with the
connector's, a separate matching estimate controls the difference
between $V_J[a]$ and the connector.}
\label{fig:radial-correction-scheme}
\end{figure}

For each fixed correction order, a finite-accuracy gluing lemma
controls the remaining curvature and the derivatives of a conformal
factor $F$ satisfying $H+8\nu F=0$ on the glued surface. We first
fix the insertion centers and scales and construct the matching
radius function. We then choose successively smaller radii at which
the correction order increases, placing the transitions away from
the connector cores. This diagonal choice makes every fixed
derivative of $F$ vanish to every prescribed order at zero.
The resulting metric is smooth and flat relative to the Euclidean
metric there, and is exactly Euclidean near each inserted
singularity. Finally, mass estimates show that the modifications
preserve $V$ as the unique tangent cone at the accumulation point
and allow stationarity to extend across it.

\medskip

\emph{Use of LLMs.} The higher order estimates of Section \ref{sec:connector-estimates}, in particular the details in Appendix \ref{app:connector-derivatives} were produced with the help of ChatGPT-6 ASTRA, in particular the idea of using $A_2$ weights is entirely LLM generated, and it was surprising to us. We also used ASTRA to complete the details of the gluing process in Appendix \ref{sec:corrector-estimates}. Finally we used ASTRA for editing the final version of this manuscript. We take full responsibility for the correctness of the manuscript.  

\subsection{Forthcoming work and questions} 
We conclude the introduction with some natural open questions, which we
expect to follow from an extension of the present construction.

\begin{conjecture}[Connectors for quadratic cones]
\label{conj:connectors}
Let $p,q\ge1$, and suppose that $C_{p,q}$ is area-minimizing.
There exists a complete, connected, embedded hypersurface
\[
 M\subset\R^{p+1}_x\times\R^{q+1}_z\times\R_y
\]
whose associated multiplicity-one varifold is stationary,
such that
\[
 \operatorname{Sing}M=\{0\},\qquad
 \lim_{\rho\downarrow0}\rho^{-1}M=C_{p+1,q},\qquad
 \lim_{R\to\infty}R^{-1}M=C_{p,q}\times\R.
\]
Both limits are unique, have multiplicity one, and are smooth
locally away from the singular sets of the limiting cones.
Moreover, $M$ is invariant under
$O(p+1)\times O(q+1)$ and under $y\mapsto-y$.
\end{conjecture}

Here the tangent cone is understood in the form
\[
 C_{p+1,q}
 =\{q(|x|^2+y^2)=(p+1)|z|^2\}.
\]
Applying the same conjecture with $p$ and $q$ interchanged
gives the other attachment, whose tangent cone is
\[
 C_{p,q+1}
 =\{(q+1)|x|^2=p(|z|^2+y^2)\}.
\]
The conjecture concerns stationary hypersurfaces. 
\begin{conjecture}[Accumulation of quadratic singularities]
\label{conj:quadratic-accumulation}
Let $p,q\ge1$, and suppose that $C_{p,q}$ is area-minimizing.
There exist an open neighborhood $\mathcal U$ of zero in
\[
 \R^{p+1}_x\times\R^{q+1}_z\times\R_y,
\]
a smooth conformally Euclidean metric $g$ on $\mathcal U$,
and an embedded hypersurface $\Sigma\subset\mathcal U$
whose associated multiplicity-one integral varifold is
stationary with respect to $g$, such that, for some $k_0$,
\[
 \operatorname{Sing}\Sigma
 =\{0\}\cup\{p_k:k\ge k_0\},
 \qquad p_k=(0,0,2^{-k}).
\]
The tangent cones are unique and satisfy
\[
 \operatorname{Tan}(\Sigma,0)
 =\{C_{p,q}\times\R\},
 \qquad
 \operatorname{Tan}(\Sigma,p_k)
 =\{C_{p+1,q}\},
\]
all with multiplicity one.

The metric is Euclidean near each $p_k$ and agrees with
the Euclidean metric to infinite order at zero.
More precisely, for every pair of nonnegative integers
$m,N$, there is a constant $C_{m,N}$ such that
\[
 \|g-g_{\rm euc}\|_{C^m(B_\rho(0))}
 \le C_{m,N}\rho^N
\]
for all sufficiently small $\rho>0$.
\end{conjecture}

Interchanging the two factors gives the analogous conjecture
with tangent cone $C_{p,q+1}$ at every $p_k$. Analogous questions can be stated replacing $\R$ with $\R^k$, but, a part for the obvious product constructions, new techniques might be required to handle the general case $k\geq 2$.

A roadmap towards the solution of this conjecture using the results of this paper is given in Appendix \ref{sec:general-cones}.

\subsection{Organization}
Section~\ref{sec:connector} reduces
Theorem~\ref{thm:main-connector} to the existence of a family of
singular solutions of Dirichlet problems on a bounded domain and to
the classification of their singularities.
Sections~\ref{sec:finite-domain-data} and~\ref{sec:classification}
address these two ingredients. Section~\ref{sec:end} studies the
connector at infinity, and Section~\ref{sec:accumulation} carries
out the matching and conformal correction, thereby proving
Theorem~\ref{thm:main-accumulation}.
Appendices~\ref{app:connector-derivatives} and~\ref{sec:corrector-estimates}
contain the detailed weighted derivative and gluing estimates.
Appendix~\ref{sec:general-cones} outlines a strategy for the
extension to other minimizing quadratic cones formulated in
Conjectures~\ref{conj:connectors}
and~\ref{conj:quadratic-accumulation}.

\textbf{Acknowledgments} A.H. was supported by the European Research Council under Grant Agreement No.101076411 (ANGEVA). Views and opinions expressed are however those of the authors only and do not necessarily reflect those of the European Union or the European Research Council. Neither the European Union nor the granting authority can be held responsible for them. L.S. thanks the NSF for its support through the grants DMS-2044954 — CAREER: Fine Structure of the Singular Set in Some Geometric Variational Problems and DMS-2603324 — Higher Multiplicity Singularities in Geometric Variational Problems.

\section{Construction of the connector}\label{sec:connector}
We prove Theorem~\ref{thm:main-connector} by rescaling singular
solutions to a suitable family of Dirichlet problems with boundary data parametrized by $\eps$ on a fixed bounded domain. The two propositions below
give the required family and classify its singularities; their
proofs occupy Sections~\ref{sec:finite-domain-data}
and~\ref{sec:classification}. Each surface has a $C_{4,3}$ singularity
at zero, whereas the family converges, as $\eps\to 0$, with multiplicity one to $V$.

The rescaling must retain the transition between these two models.
We fix an intermediate density and choose the radius where the
density ratio attains that value. These radii shrink to zero, so
compactness gives an entire limit after dilation. The same density
normalization excludes both conical limits. The classification of
symmetric, axially monotone cones then identifies the tangent cone at
zero and the cone at infinity. This explains the role of the density
gaps recorded below and the order in which the two propositions
are used.

\subsection{Notation and preliminary results}
Write
\[
 t=|x|,\qquad r=(|x|^2+|z|^2)^{1/2},\qquad
 \rho=(r^2+y^2)^{1/2},\qquad
 J_2(r,y)=\frac{y^2}{r^2}-\frac13.
\]
Thus $r=\sqrt2t$ on $V$. For an eight-dimensional stationary varifold $N$ in Euclidean space,
\[
 \Theta(N,p,s)=\frac{\|N\|(B_s(p))}{\omega_8s^8},\qquad
 \omega_8=\frac{\pi^4}{24},\qquad
 \Theta(N,p)=\lim_{s\downarrow0}\Theta(N,p,s).
\]
For a cone centered at zero we abbreviate $\Theta(N)=\Theta(N,0)$. The main logic of the existence theorem is the following density gaps:
\begin{equation}\label{eq:density-gaps}
 \theta_{43}:=\Theta(C_{4,3})=\frac{768\sqrt{21}}{2401}
 <\theta_V:=\Theta(V)=\frac{15\pi}{32}
 <\theta_{13}:=\Theta(C_{1,3})=\frac{9\sqrt3\pi}{32}.
\end{equation}
The first gap provides the interval in which we choose the
normalizing density. The second excludes off-axis zeros, whose
density would be at least $\theta_{13}$. To use the first gap, we
must also show that the solutions of the bounded-domain Dirichlet problems have density below
$\theta_V$ at zero. Their leading Jacobi field $J_2$ produces this
strict deficit through the spherical link-area calculation in
Section~\ref{sec:link-area-deficit} dictated by the second variation along such jacobi field, which is a strictly negative direction. 

For a nonnegative function $u$ on a domain
$\Omega\subset\R^4_x\times\R_y$ we write
\[
 S(u)=\{(x,z,y)\in\R^4\times\R^4\times\R:\ (x,y)\in\Omega,\
 |z| = u(x,y)\}.
\]
If $u\in C^2(\Omega)$ is positive, then $S(u)$ is a smooth
hypersurface, and it is minimal if and only if $u$ solves the
symmetric minimal surface equation (SME)
\begin{align}\label{eq:SME}
    \left(\delta_{ij} - \frac{u_i u_j}{1+|Du|^2}\right)u_{ij} = \frac{3}{u}.
\end{align}

Here indices refer to the five coordinates $(x,y)$, and repeated
indices are summed. As in \cite{FS}, we work with locally uniform
limits of positive solutions: a nonnegative function
$u\in C^0(\Omega)$ is called an \emph{SME limit} on $\Omega$ if it
is the locally uniform limit on $\Omega$ of positive solutions of
\eqref{eq:SME} of class $C^2(\Omega)$, and an SME limit with nonempty
zero set is called a \emph{singular SME solution}. An \emph{SME
cone} is a set $S(u)$, where $u$ is an SME limit on $\R^5$ which is
positively homogeneous of degree one. By \cite[Corollary~2.5]{FS}, if
$u$ is a singular SME solution, then $S(u)$, with multiplicity one,
is a stationary integral varifold in $\Omega\times\R^4$, which is
smooth away from the points $(x,0,y)$ with $u(x,y)=0$ and singular at
each of them. Identifying $(x,y)$ with $(x,0,y)$, we write this as
\[
\{u= 0\}={\rm Sing} (S(u)) \,.
\]

We also use the Hardt--Simon phase. The cone $C_{3,3}$ has two
complementary components in $\R^8$, and each of them is foliated by
the dilations of a smooth, $O(4)\times O(4)$-invariant,
area-minimizing Hardt--Simon leaf \cite{HS}. Outside a compact set,
every leaf is a normal graph over the regular part of $C_{3,3}$, with
normal height $\lambda r^{-2}+O(|\lambda|^{4/3}r^{-3})$ for some
$\lambda\ne0$, where the normal of $C_{3,3}$ points towards
$\{|z|>|x|\}$ (see \cite[Proposition~3.3]{Sz}). We denote this leaf
by $ H_\lambda$, so that $H_\lambda\subset\{|z|>|x|\}$ for $\lambda>0$
and $ H_\lambda\subset\{|z|<|x|\}$ for $\lambda<0$, and we set
$ H_0=C_{3,3}$. 

The \emph{Hardt--Simon phase}
$P:\R^4_x\times\R^4_z\times\R_y\to\R$ is defined by $P=\lambda$ on
$ H_\lambda\times\R$ and $P=0$ on the axis $\{x=z=0\}$. It is
continuous, independent of $y$, positively homogeneous of degree three
in $(x,z)$, odd under the exchange of $x$ and $z$, and smooth on each
of the two open sides of $V$. In particular, there is a constant
$c_H>0$ such that
\[
 P(0,z,y)=c_H|z|^3,\qquad P(x,0,y)=-c_H|x|^3 .
\]

\subsection{Proof of Theorem \ref{thm:main-connector}} The proof relies on two propositions, which are proved in the next two sections. 

 The first one provides singular SME solutions on a fixed
bounded domain whose normalized deviation from $V$ is governed, as
$\eps\to0$, by the Jacobi field $J_2=y^2/r^2-1/3$. Its proof uses the
degree argument of \cite{FS} on a mean-convex domain, with a family
of boundary data depending on $\eps$, which is designed to produce
this Jacobi field in the limit $\eps\to0$.

Fix $R>0$ and $0<Y_0<Y$, and choose a smooth bounded convex rounded
cylinder $\Omega\subset\R^4_x\times\R_y$, invariant under $O(4)_x$
and reflection in $y$, with vertical side
$\{|x|=R,\ |y|\le Y_0\}$ and strictly convex caps ending at the
poles $(0,\pm Y)$, as in Section~\ref{sec:finite-domain-data}.
The domain, including the cap shape, is fixed throughout the
Dirichlet construction and is independent of $\eps$;
see Figure~\ref{fig:finite-domain}.

\begin{figure}[htbp]
\centering
\begin{tikzpicture}[x=1.05cm,y=1.05cm,>=latex,
    every node/.style={font=\small},line cap=round,line join=round]
  \path[fill=black!4]
    (1.35,-1.15) -- (1.35,1.15)
    .. controls (1.35,2.05) and (0.82,2.65) .. (0,2.65)
    .. controls (-0.82,2.65) and (-1.35,2.05) .. (-1.35,1.15)
    -- (-1.35,-1.15)
    .. controls (-1.35,-2.05) and (-0.82,-2.65) .. (0,-2.65)
    .. controls (0.82,-2.65) and (1.35,-2.05) .. (1.35,-1.15)
    -- cycle;
  \fill[blue!7] (0,0) circle (0.86);
  \draw[black!35,densely dashed] (-1.35,1.15) -- (1.35,1.15);
  \draw[black!35,densely dashed] (-1.35,-1.15) -- (1.35,-1.15);
  \draw[->,black!55] (-1.9,0) -- (2.15,0)
    node[right,text=black] {$x_1$};
  \draw[->,black!55] (0,-2.98) -- (0,3.08)
    node[above,text=black] {$y$};
  \draw[line width=0.9pt]
    (1.35,-1.15) -- (1.35,1.15)
    .. controls (1.35,2.05) and (0.82,2.65) .. (0,2.65)
    .. controls (-0.82,2.65) and (-1.35,2.05) .. (-1.35,1.15)
    -- (-1.35,-1.15)
    .. controls (-1.35,-2.05) and (-0.82,-2.65) .. (0,-2.65)
    .. controls (0.82,-2.65) and (1.35,-2.05) .. (1.35,-1.15);
  \draw[blue!65!black,dashed,line width=0.7pt] (0,0) circle (0.86);
  \fill (0,2.65) circle (1.3pt);
  \fill (0,-2.65) circle (1.3pt);
  \fill (0,0) circle (1pt);
  \node[above left] at (-0.04,2.65) {$(0,Y)$};
  \node[below left] at (-0.04,-2.65) {$(0,-Y)$};
  \node[below left] at (0,0) {$0$};
  \node[below right] at (1.35,-0.03) {$R$};
  \node[below left] at (-1.35,-0.03) {$-R$};
  \node[above left,fill=black!4,inner sep=1pt] at (0,1.15) {$Y_0$};
  \node[below left,fill=black!4,inner sep=1pt] at (0,-1.15) {$-Y_0$};
  \node at (-0.55,1.83) {$\Omega$};
  \draw[->,blue!65!black,line width=0.6pt]
    (0,0) -- (0.608,0.608)
    node[midway,below right,inner sep=1pt] {$4r_*$};
  \draw[black!60,line width=0.5pt] (1.35,0.68) -- (2.15,0.93);
  \node[right,align=left] at (2.15,0.93)
    {vertical side\\$|x|=R,\quad |y|\le Y_0$};
  \draw[black!60,line width=0.5pt] (0.94,2.23) -- (2.15,2.28);
  \node[right] at (2.15,2.28) {smooth convex cap};
  \draw[blue!65!black,line width=0.5pt] (0.61,-0.61) -- (2.15,-1.27);
  \node[right,align=left,text=blue!65!black] at (2.15,-1.27)
    {interior ball\\$\overline{B^5_{4r_*}}\subset\Omega$};
\end{tikzpicture}
\caption{A schematic meridian section of the fixed domain $\Omega$,
in the plane $x_2=x_3=x_4=0$. The full five-dimensional domain is
invariant under rotations of $x\in\R^4$; in this section $|x|=|x_1|$.
The dashed circle is the section of the interior ball used in
Proposition~\ref{SME:prop:finite-family}.}
\label{fig:finite-domain}
\end{figure}
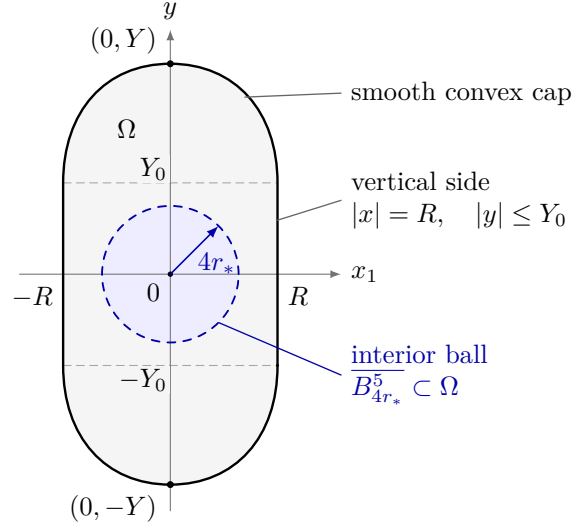

\begin{proposition}[Existence of SME solution]
\label{SME:prop:finite-family}
For the domain $\Omega$ fixed above, choose $r_*>0$ such that $B_{4r_*}^5\subset \Omega$. There exists $\eps_0>0$ such that, for every $0<\eps<\eps_0$, there exists a singular SME solution $u_\eps$ on $\Omega$, continuous on $\overline\Omega$, with the following properties:
\begin{enumerate}
    \item $u_\eps$ is $O(4)_x$-invariant, even in $y$, and nondecreasing in $y\ge0$; the varifold $M_\eps=S(u_\eps)$, with multiplicity one, is stationary in $\Omega\times\R^4$.
    \item $\{u_\eps=0\}={\rm Sing} (S(u_\eps))\neq \emptyset$. 
    \item There is a constant $K>0$, independent of $\eps$, such that the uniform bound
    \[
    |P|\leq K\,\eps\qquad \text{holds on }S(u_\eps),
    \]
    and $\sup_{0<\eps<\eps_0}\operatorname{Lip}(u_\eps;U)<\infty$ for every $U\Subset\Omega$.
    \item $u_\eps\to |x|$ locally uniformly in $\Omega$, as $\eps\to 0$, and consequently $S(u_\eps)\to V$ as varifolds in $\Omega\times\R^4$, with multiplicity one.
    \item There are a constant $C_*>0$ and smooth functions $f_\eps\colon V\cap B^9_{2r_*}\cap\{r\ge C_*\eps^{1/3}\}\to \R$ such that the part of $S(u_\eps)$ lying over this region is the normal graph $\{X+f_\eps(X)\nu_V(X)\}$, where $\nu_V$ is the unit normal of $V$ pointing towards $\{|z|>|x|\}$, and
\begin{equation}\label{roadmap:eq:regular-bounds}
 |D^k f_\eps|\le C_k\eps r^{-2-k},
 \qquad k=0,1,2,
\end{equation}
with constants independent of $\eps$.
    \item Let $\mathcal D\subset\R^3_\xi\times\R_y$ be obtained from
$\Omega$ by replacing $|x|$ with $|\xi|/\sqrt2$, and put $r=|\xi|$.
There is a harmonic, $O(3)_\xi$-invariant, reflection-even function
$W$ on $\mathcal D$ such that, with $m_0=W(0,0)$,
\[
 W_y(0,y)>0\quad(y>0,\ (0,y)\in\mathcal D),
\]
near the origin 
\begin{equation}\label{roadmap:eq:quadratic-phase}
 W(r,y)-m_0
   =c\left(y^2-\frac{r^2}{3}\right)+E(r,y),
 \qquad c>0,
\end{equation}
where the analytic remainder satisfies
\[
 |D^k_{\xi,y}E|\le C_k\rho^{4-k},\qquad k=0,1,2,3,4,
\]
for sufficiently small $\rho>0$, and
\[
 \frac{f_\eps}{\eps}
       \longrightarrow \frac{W-m_0}{r^2}\qquad \text{smoothly on regular compact subsets of $V$ over $\Omega$ as $\eps\to 0$}.
\]
\end{enumerate}
\end{proposition}

The second proposition classifies the singular set of the
solutions $u_\eps$ produced by
Proposition~\ref{SME:prop:finite-family}. Its proof combines a
classification of symmetric tangent cones with a second-variation
computation for the area of the link of $V$ in the direction of
$J_2$; the latter yields a strict drop of the density below
$\Theta(V)$.

\begin{proposition}[Classification of the singularities]\label{prop:SME-singularity-classification} There is $\eps_1\in(0,\eps_0]$ such that, for every $0<\eps<\eps_1$, the singular SME solution $u_\eps$ of Proposition~\ref{SME:prop:finite-family} has the following properties:
\begin{enumerate}
    \item $\{u_\eps = 0\}={\rm Sing}(S(u_\eps)) = \{(0,0)\}$.
    \item The density at the origin satisfies $\Theta(M_\eps,0) = \theta_{43} = \Theta(C_{4,3})$. Moreover, $C_{4,3}$ is the unique tangent cone of $M_\eps$ at the origin.
\end{enumerate}
\end{proposition}

Using Propositions~\ref{SME:prop:finite-family}
and~\ref{prop:SME-singularity-classification} together with the
compactness theory of \cite{FS}, we now prove
Theorem~\ref{thm:main-connector}. The idea is to zoom in at the scale
at which $S(u_\eps)$ passes from one cone to the other. Choose
$\eps_j\to0$ and put $u_j=u_{\eps_j}$. Each $M_j=S(u_j)$ has tangent
cone $C_{4,3}$ at the origin, whereas $M_j\to V$ at fixed scales.
Fixing $\theta_{43}<d<\theta_V$ and choosing $\rho_j$ with
\[
\Theta(M_j,0,\rho_j)=d,
\]
we obtain $\rho_j\to0$. Rescaling by $\rho_j^{-1}$ places this
transition at unit scale, and compactness gives an entire limit $M$
with
\[
\Theta(M,0,1)=d;
\]
in particular, $M$ is neither of the two cones. Monotonicity and the
classification of cones then identify $C_{4,3}$ as the tangent cone
at the origin and $V$ as the tangent cone at infinity, while the
density bound and axial monotonicity exclude further singularities.

\begin{proof}[Proof of Theorem~\ref{thm:main-connector}] Let $u_\eps$, $0<\eps<\eps_1$, be the singular SME solutions given by Propositions~\ref{SME:prop:finite-family} and~\ref{prop:SME-singularity-classification}.
Choose a sequence $\eps_j\downarrow0$, with each $\eps_j$ small enough
for both propositions to apply, and set
\[
 u_j:=u_{\eps_j},\qquad M_j:=S(u_j),\qquad X=(x,y)\in\R^5.
\]
Proposition~\ref{SME:prop:finite-family}(4) gives
\[
 u_j\longrightarrow |x|\quad\text{locally uniformly in }\Omega,
 \qquad M_j\longrightarrow V
 \quad\text{with multiplicity one over }\Omega.
\]

\noindent\textbf{Step 1. Choose a shrinking radius at density \(d\).}
Proposition~\ref{SME:prop:finite-family}(3) supplies a
constant $L$, independent of $j$, such that
$\operatorname{Lip}(u_j;B^5_{r_*})\le L$.
Proposition~\ref{prop:SME-singularity-classification}(1) gives $u_j(0)=0$,
so
\[
 0\le u_j(X)=u_j(X)-u_j(0)\le L|X|
 \qquad(X\in B^5_{r_*}).
\]
Moreover, Proposition~\ref{prop:SME-singularity-classification}(2)
identifies the tangent cone and density of each $M_j$:
\[
 \operatorname{Tan}(M_j,0)=\{C_{4,3}\},
 \qquad \Theta_{M_j}(0)=\theta_{43}.
\]

To choose a radius with density ratio exactly $d$, and later pass
that normalization to the limit, we use the fact that spheres carry
no mass. Indeed, on the regular part of $M_j$, a smooth minimal
hypersurface of dimension eight, the position vector $p\in\R^9$
satisfies
\[
\Delta_M|p|^2=16,
\]
so $|p|^2$ is not constant on any open subset of it. Analyticity implies
that every spherical level has zero eight-dimensional measure on the
regular part. The singular set also has zero eight-dimensional
measure. Hence $\|M_j\|(\partial B_r)=0$ for every $r>0$; the same
argument will apply to the limiting SME graph below. Consequently,
the density ratio is continuous in $r>0$, and varifold convergence
gives convergence of the mass of each fixed ball whose boundary has
zero mass for the limit.

Fix $d\in(\theta_{43},\theta_V)$ and a sufficiently small fixed
$r_0<r_*$. Multiplicity-one convergence to $V$ gives
\[
 \Theta_{M_j}(0)=\theta_{43}<d<
 \Theta_{M_j}(0,r_0)\longrightarrow\theta_V
\]
for all large $j$. Continuity and monotonicity therefore give a radius
$\rho_j\in(0,r_0)$ with $\Theta_{M_j}(0,\rho_j)=d$. For every fixed
$0<r\le r_0$,
\[
 \Theta_{M_j}(0,r)\longrightarrow\theta_V>d
 \quad\Longrightarrow\quad \rho_j<r\ \text{eventually}.
\]
Consequently $\rho_j\to0$.

\medskip\noindent\textbf{Step 2. Pass the normalization to an entire limit.}
Set $U_j(X)=\rho_j^{-1}u_j(\rho_jX)$. For each fixed $R$, once
$\rho_jR<r_*$,
\[
 U_j(0)=0,\qquad
 \|U_j\|_{L^\infty(B^5_R)}\le LR,\qquad
 \operatorname{Lip}(U_j;B^5_R)\le L.
\]
Arzel\`a--Ascoli and a diagonal choice of positive SME approximants give
an entire SME limit $U$ with $U(0)=0$ and global Lipschitz constant at
most $L$. The compactness statements in \cite{FS} yield
\[
 U_j\longrightarrow U\quad\text{locally uniformly},\qquad
 S(U_j)\longrightarrow M=S(U)\quad\text{with multiplicity one}.
\]
The convergence is smooth where $U>0$, and the symmetries and axial
order pass to $U$. The varifold convergence follows from the area
formula and dominated convergence, as in the proof of
Proposition~\ref{SME:prop:finite-family}(4), since the zero set of
$U$ has vanishing five-dimensional measure by
\cite[Theorem~9.1]{FS}.

For each fixed $R$ and all large $j$, monotonicity gives
\[
 \Theta_{S(U_j)}(0,R)
 =\Theta_{M_j}(0,\rho_jR)
 \le\Theta_{M_j}(0,r_0).
\]
The zero spherical mass established in Step~1 allows passage at the
specified radii, including $R=1$. Thus
\[
 \Theta_M(0,1)=d,\qquad
 \Theta_M(0,R)\le\theta_V\quad(R>0),\qquad
 \Theta_M(\infty)\le\theta_V.
\]

\medskip\noindent\textbf{Step 3. Identify the singular point and its tangent.}
For $s>0$, the rescaled profile $s^{-1}U(sX)$ is an SME limit on $\R^5$ and
has the same Lipschitz bound, rotational and reflection symmetries,
and axial monotonicity as $U$. Since $U(0)=0$, these profiles are
uniformly bounded on each fixed ball. SME compactness therefore
represents every tangent cone of $M$ at zero as an SME cone
with those same properties. Its density is at most
$\Theta_M(0,1)=d<\theta_V$, so
Lemma~\ref{lem:monotone-cone-classification} gives
\[
 \operatorname{Tan}(M,0)=\{C_{4,3}\}.
\]
Lemma~\ref{lem:off-axis-density-extraction} excludes off-axis zeros
using the global density bound. If $U(0,y_0)=0$ for $y_0\ne0$,
reflection and order would give
\[
 0\le U(0,y)\le U(0,|y_0|)=0\qquad(|y|\le|y_0|).
\]
For any $s_i\downarrow0$, the profiles $s_i^{-1}U(s_iX)$ would then
vanish at $(0,1)$ for all large $i$. Locally uniform convergence
would pass this axial zero to a tangent cone, contradicting
$C_{4,3}$. Therefore $U>0$ away from zero, and SME regularity and
\cite[Corollary~2.5]{FS} give
\[
 \operatorname{Sing}M=\{0\}.
\]

\medskip\noindent\textbf{Step 4. Identify the cone at infinity and obtain strict order.}
The same global Lipschitz bound gives compactness of $R^{-1}M$ as
$R\to\infty$ in the SME class. Let $K_\infty$ be a subsequential
limit along $R_i\to\infty$. For $0<a<b<\infty$,
\[
 \Theta_M(0,R_ib)-\Theta_M(0,R_ia)\longrightarrow0.
\]
The vanishing monotonicity defect makes $K_\infty$ a stationary
cone, and
\[
 d\le\Theta_{K_\infty}(0)=\Theta_M(\infty)\le\theta_V.
\]
It inherits the symmetries and axial order. Lemma~\ref{lem:monotone-cone-classification} now
gives
\[
 K_\infty\in\{C_{4,3},V\},\qquad
 \Theta_{K_\infty}(0)\ge d>\theta_{43}
 \quad\Longrightarrow\quad K_\infty=V.
\]
Compactness and this identification of every subsequential limit show
that the blow-down is unique and $\Theta_M(\infty)=\theta_V$.
Positive-height SME regularity gives smooth convergence on regular
compact subsets of $V$ away from its axis.

Finally, $U_y\ge0$ solves the differentiated SME on the connected
region $y>0$, where $U>0$. If $U_y$ were identically zero there,
reflection and continuity would make $U$ independent of $y$ and
give an entire zero axis, already excluded. The local Harnack
inequality for this nonnegative solution therefore yields $U_y>0$.
\end{proof}

\section{Proposition~\ref{SME:prop:finite-family}: The Dirichlet problem on a bounded domain}\label{sec:finite-domain-data}

In this section we construct the singular solutions in
Proposition~\ref{SME:prop:finite-family} by means of the
following existence theorem of Fouladgar--Simon.

\begin{theorem}[Fouladgar--Simon {\cite[Theorem~4.4]{FS}}]
\label{thm:fs-existence}
Fix integers $n,m\ge2$ and $0<\alpha<1$. On a uniformly convex
$C^{2,\alpha}$ domain $\Omega\subset\mathbb R^n$, consider
\[
 \sum_{i,j=1}^n\left(\delta_{ij}-\frac{u_i u_j}{1+|Du|^2}\right)u_{ij}
       =\frac{m-1}{u}.
\]
Let $\varphi_\tau\in C^{2,\alpha}(\overline\Omega)$,
$\tau\in[0,1]$, with $(X,\tau)\mapsto\varphi_\tau(X)$ jointly
continuous. Assume the following conditions (cf.\ \cite[Definition~4.1]{FS}).
\begin{enumerate}
\item No positive $C^2(\Omega)\cap C^0(\overline\Omega)$ solution
has trace $\varphi_0$.
\item On $\partial\Omega$, $\varphi_1>\psi$ for some
$\psi\in C^0(\partial\Omega)$ that is \emph{strongly positive}:
there exists $\eta>0$ such that every positive
$C^2(\Omega)\cap C^0(\overline\Omega)$ solution $v$ with
$v|_{\partial\Omega}\ge\psi$ satisfies $\inf_\Omega v\ge\eta$.
\end{enumerate}
Then, for some $\tau_*\in(0,1)$, there is a singular solution $u$
with trace $\varphi_{\tau_*}$ and
\[
 \varnothing\ne\{u=0\}\subset\Omega,\qquad
 \operatorname{dist}(\{u=0\},\partial\Omega)>0.
\]
Here a singular solution is a nonnegative function
$u\in C^0(\overline\Omega)$ with a zero in $\Omega$ which is the
locally uniform limit in $\Omega$ of positive classical solutions; for
$n=5$ and $m=4$ this is a singular SME solution in the sense of
Section~\ref{sec:connector}.
\end{theorem}

The remark following \cite[Theorem~4.4]{FS} states that mean convexity
of $\partial\Omega$ suffices. We use that extension for our rounded
cylinder, whose straight side is not uniformly convex. We also use
the proof in the $O(4)_x$-invariant, reflection-even space; the
preservation of this space is checked in Section~\ref{sec:finite-fs}.
For our boundary paths, which have a positive lower bound at each
fixed $\eps$, the \cite{FS} construction supplies positive approximating
solutions whose traces belong to the same path. We retain this
property to pass our a priori estimates to the singular solution.

We now prove all the conclusions of
Proposition~\ref{SME:prop:finite-family}. Section~\ref{sec:finite-data}
chooses the domain and boundary path. Section~\ref{sec:finite-estimates}
establishes estimates for any positive symmetric solution with
these data, without assuming existence. In
Section~\ref{sec:finite-fs}, we verify the \cite{FS} endpoint conditions,
obtain the singular solution for each fixed $\eps$, and prove its
symmetries, monotonicity, phase and derivative bounds, and limiting
behavior as $\eps\to0$.

The harmonic function $W$ introduced in
Section~\ref{sec:finite-data} plays two roles. Before applying
\cite{FS}, the negative value at the center of the normalized limit
$W+b$ rules out positive solutions at the lower endpoint of the
boundary path.
After applying \cite{FS}, the existence of a zero gives the opposite
coordinate contact and forces the limiting normalized phase to
vanish at the center. This identifies the leading Jacobi field
needed in the next section.

\subsection{The domain and boundary path}\label{sec:finite-data} We first fix the domain and the boundary data. Fix $R>0$ and
$0<Y_1<Y_0<Y$, and let $\Omega\subset\R^4_x\times\R_y$ be a bounded
convex domain with smooth boundary, invariant under $O(4)_x$ and
under $y\mapsto-y$, whose boundary consists of the lateral side
$\{|x|=R,\ |y|\le Y_0\}$ and two strictly convex caps (the rounded
ends) with poles $(0,\pm Y)$; see Figure~\ref{fig:finite-domain}. In
particular, $\partial\Omega$ is mean convex. Choose a smooth even
function $D:[-Y,Y]\to\R$ with
    \begin{align}
        D\le0,\qquad D'(y)>0\quad(0<y<Y_1),\qquad
    D=0\quad(|y|\ge Y_1).
    \end{align}
Fix $B_0>0$ and define
    \begin{align}
        q(y) = B_0 + \sqrt{2}R^2 D(y).
    \end{align}
For $\lambda>0$ we write $H_\lambda=\{|z|=h_\lambda(|x|)\}$
for the Hardt--Simon leaf of phase $\lambda$ (see
Section~\ref{sec:connector}); thus $h_\lambda(0)=(\lambda/c_H)^{1/3}$,
and $h_\lambda$ is increasing (cf.\ Section~\ref{sec:end}). For $\eps>0$
and $b$ with $B_0+b>0$ we define the boundary data
    \begin{equation}\label{eq:barrel-trace}
 g_{\eps,b}(x,y)=
 \begin{cases}
 h_{\eps(B_0+b)}(R)+\eps D(y),
          & |x|=R,\ |y|\le Y_0,\\
 h_{\eps(B_0+b)}(|x|),&\text{on the rounded ends}.
 \end{cases}
\end{equation}
As $\eps\to0$, these boundary data converge to $|x|$, the
boundary values of $V$, since the phase of the leaf tends to zero.
Figure~\ref{fig:boundary-radius-and-q} compares the fixed profile $q$
with the prescribed radius $g_{\eps,b}$. They are related as follows:
\begin{equation}\label{e:boq}
 \frac{g_{\eps,b}(R,y)-h_{\eps(B_0+b)}(R)}{\eps}
   =D(y)=\frac{q(y)-B_0}{\sqrt2R^2}.
\end{equation}
Thus, to first order in $\eps$, the lateral boundary data have
normalized phase $q(y)+b$ (see the proof of
Proposition~\ref{prop:harmonic-asymptotics}); this is why the
normalized limit of $u_\eps$ is described by the harmonic function $W$
below.
On the rounded ends the radius varies with $|x|$ along that same
leaf; its phase, rather than its radius, is constant there.
Since $D\le0$, we have
$g_{\eps,b}(x,y)\le h_{\eps(B_0+b)}(|x|)$ on all of
$\partial\Omega$, with equality outside the perturbed part of the side.

\begin{figure}[htbp]
\centering
\begin{tikzpicture}[font=\small,>=stealth]
  \begin{scope}[shift={(0,6.9)}]
    \node[font=\small\bfseries] at (3,3.7) {Fixed profile $q(y)$};
    \draw[->,gray] (.3,.2)--(6.05,.2) node[right] {$y$};
    \draw[->,gray] (.3,.2)--(.3,3.1);
    \draw[gray,dashed] (.3,2.6)--(5.6,2.6);
    \node[left] at (.25,2.6) {$B_0$};
    \draw[orange!85!black,very thick] (.6,2.6)--(1.6,2.6);
    \draw[orange!85!black,very thick,domain=-1:1,samples=81]
      plot ({3+1.4*\x},{2.6-1.5*(1-\x*\x)^2});
    \draw[orange!85!black,very thick] (4.4,2.6)--(5.4,2.6);
    \draw[gray,dotted] (3,.2)--(3,1.1);
 
    \fill[orange!85!black] (3,1.1) circle(1.7pt);
    \node[below right,font=\footnotesize] at (3.05,1.1) {$q(0)$};
    \draw[<->,orange!85!black] (3,1.2)--(3,2.5);
    \node[above,font=\footnotesize] at (3,2.68)
      {dip depth $\sqrt2R^2[-D(0)]$};
    \foreach \xx/\txt in {.6/-Y_0,1.6/-Y_1,3/0,4.4/Y_1,5.4/Y_0}
      {\draw[gray] (\xx,.15)--(\xx,.25);
       \node[below,font=\footnotesize] at (\xx,.1) {$\txt$};}
    \node[align=center,text width=5.8cm,font=\footnotesize] at (3,-1.15)
      {$q=B_0+\sqrt2R^2D$.\\The shape and level of $q$ are fixed.};
  \end{scope}
  \begin{scope}[shift={(7.3,6.9)}]
    \node[font=\small\bfseries] at (3,3.7) {Side radius $g_{\eps,b}(R,y)$};
    \draw[->,gray] (.3,.2)--(6.05,.2) node[right] {$y$};
    \draw[->,gray] (.3,.2)--(.3,3.1);
    \draw[gray,dashed] (.3,2.6)--(5.6,2.6);
    \node[anchor=south west,font=\footnotesize] at (.3,2.76)
      {reference height $h_{\eps(B_0+b)}(R)$};
    \draw[blue!70!black,very thick] (.6,2.6)--(1.6,2.6);
    \draw[blue!70!black,very thick,domain=-1:1,samples=81]
      plot ({3+1.4*\x},{2.6-.85*(1-\x*\x)^2});
    \draw[blue!70!black,very thick] (4.4,2.6)--(5.4,2.6);
    \draw[gray,dotted] (3,.2)--(3,1.75);
    \fill[blue!70!black] (3,1.75) circle(1.7pt);
    \draw[<->,blue!70!black] (3,1.85)--(3,2.5);
    \node[below,align=center,font=\footnotesize] at (3,1.6)
      {dip depth $\eps[-D(0)]$};
    \foreach \xx/\txt in {.6/-Y_0,1.6/-Y_1,3/0,4.4/Y_1,5.4/Y_0}
      {\draw[gray] (\xx,.15)--(\xx,.25);
       \node[below,font=\footnotesize] at (\xx,.1) {$\txt$};}
    \node[align=center,text width=5.8cm,font=\footnotesize] at (3,-1.15)
      {$g_{\eps,b}(R,y)=h_{\eps(B_0+b)}(R)+\eps D(y)$.\\
       $b$ changes the reference leaf; $\eps$ also scales the dip.};
  \end{scope}
  \node[font=\small\bfseries,align=center] at (6.7,4.65)
    {The radius trace on the whole boundary};
  \node[font=\footnotesize,align=center] at (6.7,4.05)
    {Write $t_\partial(y)=|x|$ on $\partial\Omega$:
     $t_\partial=R$ on the side, and $t_\partial(\pm Y)=0$.};
  \draw[->,gray] (.25,0)--(13.4,0) node[right] {$y$};
  \draw[->,gray] (.25,0)--(.25,3.15) node[above] {$|z|$};
  \draw[gray,dotted] (2.9,0)--(2.9,3.05);
  \draw[gray,dotted] (10.5,0)--(10.5,3.05);
  \node[font=\footnotesize,align=center] at (1.55,3.2) {Rounded end};
  \node[font=\footnotesize,align=center] at (6.7,3.2) {Straight side};
  \node[font=\footnotesize,align=center] at (11.85,3.2) {Rounded end};
  \draw[blue!70!black,very thick] (.6,.7)
    .. controls (.8,1.95) and (1.8,2.3)..(2.9,2.3)
    --(4.7,2.3);
  \draw[blue!70!black,very thick,domain=-1:1,samples=81]
    plot ({6.7+2*\x},{2.3-.85*(1-\x*\x)^2});
  \draw[blue!70!black,very thick] (8.7,2.3)--(10.5,2.3)
    .. controls (11.6,2.3) and (12.6,1.95)..(12.8,.7);
  \draw[gray,dashed,thick] (4.55,2.3)--(8.85,2.3);
  \node[above,font=\footnotesize] at (6.7,2.4)
    {$h_{\eps(B_0+b)}(t_\partial(y))$};
  \node[below,font=\footnotesize,text=blue!70!black] at (6.7,1.38)
    {$g_{\eps,b}(t_\partial(y),y)$};
  \fill[blue!70!black] (.6,.7) circle(1.8pt);
  \fill[blue!70!black] (12.8,.7) circle(1.8pt);
  \foreach \xx/\txt in {.6/-Y,2.9/-Y_0,4.7/-Y_1,6.7/0,8.7/Y_1,10.5/Y_0,12.8/Y}
    {\draw[gray] (\xx,-.05)--(\xx,.05);
     \node[below,font=\footnotesize] at (\xx,-.12) {$\txt$};}
  \node[align=center,font=\footnotesize,text width=12.7cm] at (6.7,-1.1)
    {Outside $|y|<Y_1$, the trace agrees with the reference leaf.
     At the two poles, $g_{\eps,b}(0,\pm Y)=h_{\eps(B_0+b)}(0)>0$.};
\end{tikzpicture}
\caption{The fixed profile $q$ and the actual boundary radius
$g_{\eps,b}$.}
\label{fig:boundary-radius-and-q}
\end{figure}
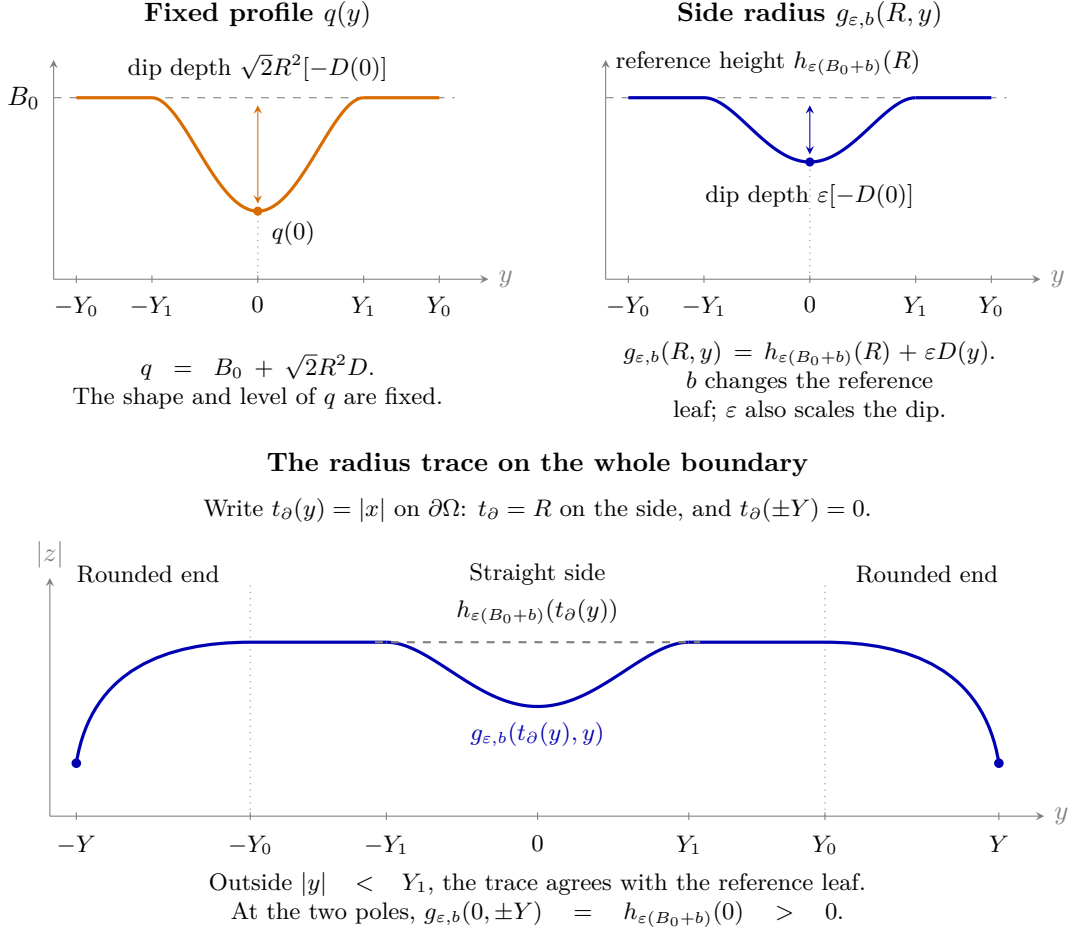

Let $\mathcal D\subset\R^3_\xi\times\R_y$ be the auxiliary domain
obtained from $\Omega$ by replacing $|x|$ with $|\xi|/\sqrt2$.
Let $W$ be the unique bounded harmonic function on $\mathcal D$ with
boundary values $q(y)$ on the vertical side and $B_0$ on the rounded ends:
\begin{align}\label{eq:harmonic-limit}
    \begin{cases}
        \Delta_{\xi,y}W = 0&\text{in }\mathcal D,\\
        W(\xi,y) = q(y)&\text{for } |\xi|=\sqrt2R,\\
        W(\xi,y) = B_0&\text{on }\partial\mathcal D\setminus\{|\xi|=\sqrt2R\}.
    \end{cases}
\end{align}
These data match smoothly. By uniqueness, $W$ is $O(3)_\xi$-invariant and even in $y$. Write
\[
 m_0=W(0,0).
\]
Since $q\le B_0$ and $q\not\equiv B_0$, the strong maximum
principle gives
\[
 \min q<m_0<B_0.
\]
Fix $\delta>0$ so small that
$\delta<(B_0-m_0)/2$, and choose the parameter interval
\begin{equation}\label{eq:barrel-b-interval}
 b_-=-m_0-\delta,\qquad b_+=-\min q+\delta.
\end{equation}
The interval $[b_-,b_+]$ is fixed independently of $\eps$, and
$B_0+b\ge B_0-m_0-\delta>0$ throughout it. For each fixed
sufficiently small $\eps>0$, set
\[
 \varphi_\tau=g_{\eps,b_-+\tau(b_+-b_-)},\qquad 0\le\tau\le1.
\]
We will apply Theorem~\ref{thm:fs-existence} to this path in the
symmetric space. It selects a pair $(b_\eps,u_\eps)$ with
$\min u_\eps=0$. The selected parameter may depend on $\eps$;
only its allowed interval is fixed independently of $\eps$.

\subsection{A priori estimates for symmetric solutions}\label{sec:finite-estimates} We now consider any positive, $O(4)_x$-invariant, reflection-even
SME solution with boundary trace
$g_{\eps,b}$, where $b\in[b_-,b_+]$. The estimates in this subsection
hold whenever such a solution exists. In particular, they
apply to a hypothetical sequence at the lower endpoint and will
allow us to rule it out in Section~\ref{sec:finite-fs}.
We first control the phase and axial monotonicity, then identify
the harmonic limit, and finally relate its sign to coordinate
contacts. In the linearized equations below we write
\[
 A^{ij}(p)=\delta_{ij}-\frac{p_ip_j}{1+|p|^2}.
\]

\begin{proposition}[Moving plane estimates]\label{prop:moving-plane-phase-estimate}
    There are constants $\eps_2>0$ and $K_0>0$, depending only on
$\Omega$, $B_0$, $D$ and $b_\pm$, with the following property. Let
$0<\eps<\eps_2$ and $b\in[b_-,b_+]$, let $u\in C^2(\overline\Omega)$
be a positive solution of \eqref{eq:SME} with
$u|_{\partial\Omega}=g_{\eps,b}$, and put $M=S(u)$. Then 
\begin{enumerate}
    \item $|P|\le
K_0\eps$ on $M$,
    \item $u(x,y)\le h_{\eps(B_0+b)}(|x|)$ in $\Omega$,
    \item $u$ is nondecreasing in $y\ge0$: $u(x,y_1)\le u(x,y_2)$ whenever $0\le y_1\le y_2$ and $(x,y_2)\in\Omega$.
\end{enumerate} 
\end{proposition}

\begin{proof}
By elliptic regularity, $u$ is smooth in $\Omega$. On the caps of
$\partial\Omega$ the boundary data lie on the leaf
$H_{\eps(B_0+b)}$, so that $P=\eps(B_0+b)$ there. On the lateral
side, $|x|=R$ and $g_{\eps,b}=R+O(\eps)$ uniformly in
$b\in[b_-,b_+]$; since $P$ is Lipschitz near the point
$(|x|,|z|)=(R,R)$, where it vanishes, we obtain $|P|\le K_0\eps$ on
$\partial M$ for a suitable constant $K_0$ and all small $\eps$. If
$P|_M$ attained its maximum $\lambda$ at an interior point, then $M$
would touch the smooth minimal hypersurface $ H_\lambda\times\R$ from
the side $\{P\le\lambda\}$ at that point. By the strong maximum
principle and the connectedness of $M$, this would give
$M\subset H_\lambda\times\R$, which is impossible because
$P|_{\partial M}$ is not constant, as $D\not\equiv0$. The same
argument applies to the minimum, and hence $|P|\le K_0\eps$ on $M$.
Since $D\le0$, moreover, $P\le\eps(B_0+b)$ on $\partial M$, and
therefore on $M$; that is,
\[
 u(x,y)\le h_{\eps(B_0+b)}(|x|).
\]
The inequality is strict in $\Omega$, because the lateral data are not
identically those of this leaf.

The monotonicity in $y$ follows from the moving plane method
\cite{GNN,BN}. Fix a plane $y=s$, $0<s<Y$, and

reflect the upper portion $\Omega_s=\Omega\cap\{y>s\}$ by
$y\mapsto2s-y$. Convexity and reflection symmetry of the domain imply
that the reflected portion is contained in $\Omega$. Put
\[
 v_s(x,y)=u(x,y)-u(x,2s-y).
\]
It vanishes on the reflecting plane. If $(x,y)\in\partial\Omega_s$
lies on a cap, then $u(x,y)=h_{\eps(B_0+b)}(|x|)\ge u(x,2s-y)$. If it
lies on the lateral side, then $|x|=R$ and $|2s-y|<y\le Y_0$, so that
$(x,2s-y)$ lies on the lateral side as well, and $D(2s-y)\le D(y)$
because $D$ is even and nondecreasing on $[0,Y]$.
Consequently
\[
 \begin{aligned}
 y=s:&\quad v_s=0,\\
 \partial\Omega\text{ on a rounded end}:&\quad
 v_s=h_{\eps(B_0+b)}(|x|)-u(x,2s-y)\geq0,\\
 \partial\Omega\text{ on the side}:&\quad
 v_s=\eps\bigl(D(y)-D(2s-y)\bigr)\geq0.
 \end{aligned}
\]
For the fixed positive smooth solution put $\widehat u_s(x,y)=
u(x,2s-y)$. Subtraction of the equations gives
\[
 A^{ij}(Du)(v_s)_{ij}+b_s^i(v_s)_i+c_sv_s=0,
 \qquad c_s=\frac3{u\widehat u_s}>0,
\]
where
\[
 b_s^k=\int_0^1\partial_{p_k}A^{ij}
     (D\widehat u_s+\tau Dv_s)\,d\tau\,
     (\widehat u_s)_{ij}.
\]
Ellipticity and all coefficient bounds are uniform in $s$ for this
fixed solution. The small-domain maximum principle, a
consequence of the Aleksandrov--Bakelman--Pucci estimate, takes the form
\[
 \|(v_s)^-\|_{L^\infty(E)}
 \leq C_{\rm MP}\|c_s\|_\infty|E|^{1/5}
          \|(v_s)^-\|_{L^\infty(E)},
 \qquad v_s|_{\partial E}\geq0,
\]
for every open set $E\subset\Omega_s$, where $C_{\rm MP}$ depends
only on $\operatorname{diam}\Omega$, the ellipticity constants and
$\sup_s\|b_s\|_\infty$.
For sufficiently small $|E|$ it forces $v_s\geq0$ in $E$, despite the
positive zero-order coefficient.

We now run the continuation argument of \cite{GNN}. Let $s_0$ be the
infimum of those $s\in(0,Y)$ such that $v_{s'}\ge0$ in $\Omega_{s'}$
for all $s'\in[s,Y)$. Since $|\Omega_s|\to0$ as $s\uparrow Y$, the
small-domain maximum principle gives $s_0<Y$. Suppose that $s_0>0$.
By continuity, $v_{s_0}\ge0$ in $\Omega_{s_0}$. Moreover,
$v_{s_0}\not\equiv0$, since for points $(x,y)$ of the upper cap the
reflected point $(x,2s_0-y)$ lies in $\Omega$, where
$u<h_{\eps(B_0+b)}(|x|)$. As $v_{s_0}\ge0$ and
$A^{ij}(Du)(v_{s_0})_{ij}+b^i_{s_0}(v_{s_0})_i=-c_{s_0}v_{s_0}\le0$,
the strong maximum principle gives $v_{s_0}>0$ in $\Omega_{s_0}$.
Choose a compact set $K\subset\Omega_{s_0}$ such that the
small-domain maximum principle applies on every open set of measure at
most $2|\Omega_{s_0}\setminus K|$. For $s<s_0$ sufficiently close to
$s_0$, uniform continuity gives $v_s>0$ on $K$ and
$|\Omega_s\setminus K|\le2|\Omega_{s_0}\setminus K|$; the small-domain
maximum principle on $E=\Omega_s\setminus K$, on whose boundary
$v_s\ge0$, then gives $v_s\ge0$ in $\Omega_s$. This contradicts the
definition of $s_0$, and hence $s_0=0$. Consequently,
\[
 0\leq y_1<y_2\quad\Longrightarrow\quad
 u(x,y_2)-u(x,y_1)=v_{(y_1+y_2)/2}(x,y_2)\geq0.
\]
These order inequalities pass to locally uniform singular limits.
Note that the moving-plane argument uses reflection symmetry of
the domain and of the boundary data, but does not
require any symmetry of the solution.
\end{proof}

The next proposition identifies the first-order effect of the
boundary data. A direct computation, recorded in the proof below, shows that for an invariant Jacobi field $f$ on $V$, the
conjugate $r^2f$ satisfies the ordinary harmonic equation in
$\R^3_\xi\times\R_y$. Boundedness removes the
auxiliary axis, so the boundary data determine the limit uniquely.
Reflection symmetry and axial monotonicity then force its first
nonconstant term at the origin to be a positive multiple of
$y^2-r^2/3$. The proof makes this reduction and the positivity
of the coefficient explicit.

\begin{proposition}[Vanishing $\eps$ asymptotics]\label{prop:harmonic-asymptotics}
    Let $\eps_i\downarrow0$ and $b_i\in[b_-,b_+]$ with $b_i\to b_*$.
For each $i$, let $u_i$ be either an $O(4)_x$-invariant positive
solution of \eqref{eq:SME} in $C^2(\overline\Omega)$ with
$u_i|_{\partial\Omega}=g_{\eps_i,b_i}$, or a uniform limit on
$\overline\Omega$ of such solutions with the same $\eps_i$ and with
parameters converging to $b_i$. Then $M_i=S(u_i)\to V$, and on every
compact subset of the regular part of $V$ over $\Omega$, for all large
$i$, $M_i$ is the normal graph of a function $f_i$, with
    \begin{align}
        \frac{f_i}{\eps_i}\to \frac{W(r,y)+b_*}{r^2} \quad\text{in } C^{\infty}_{loc}(\mathrm{Reg}( C_{3,3}\times \R)),
    \end{align}
    where $W(r,y)$ is the harmonic function defined in (\ref{eq:harmonic-limit}) and satisfies:
    \begin{equation}\label{eq:barrel-quadratic}
    W(r,y)=m_0 + c\left(y^2-\frac{r^2}{3}\right)+O((r^2+y^2)^2),
    \end{equation}
    with $m_0 = W(0,0)$ and $c>0$.
\end{proposition}
\begin{proof}
By Proposition~\ref{prop:moving-plane-phase-estimate}, $|P|\le
K_0\eps_i$ on $M_i$; this bound passes to uniform limits. Hence
$M_i\to V$ in the Hausdorff sense, and the argument in the proof of
Proposition~\ref{prop:height-derivative-analysis} below shows that,
on every compact subset of the regular part of $V$ over $\Omega$,
$M_i$ is eventually the normal graph of a function $f_i$ with
$|D^kf_i|\le C_k\eps_i$ for every $k$. Near $\partial\Omega$, away
from the poles $(0,\pm Y)$, the functions $u_i$ are bounded below by
a positive constant, and $(u_i-|x|)/\eps_i$ solves a uniformly
elliptic linear equation whose boundary values
$(g_{\eps_i,b_i}-|x|)/\eps_i$ are bounded in $C^{2,\alpha}$. Boundary
Schauder estimates therefore give bounds for $f_i/\eps_i$ which are
uniform up to $\partial\Omega\setminus\{(0,\pm Y)\}$. Consequently,
every subsequence of $f_i/\eps_i$ has a further subsequence which
converges smoothly on compact subsets of the regular part of $V$ over
$\Omega$, and continuously up to the boundary away from the poles.

On the side $t=R$, the conversion from vertical
height to normal height gives
\[
 \begin{aligned}
 f_i&=\frac{g_{\eps_i,b_i}-R}{\sqrt2}
                      +o(\eps_i),\\
 \frac{r^2f_i}{\eps_i}
      &\longrightarrow B_0+b_*+\sqrt2R^2D(y)=q(y)+b_*.
 \end{aligned}
\]
On the rounded ends the phase is exactly $\eps_i(B_0+b_i)$.
The polar points are removable for the bounded limiting harmonic
Dirichlet problem. Any subsequential limit $f=\lim_i f_i/\eps_i$ and its bounded
conjugate $w=r^2f$ satisfy the explicit linearization identities
\[
 \begin{aligned}
 L_Vf&=f_{rr}+\frac6r f_r+f_{yy}+\frac6{r^2}f=0,\\
 r^2L_V(r^{-2}w)&=w_{rr}+\frac2r w_r+w_{yy}
       =\Delta_{\R^3_\xi\times\R_y}w=0,\\
 w|_{\partial\mathcal D}&=
 \begin{cases}q(y)+b_*,&\text{on the side},\\B_0+b_*,&\text{on the rounded ends}.
 \end{cases}
 \end{aligned}
\]
Removing the auxiliary axis and using uniqueness of the bounded
Dirichlet problem gives
\begin{equation}\label{eq:barrel-harmonic-limit}
 w=W+b_*.
\end{equation}

The uniqueness shows that the entire sequence has this normalized
limit whenever $b_i\to b_*$. These estimates apply to singular limits
as well: on a fixed compact subset of the regular part of \(V\), the phase estimate keeps \(u_\varepsilon\) uniformly positive for sufficiently small \(\varepsilon\). Thus elliptic regularity still applies there, and the harmonic-limit argument goes through.

The reflection argument in the proof of
Proposition~\ref{prop:moving-plane-phase-estimate} applies to $W$,
with the ordinary maximum principle in place of the small-domain one.
Indeed, for $0<s<Y$ the function $W(\xi,y)-W(\xi,2s-y)$ is harmonic in
$\mathcal D\cap\{y>s\}$, vanishes on $\{y=s\}$, and is nonnegative on
$\partial\mathcal D\cap\{y>s\}$, since $W\le B_0$ with equality on the
rounded ends, and the lateral data $q$ are even and nondecreasing in
$|y|$. Hence $W_y\ge0$ for $y>0$. It is not identically zero, since
the boundary data vary. The strong maximum principle for its harmonic
derivative gives $W_y>0$ in the upper half-domain. On the central plane,
$W_y=0$, and Hopf's lemma at the origin of the auxiliary half-domain
gives
\[
 W_{yy}(0,0)>0.
\]
Analyticity and the two symmetries yield
\[
 W(r,y)=m_0+a_2 r^2+c y^2+O((r^2+y^2)^2),
 \qquad
 6a_2+2c=0,\qquad c=\tfrac12W_{yy}(0,0)>0.
\]
This proves the estimate, including the positivity of $c$ and the
analytic, $O(3)$-invariant, reflection-even remainder. In particular,
$W(0,y)$ has its unique minimum $m_0$ at $y=0$.
\end{proof}

\begin{proposition}\label{prop:height-derivative-analysis}
    Let $0<\eps<\eps_2$, and let $u_\eps$ be a positive solution as in
Proposition~\ref{prop:moving-plane-phase-estimate}, or a uniform limit
on $\overline\Omega$ of such solutions with the same $\eps$. There is a
constant $C_*$, independent of $\eps$ and $b$, such that the part of
$S(u_\eps)$ lying over $V\cap B^9_{2r_*}\cap\{r\ge C_*\eps^{1/3}\}$ is the
normal graph over $V$ of a function $f_\eps$ satisfying
\begin{equation}
 |D^k f_\eps|\le C_k\eps r^{-2-k},
 \qquad k=0,1,2,
\end{equation}
with constants independent of $\eps$, and the same bound holds
for every $k\ge0$, with $C_k$ depending on $k$.
\end{proposition}
\begin{proof}
    By Proposition~\ref{prop:moving-plane-phase-estimate},
$|P|\le K_0\eps$ on $S(u_\eps)$, and this bound passes to uniform limits.
Given this bound, the leaf expansion and separation estimates
\cite[Proposition~3.3 and Lemma~5.1]{Sz} give
\[
 r\ge C_*\eps^{1/3}
 \quad\Longrightarrow\quad
 t\approx u_\eps\approx r,\qquad
 |u_\eps-t|\le C\eps r^{-2}.
\]
Fix a point of this region, and work in a base ball of radius $cr$ about it, contained in the interior, and
dilate by $r^{-1}$. The rescaled functions
$\widehat u=r^{-1}u_\eps$ and $\widehat t=r^{-1}t$
have uniformly positive height and bounded derivatives, by
\cite[Theorem~8.1]{FS} and positive-height elliptic regularity.
Their difference $z=\widehat u-\widehat t$ satisfies
\[
 a^{ij}z_{ij}+b^iz_i+\frac{3}{\widehat u\,\widehat t}\,z=0,
 \qquad
 \|z\|_{C^k(B_{c/2})}\le C_k\|z\|_{C^0(B_c)}
                         \le C_k\eps r^{-3},
\]
with uniform ellipticity and coefficient bounds. Here and below a
$C^k$ estimate means any fixed finite order. Scaling back gives the
same bounds for $u_\eps-t$. For $C_*$ sufficiently large,
the $C^1$ error makes normal projection onto $V$ a diffeomorphism.
Its normal height therefore satisfies
\begin{equation}\label{eq:barrel-regular-estimates}
 |D^k f_\eps|\le C_k\eps r^{-2-k},
 \qquad r\ge C_*\eps^{1/3}.
\end{equation}
The constants are uniform on each fixed interior region.
\end{proof}

The preceding normal-graph convergence holds away from the spine
$\{x=z=0\}$. Although the limiting function $w=r^2f$ extends
continuously to the spine, we cannot evaluate the normal heights
$f_i$ there: the reference cylinder is singular at those points.
The next lemma supplies the missing link between this regular
convergence and intersections of the hypersurfaces with the
coordinate subspaces $\{x=0\}$ and $\{z=0\}$.
The phase is nonnegative on $\{x=0\}$ and nonpositive on
$\{z=0\}$. The lemma transfers these signs to the limiting value
$w(0,y_*)$, using barriers near the spine. The lemma is needed both before existence, to check the lower
endpoint in \cite{FS}, and after existence, to select the limiting phase.

\begin{lemma}[Signed transmission at the spine]
\label{lem:signed-coordinate-contact}
Let $I\subset\R$ be an open interval, and let $\mathcal U\subset\R^9$
be a fixed open neighborhood of the spine segment
\[
 \{(0,0,y):y\in I\}.
\]
Let $M_i$ be multiplicity-one stationary integral varifolds in
$\mathcal U$ of the form $S(u_i)$, where the $u_i$ are $O(4)_x$-invariant
SME limits (so that $M_i$ is invariant
under $O(4)_x\times O(4)_z$), and let $\eps_i\downarrow0$.
Assume that their signed Hardt--Simon phase satisfies
\[
 |P|\le K\eps_i\quad\text{on }\operatorname{spt}M_i\cap\mathcal U
\]
for a constant $K$ independent of $i$. Suppose that, near each compact subset
of $V\cap\mathcal U\cap\{r>0\}$, the whole support of $M_i$
is, for all sufficiently large $i$, a single normal graph over $V$.
Write $f_i$ for its signed Euclidean normal height, and assume
\[
 \frac{r^2f_i}{\eps_i}\longrightarrow w(r,y)
 \quad\text{smoothly on compact subsets of }
 V\cap\mathcal U\cap\{r>0\},
\]
where $w$ extends continuously to $r=0$ for $y\in I$.

For any $y_*\in I$, the following conclusions hold:
\begin{enumerate}[label=(\roman*),leftmargin=*]
\item If there are points $(0,z_i,y_i)\in\operatorname{spt}M_i
\cap\mathcal U$ with $y_i\to y_*$, then $w(0,y_*)\ge0$.
\item If there are points $(x_i,0,y_i)\in\operatorname{spt}M_i
\cap\mathcal U$ with $y_i\to y_*$, then $w(0,y_*)\le0$.
\end{enumerate}
In each case the phase bound forces the contact points to converge
to $(0,0,y_*)$. If both types of contact approach the same spine
point, then $w(0,y_*)=0$.
\end{lemma}

\begin{proof}
We recall precisely the signed barrier input. Let $H= H_1$ be the
Hardt--Simon leaf with $P=1$. For a uniformly $C^2$-bounded
family of real functions $q$ on a fixed axial interval,
\cite[Propositions~5.8 and~5.10]{Sz} give barriers $X_q$ modeled on
the slicewise dilated leaves
$\eps^{1/3}q(y)^3H$; when $q(y)<0$, the dilation by the negative
factor $\eps^{1/3}q(y)^3$ is understood as the dilation by
$\eps^{1/3}|q(y)|^3$ followed by the exchange of the factors $x$ and
$z$, so that in all cases the leaf has phase $\eps q(y)^9$. Given
$\eta>0$, there is a lateral radius $r_0>0$, depending only on
$\eta$ and on the family, such that for all sufficiently small $\eps$
the barriers exist in $\{r<r_0\}$ and satisfy
\begin{equation}\label{eq:signed-barrier-trapping}
 \eps(q^9-\eta)\le P|_{X_q}\le\eps(q^9+\eta).
\end{equation}
The constants are uniform for the family. Each barrier separates
the comparison region into a negative and a positive side, lies between
the two comparison leaves of phases $\eps(q^9\pm\eta)$, and cannot be
touched from its negative side by a stationary hypersurface of the
type considered here. The cited result includes
contact at their singular spine points when $q=0$; ordinary smooth
comparison alone would not suffice there. Exchanging $x$ and $z$
reverses both the phase and the comparison direction.

We prove (i). Suppose that contact points as in (i) exist, but
$w(0,y_*)<0$. Choose $c>0$ and $I_0\Subset I$ containing
$y_*$ so that $w(r,y)<-3c$ for small $r$ and $y\in I_0$.
Choose $\psi\in C_c^\infty(I_0)$ with $0\le\psi\le1$ and
$\psi=1$ near $y_*$. For $A$ sufficiently large, use
\[
 q_s=A-s\psi,\qquad 0\le s\le A+c^{1/9},\qquad
 A^9-c/4>K.
\]
This is a uniformly $C^2$-bounded signed family, with $q_s^9\ge-c$,
$q_s=A$ near the axial ends, and final value $q_s^9=-c$ near $y_*$.
Choose one lateral radius $r_0$ for the whole family in
\eqref{eq:signed-barrier-trapping}, with $\eta=c/4$.
The fixed-leaf expansion converts regular graphical convergence
to $P/\eps_i\to w$ on $r=r_0$. Thus, for large $i$,
\[
 \begin{aligned}
 r=r_0:&\quad P|_{M_i}<-2c\eps_i
               <-\tfrac54c\eps_i\le P|_{X_{q_s}},\\
 y\in\partial I_0\ \hbox{or}\ s=0:&\quad
 P|_{M_i}\le K\eps_i<(A^9-c/4)\eps_i\le P|_{X_{q_s}}.
 \end{aligned}
\]
The cylinder is chosen compactly inside the common domain.
These inequalities order its entire comparison boundary. Starting
with $s=0$, a first loss of containment cannot occur at that boundary
or in its interior, by the signed barrier comparison just recalled.
At the final parameter the barrier near $y_*$ has phase at most
$-3c\eps_i/4$; its negative side contains no point with $x=0$.
But $P(0,z_i,y_i)=c_H|z_i|^3$ and the phase bound give
$|z_i|=O(\eps_i^{1/3})$, so the assumed contact is in this comparison
region for large $i$. This is a contradiction. Exchanging the two
four-dimensional factors proves the other assertion.
\end{proof}

\subsection{Proof of Proposition~\ref{SME:prop:finite-family}}\label{sec:finite-fs}

We divide the proof into four steps.

\smallskip\noindent\emph{Step 1: verify the boundary-path and endpoint conditions.}
Use the convex domain and $r_*$ chosen in
Section~\ref{sec:finite-data}. For fixed sufficiently small
$\eps>0$, the traces
$\varphi_\tau=g_{\eps,b_-+\tau(b_+-b_-)}$ form a continuous path
in $C^{2,\alpha}(\partial\Omega)$ with the required symmetries.
Indeed $D=0$ near the junctions with the rounded ends, so the two
formulas in \eqref{eq:barrel-trace} match smoothly. The leaves
depend smoothly on the positive phase $\eps(B_0+b)$. On the side,
$g_{\eps,b}=R+O(\eps)>0$ uniformly in $b$, and on the rounded
ends it is the positive leaf radius. Compactness gives
\[
 \min_{b\in[b_-,b_+]}\min_{\partial\Omega}g_{\eps,b}>0.
\]
These traces admit symmetric $C^{2,\alpha}(\overline\Omega)$
extensions depending continuously on $\tau$. The positive lower
bound may depend on the fixed $\eps$.

At the lower endpoint, no positive solution in the imposed
symmetry class (such solutions are understood to be of class
$C^{2,\alpha}(\overline\Omega)$, as are the fixed points of the
homotopy in \cite{FS}) exists for
$b\in[b_-,b_-+\delta/2]$ when $\eps$ is sufficiently small.
Otherwise take such solutions with $\eps_i\to0$ and $b_i\to b$
in this interval. Proposition~\ref{prop:harmonic-asymptotics}
gives $r^2f_i/\eps_i\to W+b$ on regular compact subsets.
Every radius graph contains $(0,u_i(0,0)\omega,0)$, and the phase
bound makes these points approach the origin.
Lemma~\ref{lem:signed-coordinate-contact} implies $m_0+b\ge0$,
contradicting $m_0+b\le-\delta/2$. The remaining hypotheses of
that lemma follow from the phase bound and
Proposition~\ref{prop:height-derivative-analysis}.

At the upper endpoint, and throughout
$b\in[b_+-\delta/4,b_+]$, we have $q+b\ge3\delta/4$.
The side phase is $\eps(q+b)+o(\eps)$ uniformly in this interval;
on the rounded ends it is exactly $\eps(B_0+b)\ge3\delta\eps/4$.
Thus $\min_{S(g_{\eps,b})}P\ge\delta\eps/2$ for small $\eps$, where
$S(g_{\eps,b})=\{(x,z,y):(x,y)\in\partial\Omega,\ |z|=g_{\eps,b}(x,y)\}$,
and the trace strictly dominates that of the positive leaf
$h_{\delta\eps/4}$. The phase minimum principle (see the proof
of Proposition~\ref{prop:moving-plane-phase-estimate}), together with
the monotonicity of $h_{\delta\eps/4}$, gives
\begin{equation}\label{eq:barrel-positive-endpoint}
 u|_{\partial\Omega}\ge h_{\delta\eps/4}(|x|)
 \quad\Longrightarrow\quad
 \inf_\Omega u\ge h_{\delta\eps/4}(0)=:\mu_\eps>0
\end{equation}
for every positive SME solution $u\in C^2(\Omega)\cap C^0(\overline\Omega)$,
symmetric or not, with these dominating boundary
values. Hence this leaf trace is strongly positive, and
$\varphi_1$ lies strictly above it, as required.

\smallskip\noindent\emph{Step 2: apply \cite{FS} and obtain items (1) and (2).}
We use Theorem~\ref{thm:fs-existence} with $n=5$ and $m=4$,
the mean-convex extension stated after \cite[Theorem~4.4]{FS},
and its proof in the closed space
\[
 \mathcal B=\{v\in C^{1,\alpha}(\overline\Omega):
       v(Qx,y)=v(x,y)=v(x,-y),\ Q\in O(4)\}.
\]
The domain, the boundary path, and the terminal constant in the
FS homotopy have these symmetries. For invariant input, each
linear Dirichlet problem in that homotopy has invariant coefficients
and data; uniqueness makes its solution invariant. Thus the
homotopy preserves $\mathcal B$, and the nonsolvability established
in this space is precisely the lower-endpoint condition needed
for the degree argument.

Theorem \ref{thm:fs-existence} now gives a parameter $b_\eps$ and a nonnegative
singular solution $u_\eps$, with
\begin{equation}\label{eq:barrel-singular-solution}
 u_\eps|_{\partial\Omega}=g_{\eps,b_\eps},\qquad
 \min_{\overline\Omega}u_\eps=0,\qquad
 \operatorname{dist}(\{u_\eps=0\},\partial\Omega)>0.
\end{equation}
We also retain the approximation supplied by their proof: there
are positive solutions $u_{\eps,j}\in\mathcal B$ and parameters
$b_{\eps,j}\in(b_-,b_+)$ such that
\begin{equation}\label{eq:barrel-fs-approximation}
 \begin{gathered}
 u_{\eps,j}|_{\partial\Omega}=g_{\eps,b_{\eps,j}},\qquad
 b_{\eps,j}\longrightarrow b_\eps,\\
 u_{\eps,j}\longrightarrow u_\eps
       \quad\text{uniformly on }\overline\Omega.
 \end{gathered}
\end{equation}
In particular $\min u_{\eps,j}\to0$. Step 1 excludes the lower
endpoint interval for every approximant and the upper endpoint
interval for all sufficiently large $j$, since there the minimum
is at least $\mu_\eps$. Consequently
\[
 b_\eps\in[b_-+\delta/2,b_+-\delta/4]\Subset(b_-,b_+).
\]
The symmetries pass to the limit. Applying
Proposition~\ref{prop:moving-plane-phase-estimate} to the
approximants and then using \eqref{eq:barrel-fs-approximation} gives
\[
 0\le y_1\le y_2\quad\Longrightarrow\quad
 u_\eps(x,y_1)\le u_\eps(x,y_2)
\]
whenever the two points belong to $\Omega$.
By \cite[Corollary~2.5]{FS}, $M_\eps=S(u_\eps)$ is stationary
with multiplicity one, is smooth where $u_\eps>0$, and has a
singularity at each point corresponding to a zero of $u_\eps$.
This proves items (1) and (2).

\smallskip\noindent\emph{Step 3: phase, compactness, and derivative bounds, items (3)--(5)}.
The phase estimate for the same approximants passes to the limit:
\[
 |P|\le K\eps\quad\text{on }M_\eps,
\]
with $K$ independent of $\eps$. The boundary traces have a common
upper bound for small $\eps$ and $b\in[b_-,b_+]$. The maximum
principle (at an interior maximum the left-hand side of
\eqref{eq:SME} would be nonpositive) bounds the positive approximants by the same constant,
and hence also bounds $u_\eps$. The interior gradient estimate
\cite[Theorem~8.1]{FS}, applied on balls compactly contained in
$\Omega$, therefore gives
\[
 \sup_{0<\eps<\eps_0}\|u_\eps\|_{C^{0,1}(U)}<\infty
       \qquad\text{for every }U\Subset\Omega.
\]
Every sequence $\eps_i\to0$ has a locally uniform subsequential
limit. By continuity of $P$, its radius graph lies in
$\{P=0\}=V$, so the limit must be $|x|$. Thus
$u_\eps\to|x|$ locally uniformly throughout $\Omega$.
On compact subsets with $|x|>0$, the height stays uniformly
positive, and elliptic regularity upgrades this to smooth
convergence.

The area density of the radius parametrization is
$u_\eps^3\sqrt{1+|Du_\eps|^2}$, apart from integration over
$S^3$. The uniform local Lipschitz bound dominates this density;
the parametrizations and their tangent planes converge away from
$\{x=0\}$, a set of zero base measure. Dominated convergence in
the area formula gives $M_\eps\to V$ as varifolds with multiplicity
one, with no additional mass on the spine. This proves item (4).

Proposition~\ref{prop:height-derivative-analysis}, applied to these
singular limits of solutions with the prescribed boundary data,
gives a normal graph on
$V\cap B^9_{2r_*}\cap\{r\ge C_*\eps^{1/3}\}$ and
\[
 |D^k f_\eps|\le C_k\eps r^{-2-k},\qquad k=0,1,2,
\]
for fixed sufficiently large $C_*$, with constants independent of
$\eps$. This proves item (5).

\smallskip\noindent\emph{Step 4: identify the boundary parameter and prove item (6)}.
The existence of an interior zero does not by itself identify
$b_\eps$. Along any sequence $\eps_i\to0$, take a subsequence
with $b_{\eps_i}\to b$. Proposition~\ref{prop:harmonic-asymptotics}
gives
\[
 \frac{r^2f_{\eps_i}}{\eps_i}\longrightarrow W+b
\]
on regular compact subsets. If $u_{\eps_i}(x_i,y_i)=0$, reflection
and axial monotonicity imply $u_{\eps_i}(x_i,0)=0$. The signed
phase identity at $z=0$ gives
\[
 c_H|x_i|^3\le K\eps_i,
\]
so these central contacts approach the origin.
Lemma~\ref{lem:signed-coordinate-contact} then implies $m_0+b\le0$.
On the other hand, $(0,u_{\eps_i}(0,0)\omega,0)$ always belongs
to the symmetric graph. Its radius is also $O(\eps_i^{1/3})$,
and the other conclusion of the same lemma gives $m_0+b\ge0$.
Therefore $b=-m_0$. Every subsequential limit has this value, so
\[
 b_\eps\longrightarrow-m_0,\qquad
 \frac{f_\eps}{\eps}\longrightarrow\frac{W-m_0}{r^2}
\]
smoothly on regular compact subsets, with normal
$\nu=(-x,z,0)/r$.
The same proposition gives the symmetries of $W$, $W_y>0$ for
$y>0$, and
\[
 W(r,y)-m_0=c\left(y^2-\frac{r^2}{3}\right)
           +O((r^2+y^2)^2),\qquad c>0.
\]
Analyticity supplies the differentiated Taylor bounds. This proves
item (6), and all conclusions of
Proposition~\ref{SME:prop:finite-family} follow. 
\qed

\section{Proposition~\ref{prop:SME-singularity-classification}: The classification of singularities}\label{sec:classification}
The previous section produces a singular solution with symmetry,
axial monotonicity, and a controlled leading Jacobi field. It does
not yet locate all of its zeros or identify its tangent cones at such zeros.
We first show that an off-axis zero would have density at least
$\theta_{13}>\theta_V$. Monotonicity at a fixed interior radius
and convergence of the family to $V$ exclude such zeros. Symmetry
and axial monotonicity then make the zero set a nonempty closed
interval on the axis.

The next step is a strict density drop at the origin. The field
$J_2$ has negative second variation on the spherical link of $V$,
but that link has singular poles. We estimate the area of the
omitted polar neighborhoods using the uniform Lipschitz bound and expand the
regular part of the link to second order. On a sufficiently small
fixed sphere, the leading negative term dominates the errors as
$\eps\to0$. Stationarity converts this link-area deficit into
a density strictly below $\theta_V$. A nontrivial interval of
axial zeros would force density at least $\theta_V$ in a tangent cone,
so the interval reduces to the origin. Only then do we use the
minimal-bigraph classification of \cite[Corollary~3.18]{FTW} to
identify every tangent cone as $C_{4,3}$.

\subsection{Ruling out off-axis zeros}In this subsection we rule out zeros away from the axis $\{x=0\}$. The main ingredients are the density comparison between $C_{1,3}$ and $V$ and the dimension bound for the zero set of SME limits in \cite[Theorem~9.1]{FS}. Recall from \eqref{eq:density-gaps} that
\[
 \theta_V=\Theta(V)=\frac{15\pi}{32}
 <\theta_{13}=\Theta(C_{1,3})=\frac{9\sqrt3\pi}{32}.
\]

\begin{lemma}[Density at an off-axis zero]
\label{lem:off-axis-density-extraction}
Let $u$ be an $O(4)_x$-invariant SME limit on a neighborhood of
a point $(x_0,y_0)\in\R^5$ with $x_0\ne0$ and $u(x_0,y_0)=0$, and put
$N=S(u)$ and $p=(x_0,0,y_0)$. Then
\[
 \Theta_N(p)\ge \Theta(C_{1,3})>\theta_{V}.
\]
Consequently, if $u$ is an $O(4)_x$-invariant SME limit on all
of $\R^5$ with $\Theta_N(\infty)\le\theta_V$ (in particular, if $N$ is
an SME cone with $\Theta(N)\le\theta_V$), then $u$ has no zeros off the
axis $\{x=0\}$.
\end{lemma}

\begin{proof}
Let $T$ be a tangent cone of $N$ at $p$, along scales $s_i\to 0$. We first show that the tangent directions of the
symmetry orbit through $p$ are directions of translation invariance
of $T$. The orbit of $p$ under rotations of \(x\) is
\[
\{(x,0,y_0):|x|=|x_0|\},
\]
a three-dimensional sphere. When we blow-up around \(p\), this sphere becomes flatter and converges to its tangent plane
\[
E=\{(w,0,0):w\cdot x_0=0\}\cong\mathbb R^3.
\]
Given $e\in E$, choose $Q_i\in O(4)$, acting on $x$, with
$Q_ix_0=x_0+s_ie+O(s_i^2)$ and $|Q_i-I|=O(s_i)$. The maps
$Z\mapsto s_i^{-1}(Q_i(p+s_iZ)-p)$ of $\R^9$ preserve $s_i^{-1}(N-p)$ and
converge locally uniformly to the translation $Z\mapsto Z+e$. Hence
the tangent cone satisfies
\[
T+e=T\qquad\text{for every }e\in E.
\]
Identifying $E$ with $\R^3$ and writing $X \in \R^2$ for the remaining base coordinates (the $x$ direction parallel to $x_0$ and the axial direction $y$), and noticing that the local SME Lipschitz bound \cite[Theorem~8.1]{FS} passes to
the rescaled profiles of $u$ at $p$ along the same scales, we deduce that $T$ is a homogeneous radius graph $U$ and 
\[
U (e, X) = v(X)\,,\qquad  T = \R^3 \times K\,,\qquad  K = S(v) = \{(X, z) \in \R^2 \times \R^4 : |z| = v(X)\}\,,
\]
where $v(\lambda X) = \lambda v(X)$ for $\lambda > 0$. The factor $K$ is therefore a five-dimensional stationary cone in $\R^6$.

If $v(X_0)=0$ at $X_0\ne0$, the zero set of the full tangent profile
contains $\R^3\times\{\tau X_0:\tau\ge0\}$, of dimension four.
This contradicts \cite[Theorem~9.1]{FS}, according to which the
zero set of an SME limit on a five-dimensional base has Hausdorff
dimension at most $5-2=3$. Hence $v>0$ away from zero.

The link $\Sigma=K\cap S^5$ is therefore a smooth compact,
embedded minimal hypersurface of $S^5$ which is rotational, that is,
invariant under the action of $O(4)$ on the $z$-factor. It cannot be a great sphere: an
$O(4)_z$-invariant hyperplane has its normal in $\R^2$ and cannot
project onto the whole two-dimensional base as $S(v)$ does.
By \cite[Theorem~1.1]{CWZ}, its area is either the great-sphere area,
the area of $S^1(1/2)\times S^3(\sqrt3/2)$, or greater than
$2(1-1/\pi)$ times the latter. The first alternative would make
$K$ a density-one cone; monotonicity centered at a regular point
would then force $K$ to be a plane, already excluded. Therefore
\[
 \Theta_N(p)=\Theta(T)=\Theta(K)\ge
 \Theta(C_{1,3})=\theta_{13}.
\]
For an entire $N$, monotonicity centered at $p$ gives
$\Theta_N(p)\le\Theta_N(\infty)$; the limit at infinity is independent
of the center by comparison of concentric balls of radii
$R-|p|$ and $R+|p|$. A cone has the same density at its vertex and at
infinity. Both stated consequences follow.
\end{proof}

\begin{corollary}[The zero set of the Dirichlet solution lies on the axis]
\label{cor:finite-axis-zero-set}
For all sufficiently small $\eps$, every zero of $u_\eps$ lies on
$x=0$. Its zero set is a nonempty closed axial interval
$\{(0,y):|y|\le\ell_\eps\}$; in particular $u_\eps(0,0)=0$.
\end{corollary}
\begin{proof}
If off-axis zeros existed for a sequence $\eps_i\to0$, axial
monotonicity would give central off-axis zeros
$p_i=(x_i,0,0)$ of $M_{\eps_i}$. The phase bound implies
$|x_i|=O(\eps_i^{1/3})$, so $p_i\to0$.
Choose a fixed $\rho>0$ so that $B_{2\rho}(0)$ lies in the common
interior domain. Lemma~\ref{lem:off-axis-density-extraction} and
monotonicity give
\[
 \theta_{13}\le\Theta(M_{\eps_i},p_i)
 \le\frac{\|M_{\eps_i}\|(B_\rho(p_i))}{\omega_8\rho^8}
 \longrightarrow\theta_V.
\]
For completeness, the limit follows by enclosing the moving ball
between $B_{\rho-\eta}(0)$ and $B_{\rho+\eta}(0)$, using
multiplicity-one convergence to $V$, and then sending $\eta\downarrow0$.
The limiting cone gives zero mass to the bounding spheres.
This contradicts $\theta_{13}>\theta_V$.
The zero set is nonempty by construction. Reflection, axial
monotonicity, continuity, and he positivity of $u_\eps$ near
$\partial\Omega$ (the zero set has positive distance from
$\partial\Omega$) now give
the asserted interval description. Points where the profile is
positive are smooth by elliptic regularity.
\end{proof}

\subsection{Decreasing area from the degree-zero Jacobi field}
\label{sec:link-area-deficit}

We turn the leading Jacobi field into a quantitative density
deficit. The argument has three parts: remove small neighborhoods
of the two singular poles and bound their area; compute the
second-order area change on the remaining smooth link; then use
stationarity to compare link area with the density ratio. The
cutoff radius is chosen so that the omitted polar neighborhoods and cutoff errors
are smaller than the negative quadratic contribution of $J_2$.
The required inputs are the uniform Lipschitz bound, phase trapping,
and the normalized expansion from the preceding section. Write $L=V\cap S^8$.
The link \(L\) has two singular poles. Its regular part is minimal in
\(S^8\), has dimension seven, and satisfies
\[
 |A_L|^2=\frac6{r^2},\qquad
 d\mu_L\approx r^6\,dr\,d\omega
 \quad\hbox{near either pole}.
\]
Here \(d\omega\) denotes the fixed angular measure on
\(S^3\times S^3\). Define its quadratic form by
\begin{equation}\label{link:eq:quadratic-form}
 \mathcal Q_L(\phi)
 =\int_L\left(|\nabla_L\phi|^2-(|A_L|^2+7)\phi^2\right)d\mu_L.
\end{equation}
The next lemma shows that the part of a spherical slice of $S(v)$,
for an SME limit $v$, lying near the spine has small area; this will
allow us to expand the area of the remaining part.

\begin{lemma}\label{link:lem:cap-area}
Let \(M=\{|z|=v(|x|,y)\}\) be the multiplicity-one symmetric graph
of a nonnegative function with Lipschitz constant at most \(\Lambda\),
defined on a neighborhood of the two points \((x,y)=(0,\pm1)\).
Then, for sufficiently small \(s\), depending only on \(\Lambda\),
\begin{equation}\label{link:eq:cap-area}
 \mathcal H^7\bigl(M\cap S^8\cap\{r<s\}\bigr)\le C_\Lambda s^7.
\end{equation}
The estimate includes graphs which are singular where \(v=0\).
\end{lemma}

\begin{proof}For small $s$, the sphere equation
$|x|^2+v(|x|,y)^2+y^2=1$ is uniformly transverse to the
$y$-direction near either pole: on the north polar region,
$y+vv_y\ge\sqrt{1-s^2}-\Lambda s\ge1/2$, and the south pole
has the opposite sign. Each polar part therefore admits a
Lipschitz parametrization
\[
 (x,\omega)\longmapsto
 (x,\widehat v(x)\omega,y(x)),\qquad
 x\in D\subset B_s^4,\quad\omega\in S^3,
\]
where $\widehat v(x)=v(|x|,y(x))$, $0\le\widehat v\le s$, and
the Lipschitz bounds for $y$ and $\widehat v$ depend only on
$\Lambda$. Its area Jacobian is bounded by
$C_\Lambda\widehat v^3$. Hence the area formula gives
\[
 \mathcal H^7(\text{polar region})
 \le C_\Lambda\int_D\widehat v^3\,dx
 \le C_\Lambda s^3|B_s^4|
 \le C_\Lambda s^7.
\]
Collapsed fibers at $\widehat v=0$ contribute no seven-dimensional
area. Summing over the two poles proves the claim.
\end{proof}

In the next proposition we expand the area up to second order and prove the relevant estimates.

\begin{proposition}[Second-order link expansion]
\label{link:prop:expansion}
Let $\eps_i\downarrow0$, and let $M_i=S(u_i)$, where the $u_i$ are
$O(4)_x$-invariant SME limits on a fixed neighborhood of the closed
unit ball of $\R^5$, with a uniform Lipschitz bound. Suppose that,
near $S^8$ and away from the spine, the $M_i$ are Euclidean normal
graphs over $V$ of functions $f_i$, and that
\begin{equation}\label{link:eq:off-axis-convergence}
 \eps_i^{-1}f_i\longrightarrow f
 \quad\hbox{smoothly on compact subsets of }V\setminus\{r=0\}.
\end{equation}
Suppose in addition that for \(r\ge C_*\eps_i^{1/3}\),
\begin{equation}\label{link:eq:graph-bounds}
 |D^k f_i|\le C_k\eps_i r^{-2-k},
 \qquad 0\le k\le2,
\end{equation}
on a fixed annular neighborhood of the unit sphere. Then
\begin{equation}\label{link:eq:area-expansion}
 \mathcal H^7(M_i\cap S^8)
 =\mathcal H^7(L)+\frac{\eps_i^2}{2}\mathcal Q_L(f|_L)
 +o(\eps_i^2).
\end{equation}
In particular, the integral defining the displayed quadratic form
is absolutely convergent.
\end{proposition}

\begin{proof}
Choose $s_i=\eps_i^\beta$ with $2/7<\beta<1/3$. Since
$\eps_i s_i^{-3}\to0$, the graph bounds give a spherical normal
height $\phi_i$ over $L\cap\{r\ge s_i\}$ satisfying
\begin{equation}\label{link:eq:spherical-bounds}
 |\phi_i|\le C\eps_i r^{-2},\qquad
 |\nabla_L\phi_i|\le C\eps_i r^{-3},\qquad
 \eps_i^{-1}\phi_i\longrightarrow f|_L
 \quad\hbox{in }C^1_{\mathrm{loc}}.
\end{equation}
These estimates follow by solving the sphere equation
$a^2+u_i(ap)^2=1$ for $a$ near $1$, with $p\in L$.
The omitted part of $M_i\cap S^8$ lies in $\{r<Cs_i\}$.
By Lemma~\ref{link:lem:cap-area}, it and the omitted part of $L$
have total area $O(s_i^7)=o(\eps_i^2)$.

The standard area expansion for a spherical normal graph over the
minimal link $L$ has quadratic term
$\frac12(|\nabla_L\phi_i|^2-(|A_L|^2+7)\phi_i^2)$ and a
pointwise remainder bounded by
\[
 C\bigl((1+|A_L|)|\phi_i|+|\nabla_L\phi_i|\bigr)^3
 \le C\eps_i^3r^{-9}.
\]
Since $d\mu_L\sim r^6\,dr\,d\omega$ near the poles, the
integrated remainder is $O(\eps_i^3s_i^{-2})=o(\eps_i^2)$. Thus
\[
 \mathcal H^7(M_i\cap S^8)-\mathcal H^7(L)
 =\frac12\int_{L\cap\{r\ge s_i\}}
 \left(|\nabla_L\phi_i|^2-(|A_L|^2+7)\phi_i^2\right)d\mu_L
 +o(\eps_i^2).
\]

Finally, after division by $\eps_i^2$ and extension by zero on
$\{r<s_i\}$, the absolute value of the quadratic integrand is
bounded by $Cr^{-6}$, which is integrable on $L$.
The local convergence in \eqref{link:eq:spherical-bounds} and
dominated convergence therefore give
$\mathcal Q_L(f|_L)$ as its integral limit. This also proves
absolute convergence of the quadratic form and yields
\eqref{link:eq:area-expansion}.
\end{proof}

We next show that $J_2$ is a direction of negative second
variation for the area of $L$, and that the same holds for the
restrictions to small spheres of Jacobi fields whose conjugate
$w=r^2f$ has leading term proportional to $y^2-r^2/3$.

\begin{lemma}
\label{link:lem:negative-mode}
For $J_2=\frac{y^2}{r^2}-\frac13$, one has
\begin{equation}\label{link:eq:negative-J2}
 \mathcal Q_L(J_2)=-7\int_L J_2^2\,d\mu_L<0.
\end{equation}
More generally, suppose
\[
 f(r,y)=r^{-2}w(r,y),
\]
where \(w\) is smooth across \(r=0\) as a rotationally invariant
function on \(\R^3\times\R\), and near the origin
\begin{equation}\label{link:eq:quadratic-phase}
 w(r,y)=c\left(y^2-\frac{r^2}{3}\right)
       +O\bigl((r^2+y^2)^{3/2}\bigr),\qquad c\ne0/,.
\end{equation}
Then, for all sufficiently
small fixed \(\rho>0\),
\[
 \mathcal Q_L\bigl(f(\rho\, \cdot)\bigr)<0.
\]
\end{lemma}

\begin{proof}
A direct calculation gives $(\Delta_L+|A_L|^2)J_2=0$.
Since $|J_2|\le Cr^{-2}$ and $|\nabla_LJ_2|\le Cr^{-3}$,
the boundary term in integration by parts over $L\cap\{r>d\}$
is $O(d)$. Letting $d\downarrow0$ therefore gives
\[
 \mathcal Q_L(J_2)
 =-\int_LJ_2(\Delta_L+|A_L|^2+7)J_2\,d\mu_L
 =-7\|J_2\|_{L^2(L)}^2<0.
\]

By the differentiated Taylor bounds in
\eqref{link:eq:quadratic-phase},
\[
 f(\rho p)=cJ_2(p)+E_\rho(p),\qquad
 |E_\rho|\le C\rho r^{-2},\quad
 |\nabla_LE_\rho|\le C\rho r^{-3}.
\]
Since $\int_Lr^{-6}\,d\mu_L<\infty$, expanding the quadratic form
and applying these bounds yields
\[
 \mathcal Q_L\bigl( f(\rho \,\cdot)\bigr)
 =c^2\mathcal Q_L(J_2)+O(\rho)<0
\]
for all sufficiently small $\rho>0$.

\end{proof}

We can now show that the density drops below $\theta_V$.

\begin{proposition}[Strict density drop]\label{link:prop:density}
Assume that the \(M_i\) in Proposition~\ref{link:prop:expansion}
are stationary in a neighborhood of the closed unit ball.
If \(\mathcal Q_L(f|_L)<0\), then for all sufficiently large \(i\),
\[
 \Theta(M_i,0)<\Theta(V,0).
\]
The same conclusion holds if the preceding hypotheses hold, after
a fixed dilation, on a fixed neighborhood of the origin, and the
conjugate $w=r^2f$ of the limit Jacobi field satisfies
\eqref{link:eq:quadratic-phase}; in this case the link is evaluated at
a sufficiently small fixed radius.
\end{proposition}

\begin{proof}
For a compactly supported radial cutoff \(\eta\), stationarity gives
\[
 \begin{aligned}
 0&=\delta M_i(\eta(|X|)X)\\
 &=\int_{M_i}\left(8\eta(|X|)
       +|X|\eta'(|X|)\,|\nabla_{M_i}|X||^2\right)d\|M_i\|.
 \end{aligned}
\]
Approximating the characteristic function of \(B_\rho\) and applying
coarea gives, for almost every \(\rho\),
\begin{equation}\label{link:eq:radial-flux}
 8\,\|M_i\|(B_\rho)
 =\rho\int_{M_i\cap\partial B_\rho}
       |\nabla_{M_i}|X||\,d\mathcal H^7
 \le \rho\,\mathcal H^7(M_i\cap\partial B_\rho).
\end{equation}
The same first-variation calculation with the usual radial powers
gives Euclidean monotonicity in the form
\[
 \frac{\|M_i\|(B_\rho)}{\rho^8}
 -\frac{\|M_i\|(B_\sigma)}{\sigma^8}
 =\int_{M_i\cap(B_\rho\setminus B_\sigma)}
       \frac{|X^\perp|^2}{|X|^{10}}\,d\|M_i\|
 \ge0,\qquad 0<\sigma<\rho.
\]
Both calculations are varifold first-variation identities and remain
valid across the singular set. Therefore
\[
 \Theta(M_i,0)
 \le\frac{\|M_i\|(B_\rho)}{\omega_8\rho^8}
 \le\frac{\mathcal H^7(M_i\cap\partial B_\rho)}
          {8\omega_8\rho^7}.
\]
At unit radius, whenever the flux identity holds,
\[
 \begin{aligned}
 \Theta(M_i,0)
 &\le\frac{\mathcal H^7(L)}{8\omega_8}
       +\frac{\eps_i^2}{16\omega_8}\mathcal Q_L(f|_L)
       +o(\eps_i^2)\\
 &<\frac{\mathcal H^7(L)}{8\omega_8}=\Theta(V,0).
 \end{aligned}
\]
If necessary choose a nearby common admissible radius. Indeed, if
\(E_i\) is the exceptional set for \(M_i\), then
\[
 |\cup_iE_i|=0,\qquad
 \mathcal Q_L(f(\rho\,\cdot))<0
 \quad\text{for all }\rho\text{ sufficiently close to }1.
\]
The second statement follows by the weighted continuity already used
in the link expansion.

At a general fixed radius \(\rho\), set
\(F_\rho(p)=f(\rho p)\), for \(p\in L\).
Rescaling Proposition~\ref{link:prop:expansion} gives
\[
 \mathcal H^7(M_i\cap\partial B_\rho)
 =\rho^7\mathcal H^7(L)
  +\frac{\eps_i^2\rho^5}{2}\mathcal Q_L(F_\rho)
  +o(\eps_i^2).
\]
Lemma~\ref{link:lem:negative-mode} makes this coefficient negative
at every sufficiently small fixed radius under
\eqref{link:eq:quadratic-phase}. Choose a common admissible one
and apply \eqref{link:eq:radial-flux}.
\end{proof}

\subsection{Density bounds and classification of cone limits}

The next lemma shows that a symmetric SME cone with a zero on the
axis away from its vertex has density at least $\theta_V$, with
equality only for $V$.

\begin{lemma}[An axial zero of a cone forces cylindrical density]
\label{lem:axial-zero-density-extraction}
Let $K=S(u)$ be an SME cone, with $u$ $O(4)_x$-invariant. If
$u(0,1)=0$, that is, if $(0,0,1)\in\Sing K$, then
$\Theta_K(0)\geq\theta_V$.
If also $\Theta_K(0)\leq\theta_V$, then $K=V$.
\end{lemma}
\begin{proof}
Put $q=(0,0,1)$.  A tangent cone $T$ to a cone at a nonvertex point
splits its radial direction; thus $T=T_0\times\R_y$.
This follows directly by conjugating the dilations fixing the vertex
by the translation and dilation used to blow up at $q$; in the limit
they give every translation in the direction $q$.
The SME compactness and graph property are preserved in this limit.
The transverse profile is $O(4)_x$-invariant and homogeneous of degree
one, hence equals $c|x|$ for a constant $c\geq0$.
The noncollapse consequence of \cite[Lemma~2.3 and Remark~2.4]{FS}
excludes $c=0$: every base ball of radius $\rho$ contains a point
where the height is at least $c_*\rho$. Substituting the radial profile
$c|x|$ in the four-dimensional-base equation gives
\[
 \frac{d^2(ct)/dt^2}{1+c^2}+\frac3t\frac{d(ct)}{dt}
       -\frac3{ct}=\frac3t(c-c^{-1})=0
 \ \Longrightarrow\ c=1,
 \qquad T=V.
\]
Thus monotonicity centered at $q$ gives
\[
 \theta_V=\Theta_K(q)
 \leq\Theta_K(q,R)
 \longrightarrow\Theta_K(0)\qquad(R\to\infty).
\]
If $\Theta_K(0)=\theta_V$, then for $0<s<R$ the exact
monotonicity identity yields that  $K$ is a cone centered at $q$ as well and it equals its tangent cone $T=V$ there.
\end{proof}

We next classify symmetric, axially monotone SME cones with
density at most $\theta_V$.

\begin{lemma}[Cone classification with axial monotonicity]
\label{lem:monotone-cone-classification}
Let $K=S(u)$ be an SME cone with $\Theta(K)\le\theta_V$, where
$u(x,y)=u(|x|,y)$ is $O(4)_x$-invariant.
Suppose $u$ is even in $y$ and nondecreasing in $y\geq0$.
Then $K$ is either $C_{4,3}$ or $V$.
\end{lemma}
\begin{proof}
Lemma~\ref{lem:off-axis-density-extraction} excludes off-axis zeros.
By homogeneity, either $u(0,1)=0$, in which case
Lemma~\ref{lem:axial-zero-density-extraction} gives $K=V$, or
$u(0,1)>0$.  In the latter case $u>0$ on $\R^5\setminus\{0\}$,
so the cone has an isolated singularity and a smooth link.

On the connected region $y>0$, the function $u_y$ satisfies the
linear equation obtained by differentiating the SME.  It is
nonnegative by hypothesis and is positive on the positive $y$ axis,
where homogeneity gives $u(0,y)=y u(0,1)$.
The local Harnack inequality for this nonnegative solution gives
\[
 u_y(0,y)=u(0,1)>0\quad(y>0)
 \ \Longrightarrow\ u_y>0\text{ on }\{y>0\}.
\]
For each $a=|x|>0$, homogeneity gives
\[
 u(a,0)=c_0a,
 \qquad
 \frac{u(a,y)}y=u(a/y,1)\longrightarrow u(0,1)>0.
\]
Thus $y\mapsto u(a,y)$ maps $(0,\infty)$ bijectively onto
$(c_0a,\infty)$. Its inverse $G(a,b)$ satisfies
\[
 \begin{aligned}
 K\cap\{y>0\}&=\{y=G(|x|,|z|):|z|>c_0|x|\},\\
 K\cap\{y<0\}&=\{y=-G(|x|,|z|):|z|>c_0|x|\},\\
 G(a,c_0a)&=0,\qquad u_y(a,0)=0.
 \end{aligned}
\]
Smoothness along $y=0$ therefore makes the two graphs meet that
hyperplane orthogonally.

These verify the precise hypotheses of \cite[Corollary~3.18]{FTW}:
a cohomogeneity-two minimal cone, equivariant under
$O(n-k)\times O(k)\times\mathbb Z_2$, expressible as a bigraph.
Set $n=8$, $k=4$, and identify their two four-dimensional factors with
$z$ and $x$. Their conclusion is the Lawson cone
\[
 C(S^3\times S^4)
 =\left\{|z|=\frac{\sqrt3}{2}\sqrt{|x|^2+y^2}\right\}
 =C_{4,3}.
\]
The relevant ODE endpoint is also explicit in
\cite[Proposition~3.17]{FTW}: the zero-height, infinite-slope endpoint
occurs only at $a_{8,4}=\sqrt{4/3}$.
\end{proof}

We can finally prove Proposition~\ref{prop:SME-singularity-classification}.

\begin{proof}[Proof of Proposition~\ref{prop:SME-singularity-classification}]
Let $M_\eps$ be the family in
Proposition~\ref{SME:prop:finite-family}, with $0<\eps<\eps_0$. Its normalized limit is
$f=r^{-2}w$, with $w=W-m_0$, where
\[
 w=c(y^2-r^2/3)+O(\rho^4),\qquad c>0.
\]
Lemma~\ref{link:lem:negative-mode} gives, on a sufficiently small fixed
sphere of radius $s_*>0$,
\[
 \mathcal Q_L\bigl(f(s_* \, \cdot)\bigr)<0.
\]
The phase, Lipschitz and normal derivative estimates from the first
proposition verify the hypotheses of
Propositions~\ref{link:prop:expansion} and~\ref{link:prop:density}
after this fixed dilation. Consequently, arguing by contradiction
along sequences $\eps_i\to0$, we obtain $\eps_1\in(0,\eps_0]$ such
that, for $0<\eps<\eps_1$,
\[
 \Theta(M_\eps,0)<\theta_V.
\]
Decreasing $\eps_1$ if necessary, we may also assume that
Corollary~\ref{cor:finite-axis-zero-set} applies for $0<\eps<\eps_1$.

For such a fixed $\eps$, suppose that the half-length
$\ell_\eps$ of its zero interval were positive. For any sequence
$s_j\downarrow0$, the rescaled profiles
$s_j^{-1}u_\eps(s_jx,s_jy)$ vanish at $(x,y)=(0,1)$ once
$s_j<\ell_\eps$. The local Lipschitz bound and SME compactness
therefore give a homogeneous tangent profile with this same axial zero.
If $K$ is its symmetric cone, Lemma~\ref{lem:axial-zero-density-extraction}
gives
\[
 \theta_V\le\Theta(K)=\Theta(M_\eps,0)<\theta_V,
\]
a contradiction. Thus $\ell_\eps=0$. Note that this argument
uses no classification of the isolated cone: the density obstruction
alone rules out the interval.

Every tangent cone at zero inherits the symmetries and axial order and
has density less than $\theta_V$. The remaining classification,
Lemma~\ref{lem:monotone-cone-classification}, gives
\[
 \operatorname{Tan}(M_\eps,0)=\{C_{4,3}\},\qquad
 \Theta(M_\eps,0)=\theta_{43}.
\]
Positive profile points are smooth by SME regularity, including the
smooth radial axis where the profile is positive. The only zero is the
origin, so it is the only singular point of the graph over $\Omega$.
This proves all assertions of the proposition.
\end{proof}

\section{Quantitative estimates for the connector}\label{sec:connector-estimates}\label{sec:end}

The estimates needed for the gluing construction of the next sections concern the error between the connector
and the Hardt--Simon leaf with the same smoothing radius $a(y):=u(0,y)$. Write
\[
 u=u_\infty,\qquad a(y)=u(0,y),\qquad
 v(t,y)=a(y)h(t/a(y)),\qquad e=u-v,
\]
where the Hardt-Simon's foliation $h$ and the minimal-surface operator $Q$ are
as in the previous sections. The required bounds, on an axial band of size
$R\approx |y|$, are
\[
 |\partial_t^m\partial_y^n e|
 \le C_{mn}a^3R^{-2-n}(a+t)^{-m}.
\]
Subtracting the equations for $u$ and $v$ gives a linear equation
for $e$ with right-hand side $-Q(v)$. Its forcing is determined by
the derivatives of $a$. Indeed, with $s=t/a$ and
$\zeta(s)=h(s)-sh'(s)$, a straightforward computation yields
\[
 v_y=a'\zeta(s),\qquad
 v_{yy}=a''\zeta(s)+\frac{(a')^2}{a}s^2h''(s),
\]
which combined with the leaf equation for $h$ gives
\begin{equation}\label{eq:end-modulated-defect}
 |Q(v)|\le C\left(|a''|+\frac{(a')^2}{a}\right)(1+t/a)^{-2}.
\end{equation}
Thus the desired forcing scale is available once the bound
$|a^{(j)}|\le C_jaR^{-j}$ is known. Differentiating the forcing
requires the corresponding higher derivatives of $a$.
At the axis, the exact identity
\[
 4e_{tt}(0,y)+\frac{a''(y)}{1+(a'(y))^2}=0
\]
already shows that the remainder estimate and the smoothing-radius
derivative estimate cannot be separated by treating $a''$ as known data.

We obtain those derivatives by differentiating the equation for $u$
itself. The function $u_y$ solves the Jacobi equation, whose
inverse-square potential prevents ordinary estimates from being
uniform between the smoothing scale and the axial scale. Dividing by the
positive dilation Jacobi field of the local leaf absorbs the
fixed-leaf potential. The principal coefficients of the resulting
equation are comparable to
\[
 \omega_\beta(x,y)=(\beta^2+|x|^2)^{-1/2},
 \qquad \beta=u(0,y_0),
\]
where $\beta$ is held fixed on a band about $y_0$.
This is a Muckenhoupt $A_2$ weight with constant independent of
$\beta$. Weighted local estimates then control the normalized
axial derivatives at scale $R$.

To apply that theory, we must first identify the smoothing scale and
compare the actual Jacobi coefficients with the displayed weight
uniformly throughout the band. The quadratic growth estimate,
together with the positive quadratic coefficient and comparison near the axis, gives $\beta^3\approx |y_0|^2$, a positive range of
comparable leaves, and a quantitative approximation by those leaves.
The approximation also controls the residual Jacobi potential.
The $A_2$ property itself holds for every $\beta$; the growth and
comparison estimates are what place this particular surface in a
uniform weighted coefficient class. We therefore prove the quadratic
estimate first, then estimate the axial derivatives of $u$, and
finally return to the forced radial equation for $e$.

Throughout the section we consider the invariant hypersurfaces
\[
 \Sigma_v=\{|z|=v(t,y)\}.
\]
Recall that the Euclidean minimal surface equation for a positive \(v\) is
\begin{equation}\label{cap:eq:Q}
 Q(v):=
 \left(\delta_{ij}-\frac{v_i v_j}{1+|Dv|^2}\right)v_{ij}
 +\frac3t v_t-\frac3v=0,\qquad i,j\in\{t,y\}.
\end{equation}
For the normal pointing towards increasing \(|z|\), the scalar mean
curvature, with convention \(H=\operatorname{div}_{\Sigma_v}\nu\), is
\[
 H=-Q(v)(1+|Dv|^2)^{-1/2}.
\]

Let \(|z|=h(|x|)\) be the invariant Hardt-Simon leaf on the
\(|z|>|x|\) side, normalized by \(h(0)=1\). The external fixed-leaf
input is the global parametrization in
\cite[\S3.1, proof of Proposition~3.3]{Sz} and its asymptotic expansion
\cite[Proposition 3.3]{Sz}. In our vertical graph coordinates these give
\begin{equation}\label{cap:eq:HS}
 h'(0)=0,\qquad
 \frac{h''}{1+h'^2}+\frac3t h'-\frac3h=0,\qquad
 h(t)=t+b_H t^{-2}+O(t^{-3}),\quad b_H>0.
\end{equation}
together with the corresponding expansions of the derivatives of
$h$. In the notation of Section~\ref{sec:connector},
$\{|z|=h(|x|)\}=H_{c_H}$, and $0<h'<1$ on $(0,\infty)$; see
the proof of Lemma~\ref{lem:linearized-hs-inverse}.
By rotational invariance, write $u_\infty(t,y)=U(x,y)$ when $|x|=t$;
within this section we also abbreviate this profile by $u$. Our goal is the following:
\begin{proposition}[The sharp connector end]\label{prop:precise-estimates-for-connector}
Let \(M=S(u_\infty)\) be the connector of
Theorem~\ref{thm:main-connector}, and put
\[
 a_\infty(Y)=u_\infty(0,Y)\qquad(Y\ne0).
\]
There are constants \(\kappa_H,c>0\) and \(Y_*>0\) such that
\[
 a_\infty(Y)\sim\kappa_H|Y|^{2/3}\quad(|Y|\to\infty),\qquad
 |a_\infty^{(j)}(Y)|
 \le C_j a_\infty(Y)|Y|^{-j}\quad(j\ge1,\ |Y|\ge Y_*).
\]
For \(|Y|\ge Y_*\) and \(0\le T\le c|Y|\),
\begin{equation}\label{roadmap:eq:sharp-end}
 \begin{split}
 u_\infty(T,Y)
 &=a_\infty(Y)h(T/a_\infty(Y))+e_\infty(T,Y),\\
 |\partial_T^m\partial_Y^ne_\infty(T,Y)|
 &\le C_{mn}a_\infty(Y)^3|Y|^{-2-n}
                   (a_\infty(Y)+T)^{-m},
 \qquad m,n\ge0.
 \end{split}
\end{equation}
On every fixed regular angular sector
\(\{r\ge\theta\rho\}\) of \(V\), where
\(r=\sqrt2t\) and \(\rho=(r^2+Y^2)^{1/2}\), the normal graph
height \(f_\infty\) of \(M\) satisfies
\[
 |D^jf_\infty|\le C_{\theta,j}\rho^{-j}\qquad(j\ge0).
\]
Every derivative order is fixed but arbitrary. The constants
$C_j,C_{mn},C_{\theta,j}$ may depend on those orders, while $c$ and
$Y_*$ can be chosen independently of them.
\end{proposition}

The first subsection proves the phase bound and the positivity
of its leading coefficient. The second assembles the comparison in the smoothing region,
direct derivative estimates, and radial remainder bound from
Appendix~\ref{app:connector-derivatives} to prove the proposition.

\subsection{The quadratic phase bound}
The cylindrical blow-down gives only subcubic phase growth.
We first improve this to growth of order $R^{2+\eta}$ for every
$\eta>0$, using compactness and the harmonic description of Jacobi
fields. Harmonic replacements on annuli then isolate the quadratic
mode, while nonconcentration controls the part of the surface near
the spine. We must also prove that its coefficient is positive:
if it vanished, iterating the improved growth estimate would give
bounded phase. Transferring that bound and its asymptotic value
through the smoothing regions would force an interior contact with a foliation
leaf, contradicting the connector's change of side near zero.
The positive quadratic mode is what determines the smoothing scale used
in the derivative estimates. Define:
\begin{align}
    D(R) = \sup_{M\cap B_R}|P|.
\end{align}
Here $P$ is the Hardt--Simon phase of Section~\ref{sec:connector};
it is independent of $y$ and homogeneous of degree three:
\begin{align}
    P(sx,sz,y) = s^3P(x,z,y)\qquad s>0.
\end{align}
Note that $D(R)>0$ for every $R>0$, since $M$ meets both sides of
$V$ near the origin. We aim to prove:
\begin{proposition}[Quadratic bound]\label{prop:quadratic-bound}
    For the connector $S(u)=M$ of Theorem~\ref{thm:main-connector}, we have:
    \begin{align}
        D(R) \leq CR^2,
    \end{align}
    for some $C>0$ and all $R\ge1$.
Moreover there is $A>0$ such that, on every fixed regular sector
$r\ge\theta\rho$, the normal height $f$ of $M$ over $V$ satisfies
\[
 |D^k(f-AJ_2)|\le C_{k,\theta}\rho^{-7/33-k}
                  \qquad(k\ge0)
\]
for sufficiently large $\rho$.
\end{proposition}

We first establish the following preliminary bound:
\begin{lemma}[Super quadratic bound]\label{prop:super-quadratic-growth}
    For any $\eta>0$, there is $C_\eta>0$ such that:
    \begin{align}\label{eq:preliminary-superquadratic-estimate}
        D(R) \leq C_\eta R^{2+\eta}\quad \text{for }R\geq 1.
    \end{align}
\end{lemma}

\begin{proof}

   For the proof we use the nonconcentration estimates of
\cite[Proposition~5.6]{Sz} and a linearization argument, which
produces Jacobi fields in the limit. 

\medskip

\noindent\textbf{Step 1: contradiction sequence.} In the notation of
\cite[Definition~5.2(a)]{Sz}, for a fixed value of its parameter
$\beta>0$, the distance $D_{\beta,V_0}(M;B_R)$ of $M$ from $V$ in
$B_R$ is comparable to $D(R)$, the maximum of the Hardt--Simon phase.
Since $R^{-1}M\to V$ as $R\to\infty$, we have
$R^{-3}D(R)=\sup_{R^{-1}M\cap B_1}|P|\to0$. It suffices to prove the
lemma for $0<\eta<1$, and since $D$ is nondecreasing it suffices to
consider the radii $R_k=16^k$, $k\ge0$. Put
$E_k=R_k^{-(2+\eta)}D(R_k)$, and suppose that $\sup_kE_k=\infty$. We
claim that there are indices $k_1<k_2<\cdots$ such that

    \begin{align}\label{eq:record-indices}
        \begin{cases}E_{k_j} = R_{k_j}^{-(2+\eta)}D(R_{k_j})\to\infty,\\  D(16R_{k_j}) \leq 16^{3+\delta} D(R_{k_j}),&\text{for all } j\in \mathbb{N}\\
        E_{k_j}\geq E_k,& \text{for all }k\leq k_j.
        \end{cases}
    \end{align}
To see this, call $k$ a \emph{record index} if $E_k\ge E_{k'}$ for all $k'\le k$.
Since $(E_k)$ is unbounded, there are infinitely many record indices,
and $E_k\to\infty$ along them. If $k$ is a record index at which the
second condition fails, then
\[
 E_{k+1}=16^{-(2+\eta)}R_k^{-(2+\eta)}D(16R_k)>16^{1}E_k,
\]
so that $k+1$ is again a record index. The second condition cannot
fail at all the indices $k,k+1,\dots,k+\ell$ for every $\ell$:
otherwise $D(16^\ell R_k)>16^{(3+\eta)\ell}D(R_k)$ for every $\ell$,
which, since $D(R_k)>0$, contradicts $R^{-3}D(R)\to0$. Replacing each
record index by the first subsequent index at which the second
condition holds, which is again a record index with a larger value of
$E$, and passing to a subsequence, we obtain the claim. 

The first and third conditions now give
\begin{equation}\label{eq:end-record-inner}
 D(sR_{k_j})\le16^{2+\eta}s^{2+\eta}D(R_{k_j})
 \qquad\text{for }0<s\le1\text{ with }sR_{k_j}\ge1.
\end{equation}
Indeed, if $k\le k_j$ is the least index with $sR_{k_j}\le R_k$, then
$R_k\le16sR_{k_j}$ and
$D(sR_{k_j})\le D(R_k)=E_kR_k^{2+\eta}\le E_{k_j}R_k^{2+\eta}
=(R_k/R_{k_j})^{2+\eta}D(R_{k_j})$.
Relabel the selected radii as $R_j$. 

\medskip

\noindent\textbf{Step 2: Jacobi field.} Now let $f_j$ be the normal height of $R_j^{-1}M$, and set
\[
 \eps_j=\frac{D(R_j)}{R_j^3},\qquad
 v_j=\eps_j^{-1}f_j.
\]
By the second condition in \eqref{eq:record-indices},
$|P|\le16^{3+\eta}\eps_j$ on $R_j^{-1}M\cap B_{16}$. The outer bound and interior graph estimates yield
\[
 v_j\longrightarrow v\quad\text{in }C^\infty_{\mathrm{loc}}
      (\operatorname{Reg}\Cyl\cap B_8),\qquad
 L_{\Cyl}v=0,
\]
where
\[
 \begin{aligned}
 L_{\Cyl}
 &=\partial_r^2+\frac6r\partial_r+\partial_y^2+\frac6{r^2},\\
 L_{\Cyl}(r^{-2}w)
 &=r^{-2}\left(w_{rr}+\frac2r w_r+w_{yy}\right)
 =r^{-2}\Delta_{\R^3_\xi\times\R_y}w.
 \end{aligned}
\]
Thus $w=r^2v$ is harmonic off the auxiliary axis and bounded near
it. 

Nonconcentration prevents the normalized limit $w$ from vanishing.
Indeed, the nonconcentration estimate \cite[Proposition~5.6]{Sz},
with its parameter $\delta$ equal to zero, together with the maximum
principle for $|P|$, gives for every $\gamma>0$
\[
 \Delta w=0\text{ in }B_8,\qquad
 |w(X)|\le C|X|^{2+\eta}\text{ near }0.
\]
We aim to show that non-concentration prevents the normalized limit from vanishing. Notice the the non-concentration estimates of \cite[Proposition 5.6]{Sz} with the choice $\delta=0$ and the maximum principle for $|P|$ asserts that for every $\gamma>0$ we have:
\begin{equation}\label{eq:end-nonconcentration}
 D(R)\le C_\gamma
 \sup_{\substack{M\text{ represented over }\Cyl\\
                   R/2<\rho<2R,\ r>\theta_\gamma R}}
       |r^2 f|+C_0\gamma D(4R),
\end{equation}
where $f$ is the normal graph of $M$ over $C_{3,3}\times \R = V$. Here $R$ is sufficiently large, $\theta_\gamma>0$ and the threshold
may depend on $\gamma$, and $C_0$ is independent of $\gamma$. Choose $\gamma>0$ so that
\[
C_0\gamma\,16^{3+\eta}<\frac12.
\]
Apply \eqref{eq:end-nonconcentration} at $R_j$ and divide by
$D(R_j)$. Since $D(4R_j)\le D(16R_j)$, this with the second condition in \eqref{eq:record-indices} gives
\[
1\le C_\gamma
 \sup_{\substack{1/2<\rho<2\\r>\theta_\gamma}}|r^2v_j|
 +\frac12,
\qquad
\sup |r^2v|\ge(2C_\gamma)^{-1}>0.
\]

Finally, $w_y\ge0$ for $y>0$. To check this in normal coordinates,
write $t=(r-f)/\sqrt2$ and $u(t,y)=(r+f)/\sqrt2$. Then
\[
 f_y=\frac{\sqrt2u_y}{1+u_t}\ge0
\]
where the normal graph is small; the inequality passes to the limit.

\medskip

\noindent\textbf{Step 3: final contradiction.} Evenness and the strong maximum principle now give the alternatives
\[
 \begin{aligned}
 w_y\not\equiv0
 &\Longrightarrow w_y>0\text{ in }B_8\cap\{y>0\},\
                   w_y=0\text{ on }\{y=0\}\\
 &\Longrightarrow w_{yy}(0)>0
 \quad\text{(Hopf lemma)},\\
 w_y\equiv0
 &\Longrightarrow w=c_1+c_2/r
 \Longrightarrow w\equiv0
 \quad\text{(boundedness and }w(0)=0\text{)}.
 \end{aligned}
\]
The first contradicts $w=O(\rho^{2+\eta})$; the second contradicts
nontriviality and \eqref{eq:preliminary-superquadratic-estimate} follows.
\end{proof}

We next improve this bound to a quadratic one by linearization.
The main difficulty is that one cannot linearize near the spine; the
missing ingredient is the nonconcentration estimate
\cite[Proposition~5.6]{Sz}, which allows the linearization to go
through. We also need the following transfer lemma, which relies on
the barriers recalled in the proof of
Lemma~\ref{lem:signed-coordinate-contact}.

\begin{lemma}[Uniform phase convergence up to the spine]
\label{lem:bounded-phase-transfer}
Let $M$ be the connector and $P$ its signed Hardt--Simon phase.
Assume that $|P|\le K$ on $M\setminus B_{R_0}$ for some
$K,R_0>0$. Suppose there is a constant $b\in\R$ such that,
for every fixed $0<\theta<1$,
\[
 \lim_{R\to\infty}
 \sup_{M\cap\{\rho\ge R,\ r\ge\theta\rho\}}|P-b|=0.
\]
Then
\[
 \lim_{R\to\infty}\sup_{M\cap\{\rho\ge R\}}|P-b|=0.
\]
Thus the same convergence holds uniformly without any restriction
on the distance to the spine.
\end{lemma}

\begin{proof}
Fix $0<\delta<1$, $I=(2/3,4/3)$, and $I'=[3/4,5/4]$.
Choose $\chi\in C^\infty(I;[0,1])$ equal to zero on $I'$ and one
near both ends of $I$. With real ninth roots of the constants,
put
\[
 q_+=(b+\delta)^{1/9}+B\chi,\qquad
 q_-=(b-\delta)^{1/9}-B\chi.
\]
Choose $B$ so large that $q_+^9>K+2$ and $q_-^9<-K-2$ near
the ends. Choose $s_0$ so large that
$(q_++s_0)^9>K+2$ and $(q_--s_0)^9<-K-2$ throughout $I$.
These choices give uniformly $C^2$-bounded signed families for
$0\le s\le s_0$; they do not depend on the dilation below.

Set $M_R=R^{-1}M$ and $\eps=R^{-3}$. Cubic homogeneity gives
$|P|\le K\eps$ on $M_R$ in this slab for large $R$.
Let $X_\eps(q)$ denote the negative-side comparison barrier in
the proof of Lemma~\ref{lem:signed-coordinate-contact}, and let
$\mathcal S(x,z,y)=(z,x,y)$. Use the upper and lower barriers
\[
 X_\eps(q_++s),\qquad
 \mathcal S X_\eps(-q_-+s),\qquad 0\le s\le s_0.
\]
The sign in the second expression is necessary: after exchanging
the factors, its phase is within $\eps\delta/4$ of
$\eps(q_--s)^9$, and it excludes contact from its positive side.
By \eqref{eq:signed-barrier-trapping}, first choose a common fixed
lateral radius $r_0>0$ giving phase error at most $\eps\delta/4$
for both families. Only then take $R$ large enough that the barriers
exist and
\[
 |P/\eps-b|<\delta/4\quad
       \text{on }M_R\cap\{r=r_0,\ y\in I\}.
\]
The latter follows from the assumed regular-sector convergence,
since $r_0$ is fixed and the rescaled axial coordinate stays away
from zero.

At $s=s_0$, phase trapping puts the whole surface between the
barriers. The phase bound preserves strict ordering at the axial
ends for every $s$. At the lateral boundary, the upper barrier has
phase at least $\eps(b+3\delta/4)$ and the lower at most
$\eps(b-3\delta/4)$, whereas the surface has phase within
$\eps\delta/4$ of $b\eps$. Slide $s$ down to zero. Boundary
contact is excluded by these inequalities; interior contact,
including at a barrier spine point, is excluded by
\cite[Proposition~5.8]{Sz}. At $s=0$ and $y\in I'$, trapping
between the two barriers gives
\[
 |P/\eps-b|\le5\delta/4\quad\text{on }M_R\cap\{r<r_0,\ y\in I'\}.
\]
For any physical point with large $y>0$ and $r<r_0y$, apply this
estimate with $R=y$; its rescaled axial coordinate is one.
Reflection handles $y<0$. The complement lies in a fixed regular
angular sector, where the original convergence applies.
Letting $\delta\downarrow0$ proves uniform convergence on the end.
This argument uses neither a positive quadratic coefficient nor a
positive lower bound for the phase.
\end{proof}

We can now prove the quadratic decay. The combination of growth bounds and interior elliptic estimates used in the proof is also familiar from optimal regularity, see for instance \cite[\S4, Lemma~4.2 and Theorem~4.1]{FigalliObstacle}.

\begin{proof}[Proof of Proposition~\ref{prop:quadratic-bound}]
    The idea is to replace the regular normal graph by harmonic functions on
full auxiliary annuli and control the error caused by the removed tube. 
Their expansion identifies the quadratic term. If its coefficient
vanished, iteration of the replacement and nonconcentration estimates
would give bounded phase. Lemma~\ref{lem:bounded-phase-transfer} would
then give constant phase at infinity, contradicting the two-sided
geometry of the connector by the strong maximum principle.

    \medskip
   \noindent\textbf{Step 1: harmonic replacement in annuli.} Let $f$ be the normal height of $M$ over $V$ on the regular graph region in $\{\rho<8R\}$, and define
    \begin{align}
        d = D(8R),\qquad \eps_R = d/R^3,\qquad f_R(X) = R^{-1}f(RX),
    \end{align}
    Then $f_R$ is the normal height of $R^{-1}M$ over $V$, and $|P|\le\eps_R$ on $R^{-1}M\cap B_8$. The proof of Proposition~\ref{prop:height-derivative-analysis}, which uses only this phase bound and the Lipschitz bound, shows that in $B_5$ the function $f_R$ is defined on $\{r>\sigma\}$, $\sigma=C_*\eps_R^{1/3}$, for a fixed large $C_*>0$, and satisfies $|D^kf_R|\le C_k\eps_Rr^{-2-k}$ there. We set
    \begin{align}
        w_R = r^2f_{R}/\eps_R.
    \end{align}
    These estimates, together with the quadratic structure of the
nonlinear terms in the minimal surface equation for normal graphs
over $V$, give
    \begin{align}
    \begin{cases}
        w_R:B_{5}\setminus \{r \leq C_*\eps_R^{1/3}\}\to \R,\\
        w_R(x,y) = w_R(r,y)\\
        |w_R|\leq C,\\
        |\Delta_{\xi,y} w_R | \leq C\eps_R r^{-5}.
    \end{cases}
    \end{align}
Regard $w_R$ as an $O(3)$-invariant function in the auxiliary
space $\R^3_\xi\times\R_y$. Put $\mathcal A=B_4\setminus\overline B_{1/4}$.
Prescribe $g_R=w_R$ on $\partial\mathcal A\cap\{r>\sigma\}$ and
extend it to a bounded continuous function over the remaining spherical polar neighborhoods, preserving
the symmetries. Let $H_R$ be its harmonic extension to the full
annulus $\mathcal A$; then $|H_R|\le C$.

 \medskip
   \noindent\textbf{Step 2: error estimates.} We now construct harmonic replacements on dyadic annuli and compare
them on their overlaps. This uses the approximation principle
underlying harmonic-replacement proofs of Schauder estimates;
see, for example, \cite[\S1]{WangSchauder}. Here we must also
control the error across the removed tube before patching the
replacements together.

On the region $r>\sigma$, define
\[
 \mathcal B(r)=C\frac{\sigma}{r}
 +C\eps_R\left(\frac{\sigma^{-2}}r-\frac1{r^3}\right).
\]
Fixing $\sigma = C_*\eps_R^{1/3}$, the boundary values and Laplacian satisfy
\[
 \begin{aligned}
 \mathcal B&\ge0,\qquad
 \mathcal B(\sigma)=C\ge|w_R-H_R|_{\{r=\sigma\}},\\
 (w_R-H_R)|_{\{\rho=1/4,4\}\cap\{r>\sigma\}}&=0,\\
 \Delta(r^{-1})&=0,\qquad
 \Delta(r^{-3})=6r^{-5},\\
 -\Delta\mathcal B&=6C\eps_R r^{-5}\ge|\Delta(w_R-H_R)| = |\Delta(H_R-w_R)|.
 \end{aligned}
\]
The maximum principle in $\mathcal A\cap\{r>\sigma\}$ therefore gives
\begin{equation}\label{eq:end-replacement-bound}
 |w_R-H_R|\le\mathcal B
 \le C\frac{\sigma+\eps_R\sigma^{-2}}r
 \le C\frac{\eps_R^{1/3}}r.
\end{equation}
Interior estimates for the Poisson equation give the corresponding
scaled derivative bounds away from $\{r=\sigma\}$. We now perform this
construction at dyadic scales and patch the results together. Let
$R_j=2^jR_0$, with $R_0>1$ large, let $H^{(j)}$ be the harmonic
replacement constructed above at scale $R=R_j$, and put
$H_j(X)=d_jH^{(j)}(X/R_j)$, where $d_j=D(8R_j)$; thus $H_j$
approximates $r^2f$ at scale $R_j$. Let $\{\chi_j\}_{j\ge1}$ be radial
functions with $\operatorname{supp}\chi_j\subset B_{2R_j}\setminus
B_{R_j/2}$, $|D^k\chi_j|\le C_kR_j^{-k}$ and $\sum_j\chi_j=1$ on
$\{\rho\ge R_1\}$, and define
\begin{align}
    W = \sum_{j=1}^\infty \chi_j H_j.
\end{align}
For the rescaled functions, on a slightly smaller auxiliary
annulus, the replacement bound gives
\[
 \begin{aligned}
 \|w_{R_j}-H^{(j)}\|_{L^2(\mathcal A\cap\{r>\sigma_{R_j}\})}^2
 &\le C\varepsilon_{R_j}^{2/3}\int_{\sigma_{R_j}}^4 r^{-2}r^2\,dr
 \le C\varepsilon_{R_j}^{2/3},\\
 |\mathcal A\cap\{r<\sigma_{R_j}\}|^{1/2}&\le C\sigma_{R_j}^{3/2}
 =C\varepsilon_{R_j}^{1/2}\le C\varepsilon_{R_j}^{1/3}.
 \end{aligned}
\]
The $L^2$ norms here are taken on fixed rescaled annuli. Taking
the square root in the first estimate gives an error of order
$\varepsilon_{R_j}^{1/3}$ outside the tube, whereas the omitted
tube contributes only $\varepsilon_{R_j}^{1/2}$. Since
$\varepsilon_{R_j}\to0$, the first contribution dominates.
Restoring the physical amplitude $d_j$, these two contributions are
\[
 d_j\varepsilon_{R_j}^{1/3}=\frac{d_j^{4/3}}{R_j},
 \qquad
 d_j\varepsilon_{R_j}^{1/2}
 =\left(\frac{d_j}{R_j}\right)^{3/2}.
\]
Thus the overlap estimate must retain the larger first term.

For adjacent indices $j$ and $j+1$, the function $H_j-H_{j+1}$ is
harmonic on the overlap of the corresponding annuli. Combining the
first bound in the preceding display, the volume bound together with
$|H_j|+|H_{j+1}|\le C(d_j+d_{j+1})$, and interior estimates for
harmonic functions on $\R^3_\xi\times\R_y$, we obtain
\begin{align}
    \|H_j - H_{j+1}\|_{C^{k}_{\mathrm{scaled}}(B_{R_{j+1}}\setminus B_{R_j})} \leq C_k \frac{d_j^{4/3}}{R_j} + C_k\frac{d_{j+1}^{4/3}}{R_{j+1}},
\end{align}
where $\|g\|_{C^{k}_{\mathrm{scaled}}} := \sum_{i=0}^{k} R^i\|D^ig\|_\infty$. By Lemma~\ref{prop:super-quadratic-growth}, $d_j\le C_\eta R_j^p$ with $p=2+\eta$, and hence
\begin{align}
    \|H_j - H_{j+1}\|_{C^{k}_{\mathrm{scaled}}(B_{R_{j+1}}\setminus B_{R_j})}\leq C_{k,\eta} R_j^{\frac43p-1}. 
\end{align}
Accordingly, the exponent used below is $p'=\frac43p-1$.
Since $\sum_jD\chi_j=0$ and $\sum_j\Delta\chi_j=0$ on $\{\rho\ge R_1\}$, on each overlap, subtracting a fixed
neighbor $H_*$ gives
\[
 \Delta_{\xi,y} W
 =\sum_j\left((\Delta\chi_j)(H_j-H_*)
                  +2D\chi_j\cdot D(H_j-H_*)\right).
\]
Thus
\begin{equation}\label{eq:end-patched-poisson}
 |W|\le C\rho^p,\qquad
 |D^k\Delta_{\xi,y} W|\le C_k\rho^{\frac43p-3-k}.
\end{equation}
On each fixed regular angular sector $\{r\ge\theta\rho,\ \rho\approx R\}$, the replacement estimate and interior elliptic estimates give:
\begin{align}\label{eq:diff-with-normal-graph}
    \|W-r^2f\|_{C^{k}_{\mathrm{scaled}}} \leq C_{k,\theta,p} R^{\frac43p-1}.
\end{align}

 \medskip
   \noindent\textbf{Step 3: leading order of the harmonic parto fo $W$.} 
 We next subtract a particular solution of the Poisson equation
and expand the remaining harmonic function at infinity.
For the general method of modified Newton potentials and
asymptotic harmonic expansions used here, see
\cite[Theorem~1.7 and Theorem~1.17 with Remark~1]{Bartnik}.
We give the estimates needed here explicitly. For this we just integrate the Newtonian potential as follows. We will use \eqref{eq:end-patched-poisson} with $-2<p'<2$ and
$p'\ne0$. Extend its source smoothly across a fixed compact set,
preserving $O(3)$ invariance and reflection in $y$. If
$p'>0$, solve it with the modified Newton potential
\[
 v(X)=\int_{\R^4}
 \bigl[\Gamma(X-Y)-\chi(Y)\{\Gamma(-Y)+D\Gamma(-Y)\cdot X\}\bigr]
 (\Delta W)(Y)\,dY,
\]
where $\Gamma(X)=-(4\pi^2|X|^2)^{-1}$ is the fundamental solution of the Laplacian on $\R^4$, and$\chi$ is radial, zero near zero and one near infinity. The
kernel at $|Y|>2|X|$ satisfies
\[
 \begin{aligned}
 |\Gamma(X-Y)-\Gamma(-Y)-D\Gamma(-Y)\cdot X|
 &\le C|X|^2|Y|^{-4},\\
 \int_{2|X|}^\infty |X|^2s^{p'-3}\,ds
 &\le C|X|^{p'}\qquad(p'<2).
 \end{aligned}
\]
On every centered truncated ball the linear subtraction integrates to
zero by symmetry. Splitting the remaining integral into
$|Y|<|X|/2$ and $|X|/2<|Y|<2|X|$ gives
\[
 \begin{aligned}
 |X|^{-2}\int_1^{|X|/2}s^{p'+1}\,ds
 +\int_1^{|X|/2}s^{p'-1}\,ds
 &\le C|X|^{p'}\qquad(p'>0),\\
 \int_{\{|X|/2<|Y|<2|X|\}}
 (|\Gamma(X-Y)|+|\Gamma(-Y)|)\,|Y|^{p'-2}\,dY
 &\le C|X|^{p'}.
 \end{aligned}
\]
The combined integral is therefore absolutely convergent and
$v=O(\rho^{p'})$. For $-2<p'<0$, use the unmodified Newton kernel:
\[
 \rho^{-2}\int_1^\rho s^{p'+1}\,ds
 +\int_\rho^\infty s^{p'-1}\,ds
 \le C\rho^{p'}.
\]
Interior Poisson estimates yield $|D^kv|\le C_k\rho^{p'-k}$.
Hence
\[
 \Delta( W-v)=0
\]
on the full exterior auxiliary domain. The invariant even expansion
has the form
\begin{align}\label{eq:harmonic-expansion}
  W-v
 =\sum_{\ell\in2\mathbb N_0}
 \left(a_\ell\rho^\ell+b_\ell\rho^{-\ell-2}\right)
 Y_\ell(r/\rho,y/\rho),
\end{align}
where $Y_\ell$ is the $O(3)$ invariant spherical harmonic of degree
$\ell$ on $S^3$.

We first take $\eta=1/11$, so that
\begin{align}
    p = \frac{23}{11},\quad p' = \frac{59}{33} = 2-\kappa,\quad \kappa = \frac{7}{33}, 
\end{align}
and using \eqref{eq:harmonic-expansion},\eqref{eq:diff-with-normal-graph} and the growth, we see that:
\[
 \begin{aligned}
W-v&=O(\rho^{2+\eta}),\qquad 2+\eta<3,\\
\Rightarrow W-v
 &=A\left(y^2-\frac{r^2}{3}\right)+a_0+O(\rho^{-2}),\\
 f&=AJ_2+O(\rho^{-\kappa})
 \qquad\text{on }\{r>\theta\rho\}.
 \end{aligned}
\]
Axial monotonicity passes to the normal coordinates by the exact
identity $f_y=\sqrt2u_y/(1+u_t)\ge0$. The differentiated expansion
on fixed regular sectors then gives $A\ge0$.

\medskip
\noindent \textbf{Step 4: non degeneracy of the quadratic harmonic part.}
 Suppose now that $A=0$. Then $|r^2f|\le C(1+R^{p'})$ on the
regular sectors at scale $R$, and the nonconcentration estimate
\eqref{eq:end-nonconcentration} gives
\begin{align}
    D(R) \leq C_\gamma (1+R^{p'}) + C_0\gamma D(4R).
\end{align}
Now name $\alpha_0= C_0\gamma$ and choose $\gamma$ small such that $\alpha_04^p < 1$ and iterate the above to see that for every $N$,
\[
 \begin{aligned}
 D(R)&\le C_\gamma\sum_{j=0}^{N-1}
             \alpha_0^j\bigl(1+(4^jR)^{p'}\bigr)+\alpha_0^ND(4^NR),\\
 \alpha_0^ND(4^NR)&\le CR^p(\alpha_04^p)^N\longrightarrow0.
 \end{aligned}
\]
Since $\alpha_04^{p'}<\alpha_04^p<1$, letting $N\to\infty$ improves the bound to
\begin{align}
    D(R)\le C(1+R^{\max(p',0)}) = C(1+R^{\max(\frac43p-1,0)}).
\end{align}
We now repeat the argument with $p$ replaced by $p'$. Since
$p'<2$, the quadratic term is then absent from the expansion of $W-v$,
and we obtain the same conclusion with the exponent $\frac43p'-1$.
Iterating the exponent map $p\mapsto\frac43p-1$ gives
\[
 \frac{23}{11}\ \longmapsto\ \frac{59}{33}\ \longmapsto\ \frac{137}{99}
 \ \longmapsto\ \frac{251}{297}\ \longmapsto\ \frac{113}{891}
 \ \longmapsto\ -\frac{2221}{2673}.
\]
None of these exponents is zero, and the last is negative, so the
preceding bound gives $D(R)\le C$.
 
Repeat the replacement with this bounded $D$ and retain the
unweakened errors. In particular,
$d_j\eps_{R_j}^{1/3}=d_j^{4/3}/R_j=O(R_j^{-1})$.
The overlap and regular-sector estimates give
\[
 \Delta W=O(\rho^{-3}),\qquad
 W-r^2f=O_\theta(\rho^{-1})\quad\text{on }r\ge\theta\rho.
\]
The unmodified Newton potential is $O(\rho^{-1})$, and the
remaining bounded exterior harmonic function is
$b_\infty+O(\rho^{-2})$. The fixed-leaf expansion
\cite[Proposition~3.3]{Sz}, with the phase bounded, therefore yields
\[
 P=b_\infty+O_\theta(\rho^{-1})\quad\text{on }r\ge\theta\rho.
\]
Lemma~\ref{lem:bounded-phase-transfer} now supplies the missing
uniformity through the smoothing regions:
\[
 P=b_\infty+o(1)\quad\text{on the entire end of }M.
\]
If $b_\infty\ge0$, the $C_{4,3}$ tangent at zero supplies a regular
point of negative phase. Uniform convergence at infinity and
$P(0)=0$ imply that the negative global minimum of $P|_M$ is
attained at a regular interior point. At that point $M$ touches the
smooth negative Hardt--Simon leaf from one side. The strong maximum
principle gives local equality. The set of points of equality is
then open and closed in the connected regular hypersurface, by the
same comparison at each of its limit points. This would make $P$
constant there, contradicting $P(0)=0$ and the negative minimum.
If $b_\infty\le0$, a positive-phase point gives the analogous
contradiction at a positive maximum. Hence $A>0$.

\medskip
\noindent \textbf{Step 5: conclusion of the proof.}
Finally, the expansion $f=AJ_2+O(\rho^{-\kappa})$ on regular
sectors gives
\[
 \sup_{\{r>\theta R,\ \rho\approx R\}}|r^2f|\le CR^2
 \quad\overset{\eqref{eq:end-nonconcentration}}{\Longrightarrow}\quad
 D(R)\le CR^2+qD(4R)
 \quad\Longrightarrow\quad D(R)\le CR^2,
\]
using the same forward iteration (note that $16\alpha_0<\alpha_04^p<1$) and the previously known growth bound.
\end{proof}

\subsection{Higher derivative estimates}
We prove Proposition \ref{prop:precise-estimates-for-connector} by highlighting the geometric properties coming the problem, and deferring to the appendix the standard PDE techniques. The ideas of this section were provided to us by ChatGPT Astra.

\begin{proof}[Proof of Proposition~\ref{prop:precise-estimates-for-connector}]
First we apply Lemma~\ref{lem:quadratic-phase-near-spine} (which is a somewhat standard application of the estimates in \cite{Sz}), to push the quadratic decay of Proposition \ref{prop:quadratic-bound} to the spine to bound the smoothing radius by
\[
 a_\infty(Y)=\left(\frac{A}{c_H}\right)^{1/3}|Y|^{2/3}
 +O(|Y|^{2/3-1/6}),\qquad A>0.
\]
The details are given for the readers convenience in the appendix for the reader's convenience.

Fix a large $|y_0|$, put $\beta=u_\infty(0,y_0)$ and
$R=\theta_0|y_0|$, and keep $\beta$ fixed throughout the band.
Here $\theta_0>0$ is a sufficiently small geometric constant.
Then $\beta^3\approx R^2$ and $\beta/R\to0$.
For the fixed leaf family, write
\[
 H(t,\lambda)=\lambda^{1/3}h(t/\lambda^{1/3}),\qquad
 P(t,H(t,\lambda))=c_H\lambda.
\]

We now identify the weighted equation. All derivatives and divergences
in the following formulas are in the base variables
$X=(x,y)\in\R^4\times\R$. Put
\[
 W=\sqrt{1+|Du|^2},\qquad
 B=\frac{I}{W}-\frac{Du\otimes Du}{W^3},\qquad
 d(X)=H_\lambda\!\left(t,c_H^{-1}P(t,u(X))\right)\,,
\]
where $H_\lambda$ denotes the derivative of $H$ in the second variable. Differentiating the Euler--Lagrange equation for $u^3W$ in $y$
gives the exact equation
\[
 \mathcal Ju_y=0,\qquad
 \mathcal J\phi=-\operatorname{div}(u^3B D\phi)-\frac{3u}{W}\phi.
\]
For $q_1=u_y/(\beta^3d)$, the identity
$d\mathcal J(dq)= -\operatorname{div}(u^3d^2B Dq)+d(\mathcal Jd)q$
therefore gives
\begin{equation}\label{eq:direct-overview-equation}
 -\operatorname{div}(\mathcal A Dq_1)+\mathcal Vq_1=0,
 \qquad \mathcal A=u^3d^2B,\quad \mathcal V=d\mathcal Jd.
\end{equation}
The field $d$ is evaluated at the leaf through each point of the
connector. Because that leaf varies with $X$, the residual potential
$\mathcal V$ need not vanish. On a fixed leaf, however, the positive
dilation field $d_*=H_\lambda(\cdot,\lambda_*)$ satisfies
$\mathcal J_*d_*=0$. The normalization therefore cancels the
potential of the model leaf, leaving only the error from comparison
with that leaf.

Here is where the positive quadratic asymptotics enter. On the chosen
band the quadratic estimate of Proposition \ref{prop:quadratic-bound}, in the form of Lemma~\ref{lem:quadratic-phase-near-spine}, gives
\begin{equation}\label{eq:phaseosc}
 c\beta^3\le c_H^{-1}P(t,u(t,y))\le C\beta^3.
\end{equation}
(see Lemma~\ref{lem:direct-leaf-comparison}).
The fixed-leaf bounds consequently imply
$u=H(t,\lambda)\approx s$ and $d=H_\lambda(t,\lambda)\approx s^{-2}$, where $s=\beta+t$.
Since $|Du|\le\Lambda$, the matrix $B$ is uniformly elliptic, and hence
\[
 c\omega_\beta I\le\mathcal A\le C\omega_\beta I,
 \qquad \omega_\beta(X)=(\beta^2+|x|^2)^{-1/2}.
\]
The model weights satisfy $[\omega_\beta]_{A_2}\le C$ uniformly
for every $\beta>0$, as verified in
Appendix~\ref{sec:direct-weighted-theory}. Thus the scalar factor
$u^3d^2\approx \omega_\beta$ also has a uniform local $A_2$ bound.
The role of the quadratic asymptotics is to obtain this comparison
for the actual solution throughout the band; the $A_2$ property
of the model weight itself is independent of those asymptotics.

The quantitative comparison with a fixed leaf also gives
\[
 |\mathcal V|\le C\omega_\beta s^{-2}(s/R)^{1/2}.
\]
The factor $(s/R)^{1/2}$ comes from the phase oscillation estimate
\eqref{eq:phaseosc}, using
$\beta^3\approx R^2$. Relative to its principal weight, the
original Jacobi potential has size $3/(Wu^2)\approx s^{-2}$.
The new residual potential gains the factor $(s/R)^{1/2}$, as it is measuring the error of the foliation after modulating the dilation, although
it need not satisfy a pointwise bound by $C\omega_\beta R^{-2}$
near the axis. The useful smallness is instead measured by the
Dirichlet Green kernel $G_D$ of
$-\operatorname{div}(\mathcal A D)$ on an interior ball
$D=B_\rho$. Lemma~\ref{lem:direct-weighted-estimate}, using
\cite[Theorem~3.3]{FJK}, proves
\[
 \sup_{X\in D}\int_D G_D(X,Z)|\mathcal V(Z)|\,dZ
 \le C(\rho/R)^{1/2},
\]
with a constant independent of $\beta/R$.

The summability behind this estimate can be seen on dyadic shells
centered on the axis. At radius $r\gg\beta$, their weighted volume
is comparable to $r^4$, the Green kernel is bounded by $Cr^{-2}$,
and $|\mathcal V|/\omega_\beta\le Cr^{-2}(r/R)^{1/2}$.
The contribution of each shell is consequently bounded by
\[
 Cr^{-2}\,r^4\,r^{-2}(r/R)^{1/2}=C(r/R)^{1/2}.
\]
These contributions form a summable geometric series. Without
the additional factor $(r/R)^{1/2}$, the same calculation would
give a constant on each shell and a bound growing with the number
of scales between $\beta$ and $\rho$.

Choose $\rho=\delta_0R$, with $\delta_0>0$ fixed sufficiently small
that the Green norm above is at most $1/4$.
Lemma~\ref{lem:direct-weighted-estimate} then solves
\[
 \vartheta(X)=1-\int_D G_D(X,Z)\mathcal V(Z)\vartheta(Z)\,dZ,
 \qquad \tfrac12\le\vartheta\le2.
\]
Thus $\mathcal L\vartheta=0$, where
$\mathcal L=-\operatorname{div}(\mathcal A D\cdot)+\mathcal V$.
Equation~\eqref{eq:direct-overview-equation} then becomes
\[
 -\operatorname{div}\!\left(\vartheta^2\mathcal A
                   D(q_1/\vartheta)\right)=0.
\]
Writing $z=q_1/\vartheta$ and testing with $\eta^2z$ gives the
weighted Caccioppoli, or reverse Poincar\'e, inequality
\[
 \int_D\eta^2|Dz|^2\omega_\beta\,dX
 \le C\int_D|D\eta|^2z^2\omega_\beta\,dX.
\]
This energy estimate uses only weighted ellipticity. The uniform
$A_2$ bound supplies the accompanying weighted Sobolev inequality,
which gives local boundedness by iteration. This is the equation to which
\cite[Theorem~2.3.1]{FKS} applies: its symmetric coefficient matrix
is comparable to the same $A_2$ weight, with constants independent
of $\beta/R$. Appendix~\ref{sec:direct-weighted-theory} states the
weighted estimates and treats the potential and the higher-order forcing.

Using higher order analogue of the estimate \eqref{eq:phaseosc}, obtained in Lemma~\ref{lem:direct-leaf-comparison} by scaling of standard elliptic estimates, Proposition~\ref{prop:direct-axial-regularity} proves, for every
fixed finite derivative order,
\begin{equation}\label{eq:direct-overview-bounds}
 |D^k\partial_y^ju|
 \le C_{j,k}\beta^3R^{-j}(\beta+t)^{-2-k},
 \qquad j\ge1,\quad k\ge0,
\end{equation}
on $B_{R/4}(0,y_0)$. Its proof uses the same normalization
$q_j=\partial_y^ju/(\beta^3d)$ at every order. The additional
factor $(\beta/(\beta+t))^3$ in the higher-order forcing makes
its Green potential uniformly bounded. Energy estimates and the
weighted local boundedness estimate then give one factor $R^{-1}$
for each axial derivative.

At $t=0$, \eqref{eq:direct-overview-bounds} immediately gives
\[
 |a_\infty^{(j)}(y)|\le C_j\beta R^{-j}
 \le C_ja_\infty(y)R^{-j}.
\]
Proposition~\ref{prop:direct-leaf-expansion} applies these direct
derivative estimates to the radial equation. For
$0\le t\le cR$ and $|y-y_0|\le cR$ it gives
\[
 u_\infty(t,y)=a_\infty(y)h(t/a_\infty(y))+e_\infty(t,y),
\]
with
\[
 |\partial_t^m\partial_y^ne_\infty|
 \le C_{mn}a_\infty(y)^3R^{-2-n}(a_\infty(y)+t)^{-m}.
\]
For any specified $m,n$, use the direct estimates through axial
order $n+2$. The geometric band and lower height threshold are
independent of the chosen finite order; the constants may depend
on it. Taking $y_0=Y$ proves \eqref{roadmap:eq:sharp-end}, and
the smoothing-radius asymptotic identifies
$\kappa_H=(A/c_H)^{1/3}$.

On a fixed regular angular sector, the quadratic phase bound and
the fixed-leaf expansion give $|f_\infty|\le C_\theta$.
After rescaling an annulus of radius $\rho$ to unit size, the
normal height is at most $C_\theta\rho^{-1}$. The interior
normal-graph estimates of
Proposition~\ref{prop:height-derivative-analysis}, on a slightly
larger regular sector, give the same bound for every fixed
derivative order in rescaled coordinates. Scaling back yields
$|D^jf_\infty|\le C_{\theta,j}\rho^{-j}$, as required.

Set
\[
 d(X)=H_\lambda\!\left(t,c_H^{-1}P(t,u(t,y))\right)
 \approx(\beta+t)^{-2}.
\]
This is the fixed leaf's positive Jacobi field evaluated at the
solution. We differentiate the equation for $u$; $P$ is used only to select
this normalizing field.
Proposition~\ref{prop:direct-axial-regularity} proves, for every
fixed finite derivative order,
\begin{equation}\label{eq:direct-overview-bounds-recalled}
 |D^k\partial_y^ju|
 \le C_{j,k}\beta^3R^{-j}(\beta+t)^{-2-k},
 \qquad j\ge1,\quad k\ge0,
\end{equation}
on $B_{R/4}(0,y_0)$. Its proof starts with the exact Jacobi
equation for $u_y$. Normalizing each $\partial_y^ju$ by
$\beta^3d$ gives a common weighted operator. The comparison
above makes its potential small in the Green-kernel norm.
The nonlinear forcing at higher orders decays by an additional
factor $(\beta/(\beta+t))^3$, which makes its Green potential
uniformly bounded. Energy estimates and weighted local boundedness
then give one factor $R^{-1}$ for each axial derivative.

At $t=0$, \eqref{eq:direct-overview-bounds-recalled} immediately gives
\[
 |a_\infty^{(j)}(y)|\le C_j\beta R^{-j}
 \le C_ja_\infty(y)R^{-j}.
\]

Proposition~\ref{prop:direct-leaf-expansion} applies these direct
derivative estimates to the radial equation. For
$0\le t\le cR$ and $|y-y_0|\le cR$ it gives
\[
 u_\infty(t,y)=a_\infty(y)h(t/a_\infty(y))+e_\infty(t,y),
\]
with
\[
 |\partial_t^m\partial_y^ne_\infty|
 \le C_{mn}a_\infty(y)^3R^{-2-n}(a_\infty(y)+t)^{-m}.
\]
For any specified $m,n$, use the direct estimates through axial
order $n+2$. The geometric band and lower height threshold are
independent of the chosen finite order; the constants may depend
on it. Taking $y_0=Y$ proves \eqref{roadmap:eq:sharp-end}, and
the smoothing-radius asymptotic identifies
$\kappa_H=(A/c_H)^{1/3}$.

On a fixed regular angular sector, the quadratic phase bound and
the fixed-leaf expansion give $|f_\infty|\le C_\theta$.
After rescaling an annulus of radius $\rho$ to unit size, the
normal height is at most $C_\theta\rho^{-1}$. The interior
normal-graph estimates of
Proposition~\ref{prop:height-derivative-analysis}, on a slightly
larger regular sector, give the same bound for every fixed
derivative order in rescaled coordinates. Scaling back yields
$|D^jf_\infty|\le C_{\theta,j}\rho^{-j}$, as required.
\end{proof}

\section{Accumulating singularities}\label{sec:accumulation}
This section proves Theorem~\ref{thm:main-accumulation} using the quantitative connector estimates and a suitable gluing construction. 

We first describe the background model and show why direct interpolation with a modulation of the foliation fails.

We then prove Theorem~\ref{thm:main-accumulation} by preparing a background adapted to the connector insertions. The construction is organized around one transverse correction $f(t,y)$. We first specify its required estimates, then use them to prepare the background and its metric, insert the connectors, and update the
metric locally. Finally we choose $f$ from the finite-order
profiles constructed in Section~\ref{sec:finite-correctors},
where we prove the residual and matching estimates.

The purpose of $f$ is to control both the background mean curvature
and the error on the attachment annuli. These two requirements
permit a smooth conformal metric with the full Euclidean jet at
the accumulation point. The parameter-selection order is similar
to that in \cite{DHLS}: choose a background modification to
fit the intended insertions. Here its transverse dependence is
essential.

Preparing the pieces before joining them also appears in the
constant-mean-curvature construction of Mazzeo--Pacard--Pollack
\cite{MPP01}, where Green-function corrections for the Jacobi
operator adapt the summands to the inserted neck. Here we prepare
the background by successive radial corrections about Hardt--Simon
leaves and then correct the ambient metric conformally.

\subsection{The base model and the interpolation obstruction}
We first keep only the zeroth-order approximation by foliation leaves.
On a transition annulus it has exactly the connector's axis value $a(y)$,
so any error comes from how the profiles open away from the axis, that is how far the connector is from the foliation modulated by $a(y)$.
The calculation below isolates that error at the axis and shows why
decreasing the smoothing radius does not by itself give a regular conformal
correction. This identifies the terms that the higher-order profiles
must cancel.

The unmodified model is $V=C_{3,3}\times\R$. Let $a(y)>0$ for
$y\ne0$, with $a(0)=0$, and let $\mathcal R(y)\downarrow0$ be a
transverse cutoff scale, with $a(y)\ll\mathcal R(y)\ll|y|$.
Near the spine, replace the cone in the slice of height $y$ by
the positive Hardt--Simon leaf
\[
 v_0(t,y)=a(y)h(t/a(y)),\qquad t=|x|.
\]
Cut this profile off to $t$ for $t\gtrsim\mathcal R(y)$ and set
the slice at $y=0$ equal to the Simons cone. This is the first
candidate background. Later we take $a(y)\approx e^{-1/y^2}$ and
$\mathcal R(y)\approx|y|^4$: other choices are possible but $a(y)$ must vanish at infinite order at $0$ to obtain a $C^\infty$ metric whose jet at zero coincide with the euclidean metric.

At $p_k=(0,0,y_k)$, $y_k\downarrow0$, place a translated and dilated
connector
\[
 u_k(t,y)=\lambda_k u_\infty(t/\lambda_k,(y-y_k)/\lambda_k)
\]
in a ball of radius comparable to $R_k\ll y_k$ (to keep the transition regions disjoint), with
$\lambda_k/R_k\to0$ (so that the connector is blown down to resemble $V$ at infinity, that is the model). The parts of its attachment annulus near
$x=0$, above and below $y=y_k$, will be called the \emph{polar
regions}.

 On the corresponding axial matching intervals choose
$a(y)=u_k(0,y)$, so that the connector and the
background leaf have exactly the same smoothing radius. One might then try the
direct interpolation
\[
 v_\chi=v_0+\chi\!\left(
       \frac{\sqrt{2t^2+(y-y_k)^2}}{R_k}\right)(u_k-v_0),
\]
where $\chi=1$ on $[0,1]$ and $\chi=0$ on $[2,\infty)$.
The transverse cutoff of the background is one near the axis, so it
does not affect the following calculation.

Write $Q$ for the radial SME operator, as in \eqref{cap:eq:Q}:
\begin{equation}\label{eq:gluing-Q-naive}
 Q(v)=\left(\delta_{ij}-\frac{v_i v_j}{1+|Dv|^2}\right)v_{ij}
          +\frac3t v_t-\frac3v,\qquad i,j\in\{t,y\}.
\end{equation}
At $t=0$ the smooth radial interpretation gives
$Q(v)=4v_{tt}+v_{yy}/(1+v_y^2)-3/v$.
Since $h''(0)=3/4$ and $u_k$ solves $Q(u_k)=0$, as already observed in Section \ref{sec:connector-estimates}, the difference
$e=u_k-v_0$ satisfies
\[
 e(0,y)=e_t(0,y)=0,\qquad
 4e_{tt}(0,y)=-\frac{a''}{1+(a')^2}.
\]
Writing $\chi_0(y)=\chi(|y-y_k|/R_k)$, all derivatives of the
interpolation cutoff multiply data that vanish at $t=0$.
Consequently the scalar mean curvature, with the normal pointing
towards increasing $|z|$, satisfies the exact identity
\begin{equation}\label{eq:naive-axis-curvature}
 H_\chi(0,y)=-\frac{Q(v_\chi)(0,y)}{\sqrt{1+(a')^2}}
       =-(1-\chi_0(y))\frac{a''(y)}{(1+(a'(y))^2)^{3/2}}.
\end{equation}
Now suppose there is a $C^2$ metric $g=e^{2F}g_{\rm euc}$, then the equation for the mean curvature is $H_\chi+8\nu F=H^g_\chi=0$, if we require minimality.
Writing \(z=a(y)\omega\), with \(|\omega|=1\), the normal pointing toward increasing \(|z|\) is
\[
\nu=\frac{(0,\omega,-a'(y))}{\sqrt{1+(a'(y))^2}}.
\]
Since the averaged \(F\) depends on \(z\) through \(s=|z|\), this gives
\[
\nu F
=\frac{F_s(0,a(y),y)-a'(y)F_y(0,a(y),y)}
       {\sqrt{1+(a'(y))^2}}.
\]
Substitute this into \(H_\chi+8\nu F=0\):
\[
-\frac{(1-\chi_0)a''}{(1+(a')^2)^{3/2}}
+8\frac{F_s-a'F_y}{\sqrt{1+(a')^2}}=0.
\]
Multiplying by \(\sqrt{1+(a')^2}\) and dividing by \(8a\) yields exactly
\[
\frac{(1-\chi_0)a''}{8a(1+(a')^2)}
=\frac{F_s(0,a,y)}{a}
-\frac{a'}a F_y(0,a,y).
\]
Fix $\tau\in(1,2)$ such that $1-\chi(\tau)>0$, and evaluate
at $y=y_k+\tau R_k$. By
Proposition~\ref{prop:precise-estimates-for-connector},
\[
 a(y)\sim\kappa_H\lambda_k^{1/3}(y-y_k)^{2/3},
\]
with the asymptotic valid through two derivatives on the scale
$R_k$. Since $\lambda_k/R_k\to0$, it follows that
\[
 a'(y)\to0,\qquad
 \left|\frac{a'}a\right|\le CR_k^{-1},
 \qquad
 -\frac{a''}a\ge cR_k^{-2}.
\]

On the other hand, rotational invariance and $C^2$ regularity
give $F_s(0,0,y)=0$, and hence
\[
 \frac{|F_s(0,a(y),y)|}{a(y)}
 \le \|D^2F\|_{L^\infty(U)},
\]
where $U$ is a fixed neighborhood of the origin containing
these points for all sufficiently large $k$.
Taking absolute values in the preceding identity therefore yields
\[
 cR_k^{-2}
 \le \|D^2F\|_{L^\infty(U)}
       +CR_k^{-1}\|DF\|_{L^\infty(U)}.
\]
Multiplying by $R_k^2$ and letting $k\to\infty$ gives a
contradiction. Thus the interpolated hypersurface cannot be
minimal in a $C^2$ conformally Euclidean metric near the origin.

\subsection{The foliation and the prescribed insertions}

Write $(x,z,y)\in\R^4\times\R^4\times\R$ and $t=|x|$, and let $h$ be the
positive normalized Hardt--Simon profile from \eqref{cap:eq:HS}. Thus
\begin{equation}\label{bg:eq:h}
 \frac{h''}{1+h'^2}+\frac{3h'}{\rho}-\frac3h=0,
 \qquad h(0)=1,\quad h'(0)=0,\quad h''(0)=\frac34.
\end{equation}
We use its differentiated end expansion
\begin{equation}\label{bg:eq:h-end}
 h(\rho)=\rho+b_H\rho^{-2}+O(\rho^{-3}),\qquad b_H>0.
\end{equation}
Write $S(v)=\{|z|=v(|x|,y)\}$. For a positive radius
function $a$ define the uncorrected background
\begin{equation}\label{bg:eq:raw}
 V_0[a](t,y)=a(y)h(t/a(y)).
\end{equation}
It suffices to consider this surface in a tube $0\le t<\mathcal R(y)$
with $\mathcal R(y)/a(y)\to\infty$ as $y\downarrow0$.

For a rotational profile $v$, put
\begin{equation}\label{eq:gluing-Q}
 Q(v)=\frac{(1+v_y^2)v_{tt}-2v_tv_yv_{ty}
                    +(1+v_t^2)v_{yy}}{1+v_t^2+v_y^2}
                  +\frac{3v_t}{t}-\frac3v.
\end{equation}
With the normal pointing towards increasing $|z|$, its scalar
Euclidean mean curvature is $H=-Q(v)/\sqrt{1+v_t^2+v_y^2}$.

Let $u_\infty(t,Y)$ be the connector of
Theorem~\ref{thm:main-connector}. Its properties and
Proposition~\ref{prop:precise-estimates-for-connector} give the following
properties. The Cartesian profile $u_\infty(|x|,Y)$ is globally
Lipschitz and vanishes only at $(x,Y)=(0,0)$. The hypersurface
$S(u_\infty)$ is stationary with multiplicity one. Its only singular
point is the ambient origin, with unique tangent cone $C_{4,3}$;
at infinity it has finite density and unique tangent cone
$C_{3,3}\times\R$. Its cap satisfies
\[
 u_\infty(0,Y)\sim\kappa_H|Y|^{2/3},\qquad
 |\partial_Y^n u_\infty(0,Y)|\le C_n|Y|^{2/3-n}
 \quad (|Y|\text{ large}).
\]
We also use its regular-sector derivative estimates and the
matching estimates recalled in Section~\ref{sec:finite-correctors}.

Fix a small $c_*>0$. For all sufficiently large $k$, choose
\begin{equation}\label{bg:eq:insertion-data}
\begin{aligned}
 y_k&=2^{-k},& p_k&=(0,0,y_k),& R_k&=c_*y_k^4,\\
 \lambda_k&=\frac{e^{-3/y_k^2}}{\kappa_H^3R_k^2},&
 u_k(t,y)&=\lambda_k u_\infty(t/\lambda_k,(y-y_k)/\lambda_k).
\end{aligned}
\end{equation}
Thus $u_k(0,y)\approx e^{-1/y_k^2}$ when
$|y-y_k|\approx R_k$, and
$\lambda_k\ll e^{-1/y_k^2}\ll R_k\ll y_k$.
The insertion neighborhoods are disjoint.

The cutoff coordinate $\rho_k=\sqrt{2t^2+(y-y_k)^2}$ separates
an exact core $\rho_k\le R_k$ from the attachment annulus
$R_k<\rho_k<2R_k$. The parts of this annulus near $t=0$ are
its \emph{polar regions}; the remaining part is the
\emph{regular region}. Figure~\ref{fig:modulated-foliation-gluing}
shows these regions and the corrected background to be chosen.

\begin{figure}[htbp]
\centering
\scalebox{1.1}{%
\begingroup
\definecolor{glueink}{HTML}{202C3B}
\definecolor{glueblue}{HTML}{2563A6}
\definecolor{glueorange}{HTML}{B85818}
\definecolor{gluepurple}{HTML}{78549A}
\definecolor{gluegray}{HTML}{929AA5}
\definecolor{gluelight}{HTML}{DCE2E9}
\definecolor{gluebluefill}{HTML}{F0F5FB}
\definecolor{glueorangefill}{HTML}{FFF0DF}
\definecolor{gluepurplefill}{HTML}{EEE6F5}
\definecolor{gluewhite}{HTML}{FFFFFF}
\definecolor{gluegreen}{HTML}{087F73}
\definecolor{gluegreenfill}{HTML}{DDF3EA}
\begin{tikzpicture}[x=1cm,y=1cm,line join=round,line cap=round]
\path[use as bounding box] (0,0.3) rectangle (5.5,8.0);
\node[anchor=west,align=center,inner sep=0pt,text=glueink,font={\fontsize{10}{12.0}\selectfont\bfseries}] at (0.1200,7.6600) {Gluing neighborhood};
\path[draw=gluebluefill,line width=0.1pt,fill=gluebluefill] (1.1000,2.2400) -- (4.6500,2.2400) -- (4.6500,7.0000) -- (1.1000,7.0000) -- cycle;
\path[draw=gluepurple,line width=0.85pt,fill=gluepurplefill] (1.1000,2.4200) -- (1.1864,2.4217) -- (1.2726,2.4268) -- (1.3586,2.4352) -- (1.4442,2.4471) -- (1.5292,2.4623) -- (1.6136,2.4808) -- (1.6972,2.5026) -- (1.7798,2.5277) -- (1.8615,2.5560) -- (1.9419,2.5875) -- (2.0211,2.6221) -- (2.0988,2.6598) -- (2.1750,2.7005) -- (2.2495,2.7442) -- (2.3223,2.7908) -- (2.3931,2.8402) -- (2.4620,2.8923) -- (2.5288,2.9471) -- (2.5934,3.0045) -- (2.6556,3.0644) -- (2.7155,3.1266) -- (2.7729,3.1912) -- (2.8277,3.2580) -- (2.8798,3.3269) -- (2.9292,3.3977) -- (2.9758,3.4705) -- (3.0195,3.5450) -- (3.0602,3.6212) -- (3.0979,3.6989) -- (3.1325,3.7781) -- (3.1640,3.8585) -- (3.1923,3.9402) -- (3.2174,4.0228) -- (3.2392,4.1064) -- (3.2577,4.1908) -- (3.2729,4.2758) -- (3.2848,4.3614) -- (3.2932,4.4474) -- (3.2983,4.5336) -- (3.3000,4.6200) -- (3.2983,4.7064) -- (3.2932,4.7926) -- (3.2848,4.8786) -- (3.2729,4.9642) -- (3.2577,5.0492) -- (3.2392,5.1336) -- (3.2174,5.2172) -- (3.1923,5.2998) -- (3.1640,5.3815) -- (3.1325,5.4619) -- (3.0979,5.5411) -- (3.0602,5.6188) -- (3.0195,5.6950) -- (2.9758,5.7695) -- (2.9292,5.8423) -- (2.8798,5.9131) -- (2.8277,5.9820) -- (2.7729,6.0488) -- (2.7155,6.1134) -- (2.6556,6.1756) -- (2.5934,6.2355) -- (2.5288,6.2929) -- (2.4620,6.3477) -- (2.3931,6.3998) -- (2.3223,6.4492) -- (2.2495,6.4958) -- (2.1750,6.5395) -- (2.0988,6.5802) -- (2.0211,6.6179) -- (1.9419,6.6525) -- (1.8615,6.6840) -- (1.7798,6.7123) -- (1.6972,6.7374) -- (1.6136,6.7592) -- (1.5292,6.7777) -- (1.4442,6.7929) -- (1.3586,6.8048) -- (1.2726,6.8132) -- (1.1864,6.8183) -- (1.1000,6.8200) -- cycle;
\path[draw=glueorange,line width=0.85pt,fill=glueorangefill] (1.1000,3.5200) -- (1.1432,3.5208) -- (1.1863,3.5234) -- (1.2293,3.5276) -- (1.2721,3.5335) -- (1.3146,3.5411) -- (1.3568,3.5504) -- (1.3986,3.5613) -- (1.4399,3.5738) -- (1.4807,3.5880) -- (1.5210,3.6037) -- (1.5605,3.6210) -- (1.5994,3.6399) -- (1.6375,3.6603) -- (1.6747,3.6821) -- (1.7111,3.7054) -- (1.7466,3.7301) -- (1.7810,3.7562) -- (1.8144,3.7836) -- (1.8467,3.8122) -- (1.8778,3.8422) -- (1.9078,3.8733) -- (1.9364,3.9056) -- (1.9638,3.9390) -- (1.9899,3.9734) -- (2.0146,4.0089) -- (2.0379,4.0453) -- (2.0597,4.0825) -- (2.0801,4.1206) -- (2.0990,4.1595) -- (2.1163,4.1990) -- (2.1320,4.2393) -- (2.1462,4.2801) -- (2.1587,4.3214) -- (2.1696,4.3632) -- (2.1789,4.4054) -- (2.1865,4.4479) -- (2.1924,4.4907) -- (2.1966,4.5337) -- (2.1992,4.5768) -- (2.2000,4.6200) -- (2.1992,4.6632) -- (2.1966,4.7063) -- (2.1924,4.7493) -- (2.1865,4.7921) -- (2.1789,4.8346) -- (2.1696,4.8768) -- (2.1587,4.9186) -- (2.1462,4.9599) -- (2.1320,5.0007) -- (2.1163,5.0410) -- (2.0990,5.0805) -- (2.0801,5.1194) -- (2.0597,5.1575) -- (2.0379,5.1947) -- (2.0146,5.2311) -- (1.9899,5.2666) -- (1.9638,5.3010) -- (1.9364,5.3344) -- (1.9078,5.3667) -- (1.8778,5.3978) -- (1.8467,5.4278) -- (1.8144,5.4564) -- (1.7810,5.4838) -- (1.7466,5.5099) -- (1.7111,5.5346) -- (1.6747,5.5579) -- (1.6375,5.5797) -- (1.5994,5.6001) -- (1.5605,5.6190) -- (1.5210,5.6363) -- (1.4807,5.6520) -- (1.4399,5.6662) -- (1.3986,5.6787) -- (1.3568,5.6896) -- (1.3146,5.6989) -- (1.2721,5.7065) -- (1.2293,5.7124) -- (1.1863,5.7166) -- (1.1432,5.7192) -- (1.1000,5.7200) -- cycle;
\path[draw=gluegreen,line width=0.65pt,fill=gluegreenfill] (1.1000,6.8200) -- (1.1208,6.8199) -- (1.1417,6.8196) -- (1.1625,6.8191) -- (1.1833,6.8184) -- (1.2042,6.8175) -- (1.2250,6.8164) -- (1.2458,6.8152) -- (1.2667,6.8137) -- (1.2875,6.8120) -- (1.3083,6.8101) -- (1.3292,6.8080) -- (1.3500,6.8057) -- (1.3708,6.8033) -- (1.3917,6.8006) -- (1.4125,6.7977) -- (1.4333,6.7946) -- (1.4542,6.7913) -- (1.4750,6.7878) -- (1.4958,6.7841) -- (1.5167,6.7802) -- (1.5375,6.7761) -- (1.5583,6.7717) -- (1.5792,6.7672) -- (1.6000,6.7624) -- (1.6000,5.5998) -- (1.5792,5.6102) -- (1.5583,5.6200) -- (1.5375,5.6293) -- (1.5167,5.6380) -- (1.4958,5.6463) -- (1.4750,5.6541) -- (1.4542,5.6614) -- (1.4333,5.6683) -- (1.4125,5.6747) -- (1.3917,5.6806) -- (1.3708,5.6861) -- (1.3500,5.6912) -- (1.3292,5.6959) -- (1.3083,5.7001) -- (1.2875,5.7039) -- (1.2667,5.7073) -- (1.2458,5.7103) -- (1.2250,5.7129) -- (1.2042,5.7151) -- (1.1833,5.7168) -- (1.1625,5.7182) -- (1.1417,5.7192) -- (1.1208,5.7198) -- (1.1000,5.7200) -- cycle;
\path[draw=gluegreen,line width=0.65pt,fill=gluegreenfill] (1.1000,2.4200) -- (1.1208,2.4201) -- (1.1417,2.4204) -- (1.1625,2.4209) -- (1.1833,2.4216) -- (1.2042,2.4225) -- (1.2250,2.4236) -- (1.2458,2.4248) -- (1.2667,2.4263) -- (1.2875,2.4280) -- (1.3083,2.4299) -- (1.3292,2.4320) -- (1.3500,2.4343) -- (1.3708,2.4367) -- (1.3917,2.4394) -- (1.4125,2.4423) -- (1.4333,2.4454) -- (1.4542,2.4487) -- (1.4750,2.4522) -- (1.4958,2.4559) -- (1.5167,2.4598) -- (1.5375,2.4639) -- (1.5583,2.4683) -- (1.5792,2.4728) -- (1.6000,2.4776) -- (1.6000,3.6402) -- (1.5792,3.6298) -- (1.5583,3.6200) -- (1.5375,3.6107) -- (1.5167,3.6020) -- (1.4958,3.5937) -- (1.4750,3.5859) -- (1.4542,3.5786) -- (1.4333,3.5717) -- (1.4125,3.5653) -- (1.3917,3.5594) -- (1.3708,3.5539) -- (1.3500,3.5488) -- (1.3292,3.5441) -- (1.3083,3.5399) -- (1.2875,3.5361) -- (1.2667,3.5327) -- (1.2458,3.5297) -- (1.2250,3.5271) -- (1.2042,3.5249) -- (1.1833,3.5232) -- (1.1625,3.5218) -- (1.1417,3.5208) -- (1.1208,3.5202) -- (1.1000,3.5200) -- cycle;
\path[draw=gluegray,line width=0.7pt] (1.1000,2.0700) -- (1.1000,7.1900);
\path[draw=gluegray,line width=0.7pt] (1.1850,7.0400) -- (1.1000,7.1900) -- (1.0150,7.0400);
\path[draw=gluegray,line width=0.7pt] (1.1000,4.6200) -- (4.5700,4.6200);
\path[draw=gluegray,line width=0.7pt] (4.4200,4.5350) -- (4.5700,4.6200) -- (4.4200,4.7050);
\path[draw=gluegreen,line width=2.8pt] (1.1000,5.7200) -- (1.1000,6.8200);
\path[draw=gluegreen,line width=2.8pt] (1.1000,3.5200) -- (1.1000,2.4200);
\node[anchor=east,align=center,inner sep=0pt,text=glueink,font={\fontsize{8}{9.6}\selectfont}] at (0.9600,6.8200) {$2R_k$};
\node[anchor=east,align=center,inner sep=0pt,text=glueink,font={\fontsize{8}{9.6}\selectfont}] at (0.9600,5.7200) {$R_k$};
\node[anchor=east,align=center,inner sep=0pt,text=glueink,font={\fontsize{8}{9.6}\selectfont}] at (0.9600,3.5200) {$-R_k$};
\node[anchor=east,align=center,inner sep=0pt,text=glueink,font={\fontsize{8}{9.6}\selectfont}] at (0.9600,2.4200) {$-2R_k$};
\path[draw=glueink,line width=0.3pt,fill=glueink] (1.1430,4.6200) -- (1.1415,4.6311) -- (1.1372,4.6415) -- (1.1304,4.6504) -- (1.1215,4.6572) -- (1.1111,4.6615) -- (1.1000,4.6630) -- (1.0889,4.6615) -- (1.0785,4.6572) -- (1.0696,4.6504) -- (1.0628,4.6415) -- (1.0585,4.6311) -- (1.0570,4.6200) -- (1.0585,4.6089) -- (1.0628,4.5985) -- (1.0696,4.5896) -- (1.0785,4.5828) -- (1.0889,4.5785) -- (1.1000,4.5770) -- (1.1111,4.5785) -- (1.1215,4.5828) -- (1.1304,4.5896) -- (1.1372,4.5985) -- (1.1415,4.6089) -- (1.1430,4.6200) -- cycle;
\node[anchor=east,align=center,inner sep=0pt,text=glueink,font={\fontsize{8}{9.6}\selectfont}] at (0.9800,4.4500) {$p_k$};
\node[anchor=center,align=center,inner sep=0pt,text=glueink,font={\fontsize{8}{9.6}\selectfont}] at (1.3500,7.2400) {$y-y_k$};
\node[anchor=center,align=center,inner sep=0pt,text=glueink,font={\fontsize{8}{9.6}\selectfont}] at (4.4300,4.4000) {$\sqrt2\,t$};
\node[anchor=center,align=center,inner sep=0pt,text=glueblue,font={\fontsize{8.5}{10.2}\selectfont}] at (3.8000,6.2000) {Prepared\\background\\$v_b$};
\node[anchor=center,align=center,inner sep=0pt,text=gluegreen,font={\fontsize{8.5}{10.2}\selectfont}] at (3.6000,7.0400) {Upper polar\\region};
\path[draw=gluegreen,line width=0.55pt] (2.5700,6.8800) -- (1.3400,6.3100);
\path[draw=gluegreen,line width=0.55pt] (1.4404,6.4502) -- (1.3400,6.3100) -- (1.5118,6.2959);
\node[anchor=center,align=center,inner sep=0pt,text=glueorange,font={\fontsize{8.5}{10.2}\selectfont}] at (3.9000,5.3500) {Exact core\\$u_k$};
\path[draw=glueorange,line width=0.55pt] (3.0000,5.2200) -- (1.5500,4.9800);
\path[draw=glueorange,line width=0.55pt] (1.6841,5.0884) -- (1.5500,4.9800) -- (1.7119,4.9206);
\node[anchor=center,align=center,inner sep=0pt,text=gluepurple,font={\fontsize{8.5}{10.2}\selectfont}] at (3.8000,3.6800) {Regular\\region};
\path[draw=gluepurple,line width=0.55pt] (3.2600,3.9700) -- (2.7000,4.1400);
\path[draw=gluepurple,line width=0.55pt] (2.8682,4.1778) -- (2.7000,4.1400) -- (2.8188,4.0151);
\node[anchor=center,align=center,inner sep=0pt,text=gluegreen,font={\fontsize{8.5}{10.2}\selectfont}] at (3.6500,2.6000) {Lower polar\\region};
\path[draw=gluegreen,line width=0.55pt] (2.7200,2.9100) -- (1.3400,2.9300);
\path[draw=gluegreen,line width=0.55pt] (1.4912,3.0128) -- (1.3400,2.9300) -- (1.4888,2.8428);
\node[anchor=center,align=center,inner sep=0pt,text=glueink,font={\fontsize{8.5}{10.2}\selectfont}] at (2.5200,1.6700) {Attachment annulus:\\$R_k<\rho_k<2R_k$};
\node[anchor=center,align=center,inner sep=0pt,text=glueink,font={\fontsize{9}{10.799999999999999}\selectfont}] at (2.5200,0.7200) {$\rho_k=\sqrt{2t^2+(y-y_k)^2}$};
\end{tikzpicture}
\endgroup
}
\caption{The gluing neighborhood in the parameter half-plane
$(\sqrt2t,y-y_k)$ (schematic; scales are exaggerated). The semicircles
are level sets of $\rho_k$, rather than ambient distance.
The orange region contains the exact connector core. The attachment
annulus consists of the green upper and lower polar regions near
$t=0$ and the purple regular region away from the axis.
The thick green segments on $t=0$ are the axial matching intervals.
Outside the outer semicircle, the construction retains the prepared
background $v_b$. It agrees with $\widehat v=V_0[a]+f$ in the
inner transverse tube and is cut back to the cone outside that tube;
these profiles are defined below.}
\label{fig:modulated-foliation-gluing}
\end{figure}
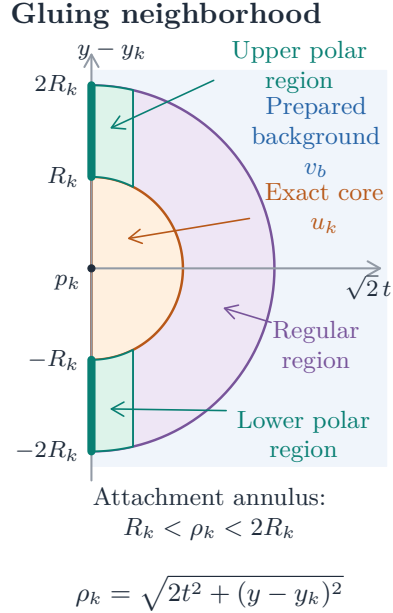

Choose $\eta_k$ supported in
$R_k/4<|y-y_k|<4R_k$, equal to one on
$R_k/2\le|y-y_k|\le3R_k$, with
$|\eta_k^{(n)}|\le C_nR_k^{-n}$. For $y\ne0$, set
\begin{equation}\label{bg:eq:radius}
 a(y)=e^{-1/y^2}+\sum_k\eta_k(y)
             \bigl(u_k(0,y)-e^{-1/y^2}\bigr),
 \qquad \mathcal R(y)=c_*|y|^4.
\end{equation}
Each summand is extended by zero outside its support. Since
$R_k=O(y_k^4)$, the functions $e^{-1/y^2}$ and
$e^{-1/y_k^2}$ are comparable on the cutoff supports. The cap
estimates therefore imply
\begin{equation}\label{eq:soft-cap-data}
 a(y)\approx e^{-1/y^2},\qquad
 |a^{(n)}(y)|\le C_n a(y)\mathcal R(y)^{-n}.
\end{equation}
In particular $a$ extends smoothly and flatly by $a(0)=0$.
On the polar matching intervals it equals the cap radius of the
corresponding connector. We have chosen the insertion data before
the background: the function $f$ will be adapted to these data.

\subsection{The transverse correction}
We now leave $f$ unspecified and write
\begin{equation}\label{bg:eq:corrected-background}
 \widehat v(t,y)=V_0[a](t,y)+f(t,y),\qquad s=a(y)+t.
\end{equation}
We impose $f(0,y)=f_t(0,y)=0$ to preserve the prescribed cap
radius and the radial first-derivative condition. The correction
changes how the profile opens away from the axis.
Here $f$ denotes a correction to the SME radius profile, distinct
from the normal graph heights used earlier in the paper.
Figure~\ref{fig:axis-radius-matching} shows the prescribed axis
data: the background stays positive at $y_k$, while the inserted
connector vanishes there. Agreement on the matching intervals
fixes the axis value; $f$ must also control the transverse error.

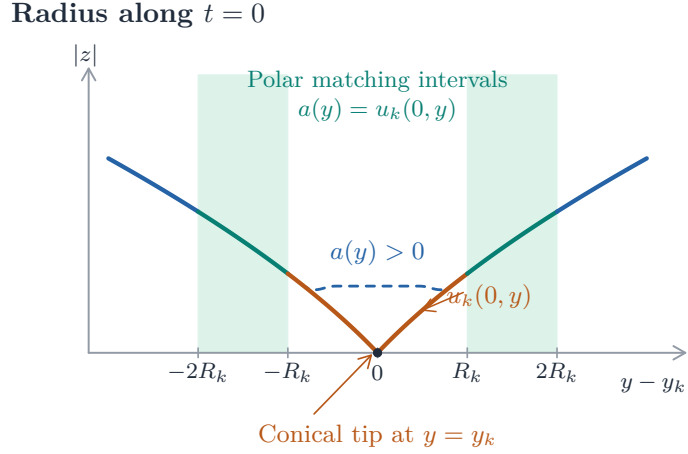
\begin{figure}[htbp]
\centering
\scalebox{1.05}{%
\begingroup
\definecolor{glueink}{HTML}{202C3B}
\definecolor{glueblue}{HTML}{2563A6}
\definecolor{glueorange}{HTML}{B85818}
\definecolor{gluepurple}{HTML}{78549A}
\definecolor{gluegray}{HTML}{929AA5}
\definecolor{gluelight}{HTML}{DCE2E9}
\definecolor{gluebluefill}{HTML}{F0F5FB}
\definecolor{glueorangefill}{HTML}{FFF0DF}
\definecolor{gluepurplefill}{HTML}{EEE6F5}
\definecolor{gluewhite}{HTML}{FFFFFF}
\definecolor{gluegreen}{HTML}{087F73}
\definecolor{gluegreenfill}{HTML}{DDF3EA}
\begin{tikzpicture}[x=1cm,y=1cm,line join=round,line cap=round]
\path[use as bounding box] (5.8,0.3) rectangle (15.25,8.0);
\node[anchor=west,align=center,inner sep=0pt,text=glueink,font={\fontsize{10}{12.0}\selectfont\bfseries}] at (6.0000,7.6600) {Radius along $t=0$};
\path[draw=gluegreenfill,line width=0.1pt,fill=gluegreenfill] (8.3600,3.4000) -- (9.4900,3.4000) -- (9.4900,6.9000) -- (8.3600,6.9000) -- cycle;
\path[draw=gluegreenfill,line width=0.1pt,fill=gluegreenfill] (11.7500,3.4000) -- (12.8800,3.4000) -- (12.8800,6.9000) -- (11.7500,6.9000) -- cycle;
\path[draw=gluegray,line width=0.7pt] (6.9814,3.4000) -- (14.4959,3.4000);
\path[draw=gluegray,line width=0.7pt] (14.3459,3.3150) -- (14.4959,3.4000) -- (14.3459,3.4850);
\path[draw=gluegray,line width=0.7pt] (6.9814,3.4000) -- (6.9814,6.9800);
\path[draw=gluegray,line width=0.7pt] (7.0664,6.8300) -- (6.9814,6.9800) -- (6.8964,6.8300);
\node[anchor=center,align=center,inner sep=0pt,text=glueink,font={\fontsize{8}{9.6}\selectfont}] at (6.9414,7.1500) {$|z|$};
\node[anchor=east,align=center,inner sep=0pt,text=glueink,font={\fontsize{8}{9.6}\selectfont}] at (14.5185,3.0200) {$y-y_k$};
\path[draw=glueblue,line width=1.0pt,dashed] (9.4900,4.3971) -- (9.5183,4.3754) -- (9.5465,4.3531) -- (9.5747,4.3301) -- (9.6030,4.3069) -- (9.6312,4.2841) -- (9.6595,4.2624) -- (9.6877,4.2427) -- (9.7160,4.2258) -- (9.7442,4.2121) -- (9.7725,4.2022) -- (9.8007,4.1962) -- (9.8290,4.1940) -- (9.8572,4.1954) -- (9.8855,4.1996) -- (9.9137,4.2060) -- (9.9420,4.2136) -- (9.9703,4.2215) -- (9.9985,4.2286) -- (10.0267,4.2343) -- (10.0550,4.2380) -- (10.0832,4.2397) -- (10.1115,4.2400) -- (10.1397,4.2400) -- (10.1680,4.2400) -- (10.1962,4.2400) -- (10.2245,4.2400) -- (10.2527,4.2400) -- (10.2810,4.2400) -- (10.3092,4.2400) -- (10.3375,4.2400) -- (10.3657,4.2400) -- (10.3940,4.2400) -- (10.4223,4.2400) -- (10.4505,4.2400) -- (10.4787,4.2400) -- (10.5070,4.2400) -- (10.5352,4.2400) -- (10.5635,4.2400) -- (10.5917,4.2400) -- (10.6200,4.2400) -- (10.6482,4.2400) -- (10.6765,4.2400) -- (10.7047,4.2400) -- (10.7330,4.2400) -- (10.7612,4.2400) -- (10.7895,4.2400) -- (10.8177,4.2400) -- (10.8460,4.2400) -- (10.8742,4.2400) -- (10.9025,4.2400) -- (10.9307,4.2400) -- (10.9590,4.2400) -- (10.9872,4.2400) -- (11.0155,4.2400) -- (11.0437,4.2400) -- (11.0720,4.2400) -- (11.1002,4.2400) -- (11.1285,4.2400) -- (11.1567,4.2397) -- (11.1850,4.2380) -- (11.2132,4.2343) -- (11.2415,4.2286) -- (11.2697,4.2215) -- (11.2980,4.2136) -- (11.3262,4.2060) -- (11.3545,4.1996) -- (11.3827,4.1954) -- (11.4110,4.1940) -- (11.4392,4.1962) -- (11.4675,4.2022) -- (11.4957,4.2121) -- (11.5240,4.2258) -- (11.5522,4.2427) -- (11.5805,4.2624) -- (11.6087,4.2841) -- (11.6370,4.3069) -- (11.6652,4.3301) -- (11.6935,4.3531) -- (11.7218,4.3754) -- (11.7500,4.3971);
\path[draw=glueblue,line width=1.55pt] (7.2300,5.8484) -- (7.2582,5.8326) -- (7.2865,5.8167) -- (7.3148,5.8008) -- (7.3430,5.7849) -- (7.3712,5.7689) -- (7.3995,5.7528) -- (7.4277,5.7368) -- (7.4560,5.7206) -- (7.4842,5.7045) -- (7.5125,5.6883) -- (7.5407,5.6720) -- (7.5690,5.6557) -- (7.5972,5.6393) -- (7.6255,5.6229) -- (7.6537,5.6065) -- (7.6820,5.5900) -- (7.7102,5.5734) -- (7.7385,5.5568) -- (7.7668,5.5402) -- (7.7950,5.5235) -- (7.8232,5.5067) -- (7.8515,5.4899) -- (7.8797,5.4730) -- (7.9080,5.4561) -- (7.9362,5.4391) -- (7.9645,5.4221) -- (7.9927,5.4050) -- (8.0210,5.3879) -- (8.0492,5.3707) -- (8.0775,5.3535) -- (8.1058,5.3361) -- (8.1340,5.3188) -- (8.1623,5.3014) -- (8.1905,5.2839) -- (8.2188,5.2663) -- (8.2470,5.2487) -- (8.2752,5.2310) -- (8.3035,5.2133) -- (8.3317,5.1955) -- (8.3600,5.1776);
\path[draw=gluegreen,line width=1.55pt] (8.3600,5.1776) -- (8.3882,5.1597) -- (8.4165,5.1417) -- (8.4447,5.1236) -- (8.4730,5.1055) -- (8.5012,5.0873) -- (8.5295,5.0690) -- (8.5577,5.0507) -- (8.5860,5.0322) -- (8.6142,5.0137) -- (8.6425,4.9952) -- (8.6707,4.9765) -- (8.6990,4.9578) -- (8.7272,4.9390) -- (8.7555,4.9201) -- (8.7837,4.9012) -- (8.8120,4.8821) -- (8.8402,4.8630) -- (8.8685,4.8438) -- (8.8967,4.8245) -- (8.9250,4.8051) -- (8.9532,4.7857) -- (8.9815,4.7661) -- (9.0097,4.7465) -- (9.0380,4.7267) -- (9.0663,4.7069) -- (9.0945,4.6870) -- (9.1227,4.6670) -- (9.1510,4.6468) -- (9.1792,4.6266) -- (9.2075,4.6063) -- (9.2357,4.5859) -- (9.2640,4.5653) -- (9.2922,4.5447) -- (9.3205,4.5240) -- (9.3487,4.5031) -- (9.3770,4.4821) -- (9.4052,4.4611) -- (9.4335,4.4398) -- (9.4617,4.4185) -- (9.4900,4.3971);
\path[draw=glueorange,line width=1.55pt] (9.4900,4.3971) -- (9.5183,4.3755) -- (9.5465,4.3538) -- (9.5747,4.3320) -- (9.6030,4.3100) -- (9.6312,4.2879) -- (9.6595,4.2657) -- (9.6877,4.2434) -- (9.7160,4.2208) -- (9.7442,4.1982) -- (9.7725,4.1754) -- (9.8007,4.1524) -- (9.8290,4.1293) -- (9.8572,4.1060) -- (9.8855,4.0826) -- (9.9137,4.0590) -- (9.9420,4.0352) -- (9.9703,4.0113) -- (9.9985,3.9872) -- (10.0267,3.9629) -- (10.0550,3.9384) -- (10.0832,3.9137) -- (10.1115,3.8888) -- (10.1397,3.8637) -- (10.1680,3.8384) -- (10.1962,3.8129) -- (10.2245,3.7872) -- (10.2527,3.7612) -- (10.2810,3.7350) -- (10.3092,3.7086) -- (10.3375,3.6819) -- (10.3657,3.6550) -- (10.3940,3.6278) -- (10.4223,3.6004) -- (10.4505,3.5727) -- (10.4787,3.5446) -- (10.5070,3.5163) -- (10.5352,3.4877) -- (10.5635,3.4588) -- (10.5917,3.4296) -- (10.6200,3.4000);
\path[draw=glueorange,line width=1.55pt] (10.6200,3.4000) -- (10.6482,3.4296) -- (10.6765,3.4588) -- (10.7047,3.4877) -- (10.7330,3.5163) -- (10.7612,3.5446) -- (10.7895,3.5727) -- (10.8177,3.6004) -- (10.8460,3.6278) -- (10.8742,3.6550) -- (10.9025,3.6819) -- (10.9307,3.7086) -- (10.9590,3.7350) -- (10.9872,3.7612) -- (11.0155,3.7872) -- (11.0437,3.8129) -- (11.0720,3.8384) -- (11.1002,3.8637) -- (11.1285,3.8888) -- (11.1567,3.9137) -- (11.1850,3.9384) -- (11.2132,3.9629) -- (11.2415,3.9872) -- (11.2697,4.0113) -- (11.2980,4.0352) -- (11.3262,4.0590) -- (11.3545,4.0826) -- (11.3827,4.1060) -- (11.4110,4.1293) -- (11.4392,4.1524) -- (11.4675,4.1754) -- (11.4957,4.1982) -- (11.5240,4.2208) -- (11.5522,4.2434) -- (11.5805,4.2657) -- (11.6087,4.2879) -- (11.6370,4.3100) -- (11.6652,4.3320) -- (11.6935,4.3538) -- (11.7217,4.3755) -- (11.7500,4.3971);
\path[draw=gluegreen,line width=1.55pt] (11.7500,4.3971) -- (11.7782,4.4185) -- (11.8065,4.4398) -- (11.8347,4.4611) -- (11.8630,4.4821) -- (11.8912,4.5031) -- (11.9195,4.5240) -- (11.9477,4.5447) -- (11.9760,4.5653) -- (12.0042,4.5859) -- (12.0325,4.6063) -- (12.0607,4.6266) -- (12.0890,4.6468) -- (12.1172,4.6670) -- (12.1455,4.6870) -- (12.1737,4.7069) -- (12.2020,4.7267) -- (12.2302,4.7465) -- (12.2585,4.7661) -- (12.2867,4.7857) -- (12.3150,4.8051) -- (12.3432,4.8245) -- (12.3715,4.8438) -- (12.3997,4.8630) -- (12.4280,4.8821) -- (12.4562,4.9012) -- (12.4845,4.9201) -- (12.5127,4.9390) -- (12.5410,4.9578) -- (12.5692,4.9765) -- (12.5975,4.9952) -- (12.6258,5.0137) -- (12.6540,5.0322) -- (12.6822,5.0507) -- (12.7105,5.0690) -- (12.7387,5.0873) -- (12.7670,5.1055) -- (12.7952,5.1236) -- (12.8235,5.1417) -- (12.8517,5.1597) -- (12.8800,5.1776);
\path[draw=glueblue,line width=1.55pt] (12.8800,5.1776) -- (12.9082,5.1955) -- (12.9365,5.2133) -- (12.9647,5.2310) -- (12.9930,5.2487) -- (13.0212,5.2663) -- (13.0495,5.2839) -- (13.0777,5.3014) -- (13.1060,5.3188) -- (13.1342,5.3361) -- (13.1625,5.3535) -- (13.1907,5.3707) -- (13.2190,5.3879) -- (13.2472,5.4050) -- (13.2755,5.4221) -- (13.3037,5.4391) -- (13.3320,5.4561) -- (13.3602,5.4730) -- (13.3885,5.4899) -- (13.4167,5.5067) -- (13.4450,5.5235) -- (13.4732,5.5402) -- (13.5015,5.5568) -- (13.5297,5.5734) -- (13.5580,5.5900) -- (13.5862,5.6065) -- (13.6145,5.6229) -- (13.6427,5.6393) -- (13.6710,5.6557) -- (13.6992,5.6720) -- (13.7275,5.6883) -- (13.7557,5.7045) -- (13.7840,5.7206) -- (13.8122,5.7368) -- (13.8405,5.7528) -- (13.8687,5.7689) -- (13.8970,5.7849) -- (13.9252,5.8008) -- (13.9535,5.8167) -- (13.9817,5.8326) -- (14.0100,5.8484);
\path[draw=gluegray,line width=0.6pt] (8.3600,3.3400) -- (8.3600,3.4600);
\node[anchor=center,align=center,inner sep=0pt,text=glueink,font={\fontsize{8}{9.6}\selectfont}] at (8.3600,3.1500) {$-2R_k$};
\path[draw=gluegray,line width=0.6pt] (9.4900,3.3400) -- (9.4900,3.4600);
\node[anchor=center,align=center,inner sep=0pt,text=glueink,font={\fontsize{8}{9.6}\selectfont}] at (9.4900,3.1500) {$-R_k$};
\path[draw=gluegray,line width=0.6pt] (10.6200,3.3400) -- (10.6200,3.4600);
\node[anchor=center,align=center,inner sep=0pt,text=glueink,font={\fontsize{8}{9.6}\selectfont}] at (10.6200,3.1500) {$0$};
\path[draw=gluegray,line width=0.6pt] (11.7500,3.3400) -- (11.7500,3.4600);
\node[anchor=center,align=center,inner sep=0pt,text=glueink,font={\fontsize{8}{9.6}\selectfont}] at (11.7500,3.1500) {$R_k$};
\path[draw=gluegray,line width=0.6pt] (12.8800,3.3400) -- (12.8800,3.4600);
\node[anchor=center,align=center,inner sep=0pt,text=glueink,font={\fontsize{8}{9.6}\selectfont}] at (12.8800,3.1500) {$2R_k$};
\path[draw=glueink,line width=0.3pt,fill=glueink] (10.6720,3.4000) -- (10.6702,3.4135) -- (10.6650,3.4260) -- (10.6568,3.4368) -- (10.6460,3.4450) -- (10.6335,3.4502) -- (10.6200,3.4520) -- (10.6065,3.4502) -- (10.5940,3.4450) -- (10.5832,3.4368) -- (10.5750,3.4260) -- (10.5698,3.4135) -- (10.5680,3.4000) -- (10.5698,3.3865) -- (10.5750,3.3740) -- (10.5832,3.3632) -- (10.5940,3.3550) -- (10.6065,3.3498) -- (10.6200,3.3480) -- (10.6335,3.3498) -- (10.6460,3.3550) -- (10.6568,3.3632) -- (10.6650,3.3740) -- (10.6702,3.3865) -- (10.6720,3.4000) -- cycle;
\node[anchor=center,align=center,inner sep=0pt,text=glueblue,font={\fontsize{9}{10.799999999999999}\selectfont}] at (10.6200,4.6600) {$a(y)>0$};
\node[anchor=center,align=center,inner sep=0pt,text=glueorange,font={\fontsize{9}{10.799999999999999}\selectfont}] at (12.0300,4.1100) {$u_k(0,y)$};
\path[draw=glueorange,line width=0.55pt] (11.7000,4.1600) -- (11.1850,3.9384);
\path[draw=glueorange,line width=0.55pt] (11.2892,4.0757) -- (11.1850,3.9384) -- (11.3564,3.9196);
\node[anchor=center,align=center,inner sep=0pt,text=glueorange,font={\fontsize{9}{10.799999999999999}\selectfont}] at (10.6200,2.3500) {Conical tip at $y=y_k$};
\path[draw=glueorange,line width=0.6pt] (9.8200,2.6300) -- (10.5400,3.3700);
\path[draw=glueorange,line width=0.6pt] (10.4963,3.2032) -- (10.5400,3.3700) -- (10.3745,3.3218);
\node[anchor=center,align=center,inner sep=0pt,text=gluegreen,font={\fontsize{8.5}{10.2}\selectfont}] at (10.6200,6.6300) {Polar matching intervals\\$a(y)=u_k(0,y)$};
\node[anchor=center,align=center,inner sep=0pt,text=glueink,font={\fontsize{8.5}{10.2}\selectfont}] at (10.6200,1.4200) {$\widehat v(0,y)=a(y)$: $f$ preserves the axis value.};
\node[anchor=center,align=center,inner sep=0pt,text=glueink,font={\fontsize{8.5}{10.2}\selectfont}] at (10.6200,0.7000) {It improves the transverse profile and mean curvature.};
\end{tikzpicture}
\endgroup
}
\caption{Axis radii before and after inserting the connector.
The horizontal coordinate is $y-y_k$ and the vertical coordinate
is the radius $|z|$ at $t=0$. The dashed blue curve is the positive
background radius $a(y)$; the orange curve is $u_k(0,y)$ in the
exact core and vanishes conically at $y_k$. On the green matching
intervals $R_k\le|y-y_k|\le2R_k$, the two radii agree.
The solid blue portions represent the retained background.
Matching these values does not match the full transverse profiles.
The correction $f$ preserves the axis value and improves the
transverse profile and mean curvature. All curves are schematic.}
\label{fig:axis-radius-matching}
\end{figure}

The two essential requirements are small mean curvature and
accurate matching to the connectors. The following formulation
records them without mentioning a correction order.

\begin{definition}[An admissible transverse correction]\label{bg:def:admissible}
On a sufficiently small punctured neighborhood of $y=0$,
the function $f$ is smooth, with a smooth radial
extension in $x$, and satisfies $f(0,y)=f_t(0,y)=0$.
On a tube $0\le t<C\mathcal R(y)$, the profile $\widehat v$ satisfies
\begin{equation}\label{bg:eq:geometry}
 \widehat v\approx s,\qquad |D^m\widehat v|\le C_m s^{1-m}\quad(m\ge1).
\end{equation}
Here and below $D$ denotes Cartesian base derivatives in $(x,y)$;
radial formulas at $x=0$ are interpreted by their smooth extension.
For every $L,m,N\ge0$, its residual obeys
\begin{equation}\label{bg:eq:admissible-residual}
 |D^m Q(\widehat v)|\le C_{L,m,N}|y|^N s^{L-m}.
\end{equation}
Fix small constants $0<c<c_1$ so that $t\le cR_k$ lies inside
$t<c_1\mathcal R(y)$ on each polar attachment region. On
\[
 R_k\le\sqrt{2t^2+(y-y_k)^2}\le2R_k,\qquad t\le cR_k,
\]
we require, for every $L,m,N\ge0$,
\begin{equation}\label{bg:eq:admissible-matching}
 |D^m(u_k-\widehat v)|\le C_{L,m,N}|y|^N s^{L+2-m}.
\end{equation}
The constants are independent of $k$. Finally, throughout each
fixed transverse collar $\theta\mathcal R\le t\le C\mathcal R$,
$\theta>0$, we require
\begin{equation}\label{bg:eq:admissible-collar}
 |D^m(\widehat v-t)|\le C_{\theta,m,N}|y|^N\qquad(m,N\ge0).
\end{equation}
\end{definition}

The quantifiers in \eqref{bg:eq:admissible-residual} and
\eqref{bg:eq:admissible-matching} are useful: every fixed power lost
by differentiating at the small scale $s$ can be absorbed by
choosing $L$ larger. The annular estimate has no such small-scale
issue because $s\approx\mathcal R$ there and $\mathcal R$ is a
fixed power of $|y|$.

\begin{proposition}[Choice of the correction]\label{bg:prop:choose-f}
For the radius function \(a\) and scaled connector profiles \(u_k\) fixed in \eqref{bg:eq:insertion-data}--\eqref{bg:eq:radius}, there exists an admissible
transverse correction $f$.
\end{proposition}

We prove this proposition after the geometric construction.
No smooth extension of $f$ to $(x,y)=0$ is asserted. Its high
transverse derivatives may grow in the shrinking caps. The smooth
flat object obtained below is the ambient metric exponent. Given such a correction we can conclude the proof of Theorem \ref{thm:main-accumulation}.

\begin{proof}[Proof of Theorem~\ref{thm:main-accumulation}]
Take an admissible correction given by
Proposition~\ref{bg:prop:choose-f}. We first make its background
minimal in a smooth conformal metric. We then insert the connectors
and update that same metric formula locally. The estimates in
Definition~\ref{bg:def:admissible} control both stages.

\smallskip\noindent\textbf{Step 1: Prepare the background and its metric.}
Choose a smooth cutoff $\chi$ equal to one on $[0,1]$ and zero
on $[2,\infty)$. Given an admissible $f$, set
\begin{equation}\label{bg:eq:cut-background}
 v_b(t,y)=t+\chi\!\left(\frac{t}{c_1\mathcal R(y)}\right)
                  \bigl(\widehat v(t,y)-t\bigr)\quad(y\ne0),
 \qquad v_b(t,0)=t.
\end{equation}
The background $S(v_b)$ retains the prescribed cap $a(y)$ and
agrees with the cone outside the narrow tube. Its slope is
uniformly bounded. It is smooth away from the origin.

We use the explicit conformal extension from
Lemma~\ref{lem:conformal-correction} for both the background and
the surface after insertion. For any positive profile $v$,
let $H_v$ be its Euclidean scalar mean curvature. Fix
$\vartheta\in C_c^\infty((-1,1))$ equal to one near zero and define
\begin{equation}\label{bg:eq:extension}
 F[v](x,z,y)=
 -\frac{H_v(t,y)}{8\sqrt{1+v_t^2+v_y^2}}\,
       (|z|-v(t,y))
       \vartheta\!\left(\frac{4(|z|-v(t,y))}{v(t,y)}\right).
\end{equation}
A common outer cutoff, equal to one on the working neighborhood,
may be included. The radial support avoids $z=0$, and the smooth
even dependence in $x$ makes the expression smooth at $x=0$.
On $S(v)$,
\begin{equation}\label{bg:eq:normal-extension}
 F[v]=0,\qquad \nu_vF[v]=-H_v/8,\qquad
 H_{S(v),e^{2F[v]}g_{\rm euc}}=0.
\end{equation}
The formula is zero wherever $H_v=0$.

Here is the derivative estimate that explains the choice of
conditions on $f$. On a region with geometric scale $s=a+t$,
the profile bounds \eqref{bg:eq:geometry} and differentiation of
\eqref{bg:eq:extension} give
\begin{equation}\label{bg:eq:extension-estimate}
 |D^mF[v]|\le C_m\sum_{j=0}^m
                 s^{1-m+j}|D^jH_v|.
\end{equation}
Derivatives on the left are ambient Cartesian derivatives;
those on the right are base derivatives of the scalar curvature.
The factor $s$ comes from $|z|-v$ on the support, and each
derivative of a profile coefficient or cutoff costs $s^{-1}$.

Inside the tube, \eqref{bg:eq:admissible-residual} gives
$|D^jH_{\widehat v}|\le C|y|^Ns^{L-j}$. Taking $L\ge m-1$ in
\eqref{bg:eq:extension-estimate} controls the $m$th metric derivative
by $C|y|^N$. On the transverse cutoff collar,
\eqref{bg:eq:admissible-collar} gives the same conclusion: cutoff
derivatives cost only powers of $\mathcal R^{-1}$, which are
absorbed by increasing $N$ in that estimate.
Consequently
\begin{equation}\label{bg:eq:background-flat}
 F_b:=F[v_b],\qquad
 |D^mF_b(x,z,y)|\le C_{m,N}|y|^N\quad(m,N\ge0).
\end{equation}
Extend $F_b$ by zero at the origin and wherever the background
is exactly the cone. It is smooth and flat at zero, and
$g_b=e^{2F_b}g_{\rm euc}$ makes the prepared background minimal.

\smallskip\noindent\textbf{Step 2: Insert the connectors and update the metric.}
With the same cutoff $\chi$, define
\begin{equation}\label{bg:eq:final-profile}
 v=v_b+\sum_k
  \chi\!\left(\frac{\sqrt{2t^2+(y-y_k)^2}}{R_k}\right)
                       (u_k-v_b).
\end{equation}
The supports are disjoint. The surface $\Sigma=S(v)$ agrees
with the exact connector in each inner region and with the
prepared background outside the insertion regions.

On a polar attachment region, the transverse cutoff in
\eqref{bg:eq:cut-background} is one. The matching estimate
\eqref{bg:eq:admissible-matching} loses two derivatives when the
curvature operator is applied. Its remaining bound is therefore
of the form \eqref{bg:eq:admissible-residual}. Differentiating the
interpolation cutoff costs $R_k^{-1}\le C s^{-1}$, so these terms
obey the same bound. The positive profiles and their first
derivatives stay in the controlled range in
\eqref{bg:eq:geometry}, after restricting to sufficiently large $k$.

On the rest of the attachment annulus, $t\approx R_k$.
The connector's regular-sector estimates give
\[
 |D^m(u_k-t)|\le C_m\lambda_kR_k^{-m}
       \le C_m e^{-3/y_k^2}R_k^{-2-m}.
\]
Together with \eqref{bg:eq:admissible-collar}, this controls the
interpolation and every curvature derivative by every prescribed
power of $|y|$. Since $s\approx R_k$ here, the metric extension
again has the bounds \eqref{bg:eq:background-flat}. Inside the
exact connector regions its curvature is zero, so the extension
vanishes identically; there is no derivative cost at scale
$\lambda_k$.

It follows that $F=F[v]$, defined by \eqref{bg:eq:extension} off
the singular points and extended by zero on the exact cores, is
smooth away from zero and satisfies
\begin{equation}\label{bg:eq:final-flat}
 |D^mF(x,z,y)|\le C_{m,N}|y|^N\quad(m,N\ge0).
\end{equation}
It extends smoothly and flatly at zero. The final metric is
$g=e^{2F}g_{\rm euc}$.

The background-first metric construction can be stated literally:
\begin{equation}\label{bg:eq:metric-update}
 g=e^{2\Delta F}g_b,\qquad \Delta F=F[v]-F[v_b].
\end{equation}
Use the same cutoffs in the two extensions. Locality of
\eqref{bg:eq:extension} gives $\Delta F=0$ where the profiles agree
on a base neighborhood. Thus the metric update is supported
over the insertion regions. On an exact connector core,
$F[v]=0$ and $\Delta F=-F_b$: the final metric is Euclidean
there, although the preliminary metric need not be. The estimates
for $\Delta F$ follow from \eqref{bg:eq:background-flat} and
\eqref{bg:eq:final-flat}. No independent patching of metric pieces
is needed.

\smallskip\noindent\textbf{Step 3: Mass bounds and stationarity.} For completeness, the varifold conclusions follow from the same
localization. The modified background has mass in $B_r$ bounded by
\[
 C\int_{|y|<r}\int_0^{C|y|^4}t^3(t+a(y))^3\,dt\,dy
       \le Cr^{29}.
\]
Each inserted region has mass at most $CR_k^8$, using the bounded
slopes on its annulus and the connector's finite density at
infinity on its core. The total contribution near zero is
$C\sum_{y_k\le2r}R_k^8\le Cr^{32}$. The removed cone pieces
have the same negligible bounds. Hence the unique tangent at zero
is $C_{3,3}\times\R$ with multiplicity one, and
$\mu_\Sigma(B_r)\le Cr^8$. Exact agreement on smaller ambient
balls about $p_k$ gives the unique $C_{4,3}$ tangent there.
There are no other singularities.

Stationarity holds on the regular part by
\eqref{bg:eq:normal-extension}, and at each $p_k$ by the exact
Euclidean connector. Removing $B_r(0)$ with a cutoff whose
derivative is $O(r^{-1})$ gives a first-variation error
$O(r^8+r^7)\to0$. This proves stationarity also at zero.
Finally, \eqref{bg:eq:final-flat} implies
$D^m(g-g_{\rm euc})(0)=0$ for all $m$. This proves
the theorem after restricting to a sufficiently small open ball.
Indeed, the same bounds give
$\|g-g_{\rm euc}\|_{C^m(B_\rho)}\le C_{m,N}\rho^N$
for every $m,N$.
\end{proof}

\subsection{Successive corrections around the foliation and the choice of \texorpdfstring{$f$}{f}}\label{sec:finite-correctors}
We establish two estimates with different roles.
Proposition~\ref{prop:correctors-informal} summarizes the control of the corrected
background and its remaining mean curvature for a slowly varying
radius function. Corollary~\ref{cor:arbitrary-order-matching} shows that,
when the radius function comes from the connector, the same profiles approximate
that connector to the corresponding order. The first estimate
controls the region between insertions; the second controls the
attachment annuli. In both, the axis value is kept fixed.

We construct the finite-order profiles used below to choose the transverse
correction $f$. At each fixed $y$, we retain the prescribed axis value
$a(y)$ and solve a radial linearized equation around the leaf
$a(y)h(t/a(y))$. Successive solutions give the approximations $V_J[a]$
and cancel their mean-curvature errors through increasing even orders.
When $a$ is the connector radius, these profiles also approximate the
connector to the corresponding order. The construction is a finite
expansion in the normalized axial jets $q_j=a^{j-1}a^{(j)}$, with
smooth radial coefficients in $\rho=t/a(y)$ determined by the
linearized Hardt--Simon equation. No convergence of an infinite
series is required.

A related scheme for successively improving approximate
constant-mean-curvature hypersurfaces by solving model linearized
equations appears in Mahmoudi--Mazzeo--Pacard \cite{MMP2006}.

At order $J$, the estimates stated below and proved in
Appendix~\ref{app:corrector-proofs}, together with the conformal
correction lemma, yield
\begin{equation}\label{eq:gluing-metric-budget}
 |D^mF|\le C_{J,m}\left(
      a^3\mathcal R^{-m-3}
     +a^{2J+2-m}\mathcal R^{-2J-2}\right).
\end{equation}
The second term measures the cost near the axis. In the uncorrected
interpolation that cost has the scaling $a^{2-m}R_k^{-2}$;
successive corrections supply the additional positive powers of
$a/\mathcal R$. For each fixed $m$, taking $2J+2>m$ makes both
terms flat in $y$. The diagonal choice below increases $J$ as $y$ tends to zero,
while retaining exact connector cores. This is the mechanism
that gives one smooth metric with the full Euclidean jet at zero.

We now define the correctors and state the estimates used in
this construction. The scalar mean curvature is
$H=-Q(v)/\sqrt{1+|Dv|^2}$, with $Q$ as in \eqref{eq:gluing-Q}.
Set $\rho=t/a(y)$, $q_j=a^{j-1}a^{(j)}$, and
\[
 \mathcal I_d
 =\left\{\gamma\in\mathbb N_0^d:
          \sum_{j=1}^d j\gamma_j=d\right\},
 \qquad q^\gamma=\prod_{j=1}^d q_j^{\gamma_j}.
\]
The initial approximation is
\[
 V_0[a](t,y)=a(y)h\!\left(\frac{t}{a(y)}\right).
\]

Suppose $V_{J-1}[a]$ has been constructed, and write
$v=V_{J-1}[a]$. For a vertical perturbation
$v_\tau=v+\tau\varphi$, its mean curvature, viewed as a
function of $(t,y)$, is
\[
 H_{v_\tau}
 =-\frac{Q(v_\tau)}{\sqrt{1+|Dv_\tau|^2}}.
\]
Consequently,
\[
 \left.\frac{d}{d\tau}H_{v+\tau\varphi}\right|_{\tau=0}
 =
 -\frac{DQ_v[\varphi]}{\sqrt{1+|Dv|^2}}
 +\frac{Q(v)\,Dv\cdot D\varphi}
        {(1+|Dv|^2)^{3/2}},
 \qquad
 DQ_v[\varphi]
 :=\left.\frac{d}{d\tau}Q(v+\tau\varphi)\right|_{\tau=0}.
\]

Express the normalized residual $aQ(v)$ in the variables
$(\rho,q_1,q_2,\ldots)$, and define $S_{2J,\gamma}$ by
\[
 [aQ(v)]_{2J}
 =\sum_{\gamma\in\mathcal I_{2J}}
       q^\gamma S_{2J,\gamma}(\rho).
\]
Here $[\cdot]_{2J}$ denotes the coefficient of
$\varepsilon^{2J}$ after the substitution
$q_j\mapsto\varepsilon^j q_j$.
By induction, all residual terms of weight less than $2J$
vanish.

Consider a correction of the form
\[
 \varphi(t,y)
 =a(y)\sum_{\gamma\in\mathcal I_{2J}}
       q^\gamma(y)P_{2J,\gamma}(\rho).
\]
At weight $2J$, its contribution to $aDQ_v[\varphi]$ is
\[
 [aDQ_v[\varphi]]_{2J}
 =\sum_{\gamma\in\mathcal I_{2J}}
       q^\gamma L_hP_{2J,\gamma},
\]
where
\[
 L_hw
 =\frac{w''}{1+h'^2}
  +\left(\frac3\rho
         -\frac{2h'h''}{(1+h'^2)^2}\right)w'
  +\frac3{h^2}w.
\]
Indeed, replacing the linearization at $v/a$ by that at $h$
only changes higher-weight terms, as do the axial and
nonlinear contributions of the new correction.
Thus the perturbed mean curvature satisfies
\[
 [aH_{v+\varphi}]_{2J}
 =-\frac{1}{\sqrt{1+h'^2}}
   \sum_{\gamma\in\mathcal I_{2J}}
   q^\gamma
   \bigl(S_{2J,\gamma}+L_hP_{2J,\gamma}\bigr).
\]

We therefore define $P_{2J,\gamma}$ as the unique regular
even solution of
\[
 L_hP_{2J,\gamma}=-S_{2J,\gamma},
 \qquad
 P_{2J,\gamma}(0)=P'_{2J,\gamma}(0)=0.
\]
Set
\begin{align}\label{eq:corrector-inductive-def}
     W_J[a](t,y)
 =a(y)\sum_{\gamma\in\mathcal I_{2J}}
       q^\gamma(y)
       P_{2J,\gamma}\!\left(\frac{t}{a(y)}\right),
 \qquad
 V_J[a]=V_{J-1}[a]+W_J[a].
\end{align}
This cancels the weight-$2J$ residual without changing the
lower-weight cancellations. Since only even weights occur,
the residual $aQ(V_J[a])$ has no terms of weight less than
$2J+2$.

In Proposition~\ref{prop:correctors-estimates} we prove precise estimates for the profiles constructed above, for arbitrary given slowly varying profile $a$. We state the result informally here and collect the precise estimates in the Appendix~\ref{sec:corrector-estimates}.

\begin{proposition}[Informal version; see Proposition~\ref{prop:correctors-estimates} in Appendix~\ref{sec:corrector-estimates}]
\label{prop:correctors-informal}
Let $a(y)>0$ be any prescribed smooth radius varying on an axial
scale $R$, with $|a^{(n)}|\le C_n aR^{-n}$. For each fixed $J$,
if $a/R$ is sufficiently small, the universal corrections in
\eqref{eq:corrector-inductive-def} give a positive smooth profile
$V_J[a]\approx a+t$ with the same axis data as $V_0[a]$.
They cancel successive even-weight curvature errors, leaving
\[
 |Q(V_J[a])|\le C_J a^3R^{-2J-2}(a+t)^{2J-2}
\]
on the interior of the band, with corresponding weighted estimates
for all derivatives. The construction uses only $a$ and its
derivatives; it does not require $a$ to come from the connector.
\end{proposition}

For the connector radius, the same profiles also approximate
$u_\infty$ to arbitrary finite order. The proof of the following
corollary is given in Appendix~\ref{app:corrector-proofs}.
\begin{corollary}[Matching to arbitrary order]\label{cor:arbitrary-order-matching}
Take $a(y)=u_\infty(0,y)$ and let
$V_J[a]$ be the profiles supplied by
Proposition~\ref{prop:correctors-estimates}. For each $J>0$ and $|y_0|$ large enough depending on $J$ and $R = \theta_0|y_0|$ with $\theta_0>0$ small and fixed, we have:
    \begin{equation}\label{glue:eq:matching}
 |\partial_t^m\partial_y^n(u_\infty-V_J[a])|
 \le C_{J,m,n}a(y)^3R^{-2J-2-n}(a(y)+t)^{2J-m}
 \qquad(m,n\ge0),
\end{equation}
for all $m,n\in\mathbb N_0$ and $(t,y) \in \{t^2 + |y-y_0|^2 \leq R^2\}$.
\end{corollary}

We now use these finite-order estimates to choose the single
correction $f$ required by Proposition~\ref{bg:prop:choose-f}.

\begin{proof}[Proof of Proposition~\ref{bg:prop:choose-f}]
The preceding construction used only the properties in
Definition~\ref{bg:def:admissible}. We choose $f$ using the
finite-order profiles of
Proposition~\ref{prop:correctors-estimates} and
Corollary~\ref{cor:arbitrary-order-matching}, whose proofs are given
in Appendix~\ref{app:corrector-proofs}.

\smallskip\noindent\textbf{Step 1: The finite-order estimates.}
The complete statement of
Proposition~\ref{prop:correctors-estimates} applies on an enlarged
band of scale $R$ where
$a(y)\approx a(y_0)$, $|a^{(n)}|\le C_n aR^{-n}$, and
$a/R$ is sufficiently small. It provides profiles satisfying
\[
 V_J[a]=V_0[a]+\sum_{j=1}^J W_j[a],\qquad
 W_j(0,y)=\partial_tW_j(0,y)=0.
\]
For every fixed finite collection of derivative orders,
\begin{align}
 |\partial_t^m\partial_y^nW_j[a]|
 &\le C_{j,m,n}a^3R^{-2j-n}(a+t)^{2j-2-m},
       \label{bg:eq:increment-input}\\
 |\partial_t^m\partial_y^nQ(V_J[a])|
 &\le C_{J,m,n}a^3R^{-2J-2-n}(a+t)^{2J-2-m}.
       \label{bg:eq:residual-input}
\end{align}
The profiles are comparable to $a+t$ and have the corresponding
scaled derivative bounds.
For matching, translate and dilate the connector and cover each
polar attachment region
by finitely many bands whose radii are comparable to $\theta_0R_k$.
The enlarged bands lie where $a(y)=u_k(0,y)$; choose the fixed
polar width $c$ small enough for their transverse estimates.
Corollary~\ref{cor:arbitrary-order-matching} then gives
\begin{equation}\label{bg:eq:matching-input}
 |\partial_t^m\partial_y^n(u_k-V_J[a])|
 \le C_{J,m,n}a^3R_k^{-2J-2-n}(a+t)^{2J-m}.
\end{equation}
These are precisely the model and matching bounds
\eqref{glue:eq:correction}--\eqref{glue:eq:defect} and
\eqref{glue:eq:matching}. Constants and thinness thresholds may
depend on $J$ and
on the prescribed derivative orders. No uniformity in $J$ or
convergence of an infinite series is assumed. Even radial
smoothness gives the same weighted bounds in Cartesian variables.

\smallskip\noindent\textbf{Step 2: Choice of $f$ and conclusion.}
Choose radii $r_j\downarrow0$, $j\ge2$, and smooth cutoffs
$\zeta_j$ equal to one on $|y|\le r_j$, zero on
$|y|\ge5r_j/4$, with $|\zeta_j^{(n)}|\le C_nr_j^{-n}$.
Take $r_{j+1}<r_j/4$ and arrange their transition intervals
between consecutive insertion neighborhoods. For example, use
$r_j=\tfrac34 2^{-k_j}$ with sufficiently separated large
integers $k_j$. Set
\begin{equation}\label{bg:eq:choose-f}
 \boxed{\quad f=W_1[a]+\sum_{j=2}^{\infty}\zeta_j(y)W_j[a].\quad}
\end{equation}
For each $y\ne0$ the sum is locally finite. The profile
$V_0+f$ is either a completed $V_J$ or a transition
$V_J+\zeta_{J+1}W_{J+1}$. Its order is constant throughout
each insertion neighborhood. Since $\mathcal R(y)\ll r_j$ on
the $j$th transition, derivatives of its cutoff cost at most
the corresponding powers of $\mathcal R^{-1}$.

The transition obeys the residual estimate at the lower order
$J$. Indeed, \eqref{bg:eq:increment-input} bounds its added term
by $Ca^3\mathcal R^{-2J-2}s^{2J}$, with differentiated bounds,
and the curvature operator costs two derivatives at scale $s$.
Thus \eqref{bg:eq:residual-input} remains valid during the transition.

Here is the diagonal choice needed for admissibility. Fix
$L,m,N$. Convert the axial factors $\mathcal R^{-n}$ in the
input estimates to $Cs^{-n}$, since $s\le C\mathcal R$.
After dividing the residual by $s^{L-m}$, or the matching error
by $s^{L+2-m}$, both estimates reduce to
\begin{equation}\label{bg:eq:admissibility-budget}
 C_J a^3\mathcal R^{-2J-2}s^{2J-2-L}
 \le C_J\left(a^{2J+1-L}\mathcal R^{-2J-2}
                    +a^3\mathcal R^{-L-4}\right).
\end{equation}
The two terms on the right are the endpoint bounds for
$a\le s\le C\mathcal R$. If $2J+1>L$, both are positive
powers of $a$ times fixed powers of $\mathcal R^{-1}$.
By \eqref{eq:soft-cap-data}, they are bounded by $|y|^N$
sufficiently near zero.

Choose $r_j$ successively so small that these bounds hold
throughout $0<|y|\le2r_j$ for $J=j-1,j$ and all $L,m,N\le j$,
including the finite number
of additional derivatives needed for the curvature estimates.
All relevant constants and thinness thresholds are included
in this choice. Here $2(j-1)+1-L\ge j-1>0$ for $j\ge2$.
On a transverse collar, the leaf expansion and
\eqref{bg:eq:increment-input} give
$|D^m(V_J-t)|\le C_{J,m}a^3\mathcal R^{-2-m}$;
include the corresponding flatness requirements as well.

The same choice can ensure that each activated $W_j$, with its
cutoff, has scaled $C^m$ norm at most $2^{-j}$ for $m\le j$.
To see that this is available, divide
\eqref{bg:eq:increment-input} by $s^{1-m-n}$. For $j\ge2$ the
result is bounded by $C_{j,m,n}(a/\mathcal R)^3$ on the tube.
The fixed first correction is bounded in the same norm by
$C(a/\mathcal R)^2$. These quantities tend to zero. This
gives the uniform geometric bounds \eqref{bg:eq:geometry} and
preserves positivity after the working neighborhood is reduced.

For each fixed $L,m,N$, sufficiently small nonzero heights
lie in bands whose completed orders exceed all these indices.
The finitely many outer bands only change the constants.
Equations \eqref{bg:eq:admissibility-budget} and
\eqref{bg:eq:matching-input} therefore give
\eqref{bg:eq:admissible-residual}--\eqref{bg:eq:admissible-collar}.
The zero cap data and radial smoothness follow term by term.
\end{proof}

\appendix

\section{Higher derivative estimates for the connector}\label{app:connector-derivatives}

This appendix proves the analytic implications used in
Proposition~\ref{prop:precise-estimates-for-connector}. The first
step transfers the regular-sector quadratic asymptotics to the
spine, determining the smoothing radius without assuming bounds for its
derivatives. Local comparison with fixed leaves then identifies a
uniform weighted class for the differentiated equation. We prove
the axial derivative estimates by induction and only afterward
use the radial equation to estimate the difference from the leaf
with the same smoothing radius. This order avoids using derivatives of the smoothing
radius before they have been controlled.

\subsection{Quadratic phase bound near the spine}
Here we improve the quadratic estimates of Proposition~\ref{prop:quadratic-bound} to hold near the spine:
\begin{lemma}[Quadratic phase asymptotics near the spine]
\label{lem:quadratic-phase-near-spine}
Let $M=S(u_\infty)$ be the connector of
Theorem~\ref{thm:main-connector}, and assume the quadratic bound
of Proposition~\ref{prop:quadratic-bound}. Let $A>0$ be the
quadratic coefficient in that proposition, and let $c_H>0$
be determined by
\[
 P(0,z,Y)=c_H|z|^3.
\]
Then, as $|Y|\to\infty$,
\[
 a_\infty(Y)
 =\left(\frac{A}{c_H}\right)^{1/3}|Y|^{2/3}
  +O\bigl(|Y|^{2/3-\kappa_2}\bigr),
 \qquad a_\infty(Y)=u_\infty(0,Y),
\]
where one may take $\kappa_2=1/6$. Moreover, there is a fixed
$\theta_0>0$ such that, for $T=|y_0|$ large and
$R=\theta_0T$, throughout the base ball $B_{4R}(0,y_0)$,
\begin{equation}\label{eq:quadratic-uniform-band}
 \left|P(t,u_\infty(t,y))
 -A\left(y^2-\frac{t^2+u_\infty(t,y)^2}{3}\right)\right|
 \le CT^{11/6},\qquad
 cT^2\le P(t,u_\infty(t,y))\le CT^2.
\end{equation}
These constants are independent of $y_0$.
\end{lemma}

\begin{proof}
By reflection symmetry, it suffices to consider $Y\to+\infty$.
For $R$ large, set
\[
 M_R=R^{-1}M,\qquad \eps=R^{-1},\qquad
 \Pi_R(X)=R^{-2}P(RX)\quad(X\in M_R).
\]
The cubic homogeneity of $P$ gives
\[
 \Pi_R(X)=R\,P(X)=\eps^{-1}P(X).
\]
Thus $\Pi_R$ is the physical phase divided by $R^2$, pulled
back to the rescaled surface; the actual phase on $M_R$ is
$\eps\Pi_R$. In what follows, $(r,y)$ are rescaled coordinates,
and the annulus under consideration has fixed size.

We first establish the regular-region estimate
\begin{equation}\label{eq:cap-normalized-phase}
 \left|\Pi_R-A\left(y^2-\frac{r^2}{3}\right)\right|
 \le C\frac{R^{-\kappa}}{r},
 \qquad
 \frac12<\sqrt{r^2+y^2}<2,\quad
 r>C_*R^{-1/3},
\end{equation}
where we take $\kappa=1/4$ and $C_*$ is sufficiently large.

To see this, put
\[
 d_R=D(8R),\qquad \eps_R=d_R/R^3.
\]
The quadratic bound gives $d_R\le CR^2$. Multiplying
\eqref{eq:end-replacement-bound} by $d_R$, the error between
the physical conjugate $s^2f(s,Y)$ and its harmonic
replacement, at $(s,Y)=(Rr,Ry)$, is bounded by
\[
 C\frac{d_R\eps_R^{1/3}}{r}
 =C\frac{d_R^{4/3}}{Rr}
 \le C\frac{R^{5/3}}{r}.
\]
Repeat the dyadic patching construction in the proof of
Proposition~\ref{prop:quadratic-bound}, now using this quadratic
bound. The overlap errors are $O(R^{5/3})$, so the resulting
function $\widehat W$ on the exterior auxiliary space satisfies
\[
 |\widehat W|\le C\rho^2,\qquad
 |\Delta_{\xi,Y}\widehat W|\le C\rho^{-1/3},
 \qquad \rho=(s^2+Y^2)^{1/2}.
\]
The same Newton-potential construction gives a correction
$v=O(\rho^{5/3})$ for which $\widehat W-v$ is harmonic.
Its invariant, reflection-even harmonic expansion therefore
gives
\[
 \widehat W(s,Y)
 =A\left(Y^2-\frac{s^2}{3}\right)+O(\rho^{5/3}).
\]
The coefficient agrees with the previously identified $A$
by comparison on any fixed regular angular sector.

The fixed-leaf expansion
\cite[Proposition~3.3]{Sz} gives
\[
 f_{\mathrm{leaf}}(s)
 =\lambda s^{-2}+O(|\lambda|^{4/3}s^{-3}).
\]
Consequently, at points of the normal graph with
$s>C d_R^{1/3}$,
\[
 |P-s^2f|\le C\frac{d_R^{4/3}}{s}.
\]
Combining these estimates and dividing by $R^2$ yields
\[
 \Pi_R
 =A\left(y^2-\frac{r^2}{3}\right)
  +O\left(\frac{R^{-1/3}}{r}\right),
\]
which implies \eqref{eq:cap-normalized-phase}.
Here the computation initially uses the cone radius.
It also holds with $r=|(x,z)|$ on $M_R$: the squared ambient
and cone radii differ by $f_R^2$, where
$|f_R|\le C R^{-1}r^{-2}$, and this difference is absorbed
by the displayed error.

We next transfer this estimate near the spine. Fix
\[
 I=(2/3,4/3),\qquad J=[3/4,5/4],
 \qquad r_0=R^{-\kappa/3}.
\]
On the lateral boundary $r=r_0$, \eqref{eq:cap-normalized-phase}
gives
\begin{equation}\label{eq:cap-lateral-phase}
 |\Pi_R-Ay^2|
 \le C\left(r_0^2+\frac{R^{-\kappa}}{r_0}\right).
\end{equation}
Moreover, the quadratic bound implies $|\Pi_R|\le K$ on
$M_R\cap\{r<r_0,\ y\in I\}$, for a fixed $K$.

Let $H_{\mathrm{ph}}$ be the positive Hardt-Simon leaf
normalized by $P=1$. For a uniformly bounded $C^2$ family
of functions $q:I\to\R$, the barriers of
\cite[Proposition~5.8]{Sz} are normal graphs over the
slice-wise leaves
\[
 \eps^{1/3}q(y)^3H_{\mathrm{ph}},
 \qquad
 P\big|_{\eps^{1/3}q(y)^3H_{\mathrm{ph}}}
 =\eps q(y)^9.
\]
Negative dilations here mean positive dilations with the
$x$ and $z$ factors interchanged. These barriers have
negative mean curvature; interchanging the factors also
provides the barriers with the opposite sign.

We record the quantitative phase error. The normal
correction $\psi_q$ in that construction satisfies
\[
 |\psi_q|\le C\eps(|q|^3+r),\qquad
 r\ge c\eps^{1/3}|q|^3.
\]
Hence $|\psi_q|/r\le C(\eps^{2/3}+\eps)$.
Using $|DP|\le Cr^2$, we obtain, on $r<r_0$,
\begin{equation}\label{eq:cap-barrier-phase-error}
 \eps^{-1}|P-\eps q^9|\le Cr_0^2.
\end{equation}
This is the quantitative form of the phase comparison
used in \cite[Proposition~5.10]{Sz}; it follows directly
from the correction estimates in its proof.

Choose
\[
 \delta=L\left(r_0^2+\frac{R^{-\kappa}}{r_0}\right),
\]
where $L$ is a sufficiently large fixed constant.
Let $\chi\in C^\infty(I;[0,1])$ vanish on $J$ and equal
one near the endpoints of $I$, and define
\[
 q_+(y)=(Ay^2+\delta)^{1/9}+B\chi(y),\qquad
 q_-(y)=(Ay^2-\delta)^{1/9}-B\chi(y).
\]
For large $R$, the expressions inside the roots are
uniformly positive. Choose the fixed constant $B$ so large
that, near the endpoints of $I$,
\[
 q_+^9>K+2,\qquad q_-^9<-K-2.
\]
These functions have uniformly bounded $C^2$ norms and satisfy
\[
 q_+^9\ge Ay^2+\delta,\qquad
 q_-^9\le Ay^2-\delta
 \quad\text{on }I,
\]
with equality on $J$.

Apply the upper barriers to $q_++s$ and the lower barriers
to $q_-s$. Start with a sufficiently large fixed $s$,
so that the barriers lie beyond the global phase range,
and decrease $s$ to zero. The axial boundary comparisons
follow from the choice of $B$. The lateral comparisons
follow from \eqref{eq:cap-lateral-phase},
\eqref{eq:cap-barrier-phase-error}, and the choice of $L$.
Interior first contact is excluded by
\cite[Proposition~5.8]{Sz}, including at singular points
of the barriers. Thus the comparison persists to $s=0$.

On the smaller band $J$, the resulting phase bounds are
\[
 |\Pi_R-Ay^2|\le C\delta
 \qquad\text{on }M_R\cap\{r<r_0,\ y\in J\}.
\]
The axis orbit at height $y=1$ belongs to this region for large $R$:
indeed, the quadratic phase bound gives
$a_\infty(R)/R\le CR^{-1/3}\ll r_0$.
Evaluating there gives
\[
 \left|\frac{c_Ha_\infty(R)^3}{R^2}-A\right|
 \le C\delta.
\]
Finally, since $\kappa=1/4$,
\[
 r_0=R^{-1/12},\qquad
 \delta=O(R^{-1/6}).
\]
Therefore
\[
 c_Ha_\infty(R)^3=AR^2+O(R^{2-1/6}).
\]
Taking cube roots, using $A>0$, proves the claimed expansion
with $\kappa_2=1/6$.

We also record the uniform band conclusion. Apply the preceding
rescaling at $T=y_0$. Choose $\theta_0$ so small that the graph over
$B_{4\theta_0T}(0,y_0)$ lies in the rescaled annulus and has
$y\in J$. This follows from $|Du_\infty|\le\Lambda$ and
$a_\infty(y_0)/T=O(T^{-1/3})$. In the shrinking tube
$r<r_0=T^{-1/12}$, the barrier comparison gives
$|\Pi_T-Ay^2|\le CT^{-1/6}$, and the omitted term $Ar^2/3$
has the same bound. Outside that tube,
\eqref{eq:cap-normalized-phase} gives an error at most
$CT^{-1/4}/r_0=CT^{-1/6}$. Here
$r^2=T^{-2}(t^2+u_\infty^2)$ in the physical graph coordinates.
Multiplying by $T^2$ proves the first assertion of
\eqref{eq:quadratic-uniform-band}. The quadratic polynomial is
uniformly positive on the chosen band, so the second assertion
follows for large $T$. Reflection gives all conclusions for
negative $y_0$.
\end{proof}

\subsection{Local comparison with fixed leaves}
We now work directly with the minimal graph $u=u_\infty$.
The quadratic asymptotics first give local comparison estimates;
these will control the coefficients of its differentiated equation.
Throughout this subsection and the next, set
\[
 T=|y_0|,\qquad R=\theta_0T,\qquad
 \beta=u(0,y_0),\qquad X_0=(0,y_0),\qquad s(X)=\beta+|x|.
\]
Choose the fixed $\theta_0>0$ sufficiently small and then $T$
sufficiently large. In particular $\beta^3\approx T^2$ and
$\beta<R/100$. Cartesian derivatives are denoted by $D$.

\begin{lemma}[Fixed-leaf comparison and Jacobi normalization]
\label{lem:direct-leaf-comparison}
On $B_{4R}(X_0)$ the leaf labels satisfy
\[
 c\beta^3\le c_H^{-1}P(t,u(t,y))\le C\beta^3.
\]
For $X_*=(x_*,y_*)\in B_{2R}(X_0)$ put $s_* =\beta+|x_*|$,
$\lambda_*=c_H^{-1}P(t_*,u(X_*))$, and
$u_*(t)=H(t,\lambda_*)$. There is a fixed $c>0$ such that
\begin{equation}\label{eq:direct-frozen-comparison}
 |D^k(u-u_*)|\le C_k\beta^3s_*^{-2-k}(s_*/R)^{1/2}
 \quad\hbox{on }B_{cs_*}(X_*),\qquad k\ge0.
\end{equation}
Define, by evaluation of fixed functions at $u$,
\begin{equation}\label{eq:direct-jacobi-normalization}
 d(X)=H_\lambda\!\left(t,c_H^{-1}P(t,u(X))\right),\qquad
 W=\sqrt{1+|Du|^2},\qquad
 B=\frac{I}{W}-\frac{Du\otimes Du}{W^3}.
\end{equation}
Then $d\approx s^{-2}$ and $|D^kd|\le C_ks^{-2-k}$.
For the vertical Jacobi operator and its conjugate, set
\begin{equation}\label{eq:direct-operators}
 \begin{split}
 \mathcal J\phi&=-\operatorname{div}(u^3B D\phi)-\frac{3u}{W}\phi,\\
 \mathcal A&=u^3d^2B,\qquad \mathcal V=d\mathcal Jd,\qquad
 \mathcal Lq=-\operatorname{div}(\mathcal A Dq)+\mathcal Vq.
 \end{split}
\end{equation}
With $\omega_\beta=(\beta^2+t^2)^{-1/2}$, these coefficients satisfy
\begin{equation}\label{eq:direct-coefficients}
 \begin{gathered}
 c\omega_\beta I\le\mathcal A\le C\omega_\beta I,\qquad
 |D^k\mathcal A|\le C_k\omega_\beta s^{-k},\\
 |D^k\mathcal V|\le C_k\omega_\beta s^{-2-k}(s/R)^{1/2}
 \quad\hbox{on }B_{2R}(X_0).
 \end{gathered}
\end{equation}
All constants are independent of $\beta/R$.
\end{lemma}

\begin{proof}
Lemma~\ref{lem:quadratic-phase-near-spine} gives the positive
leaf band and the uniform estimate
\[
 \left|P(t,u)-A\left(y^2-\frac{t^2+u^2}{3}\right)\right|
 \le CT^{11/6}.
\]
The polynomial evaluated on the graph is Lipschitz with constant
$CT$ on $B_{4R}(X_0)$, because $|Du|\le\Lambda$ and
$u\le C(\beta+t)$. Its variation on a ball of radius $cs_*$ is
therefore at most $CTs_*$. Since $\beta^3\approx T^2$,
\[
 \frac{\operatorname{osc}_{B_{2cs_*}(X_*)}P(t,u)}{\beta^3}
 \le C\left(\frac{s_*}{R}+T^{-1/6}\right)
 \le C(s_*/R)^{1/2}.
\]
The last inequality uses $s_*\ge\beta\approx T^{2/3}$ and
$s_*/R\le C$. This is a consequence of the quantitative geometric
comparison, without differentiating a leaf-coordinate equation.

For $\lambda\approx\beta^3$, the fixed-leaf bounds give
\begin{equation}\label{eq:direct-leaf-parameter-bounds}
 H\approx\beta+t,\qquad H_\lambda\approx(\beta+t)^{-2},\qquad
 |\partial_t^k\partial_\lambda^rH|
 \le C_{k,r}\beta^{3-3r}(\beta+t)^{-2-k},\quad r\ge1.
\end{equation}
The mean-value theorem gives the $k=0$ case of
\eqref{eq:direct-frozen-comparison}. After rescaling lengths
and graph heights by $s_*$, both $u$ and $u_*$ have uniformly
positive height and bounded gradient. Standard interior regularity
for the SME bounds their coefficient derivatives, and their
difference satisfies a linear uniformly elliptic equation with
uniformly controlled lower-order coefficients on this scale.
Interior estimates for that difference prove
\eqref{eq:direct-frozen-comparison} for all $k$.

For clarity, the required differentiation of the fixed functions
in \eqref{eq:direct-jacobi-normalization} can be checked using
the implicit identity
$u=H(t,c_H^{-1}P(t,u))$ on this local ball.
Inverting its parameter derivative and using
\eqref{eq:direct-leaf-parameter-bounds} gives, with
$d_*(t)=H_\lambda(t,\lambda_*)$,
\[
 |D^k(d-d_*)|\le C_ks_*^{-2-k}(s_*/R)^{1/2}.
\]
This also gives the asserted bounds on $d$ and $\mathcal A$.
The field $d_*$ is an exact Jacobi field on the fixed leaf $u_*$,
so $\mathcal J_*d_*=0$. Subtract this identity from
$\mathcal Jd$ on the local ball. Coefficients of the principal
part of $\mathcal J$ have size $s_*^3$, and its potential has
size $s_*$. Their differences from the fixed-leaf coefficients
retain a factor $(s_*/R)^{1/2}$. Together with the displayed
bound for $d-d_*$, this gives
\[
 |D^k(\mathcal Jd)|\le C_ks_*^{-1-k}(s_*/R)^{1/2}.
\]
Multiplication by $d\approx s_*^{-2}$ proves the bound on
$\mathcal V$ in \eqref{eq:direct-coefficients}.
\end{proof}

\subsection{Weighted estimates for the differentiated equation}\label{sec:direct-weighted-theory}
The Euler--Lagrange expression for the area density
$u^3\sqrt{1+|Du|^2}$ is
$-u^3Q(u)/W$. Its linearization at a solution is the operator
$\mathcal J$ in \eqref{eq:direct-operators}. Thus
\begin{equation}\label{eq:direct-jacobi-identities}
 \mathcal Ju_y=0,\qquad
 \mathcal J\phi=-\frac{u^3}{W}DQ_u[\phi],\qquad
 d\mathcal J(dq)=\mathcal Lq.
\end{equation}
These identities explain both the direct differentiation and the
Jacobi normalization. In particular, the remaining potential
$\mathcal V$ is controlled by the local difference from a fixed leaf.

We recall the precise weighted results used below. For a positive
weight $\omega$ on $\R^5$, its Muckenhoupt characteristic is
\begin{equation}\label{eq:direct-a2-definition}
 [\omega]_{A_2}
 :=\sup_{\mathbb B}
 \left(\frac1{|\mathbb B|}\int_{\mathbb B}\omega\,dX\right)
 \left(\frac1{|\mathbb B|}\int_{\mathbb B}\omega^{-1}\,dX\right),
\end{equation}
where the supremum is over Euclidean balls. Membership in $A_2$
means that this quantity is finite. With $d\mu=\omega\,dX$,
\cite[Theorems~1.2 and~1.5]{FKS} give, for some $\kappa>1$,
\[
 \begin{split}
 \left(\frac1{\mu(\mathbb B)}
       \int_{\mathbb B}|\psi|^{2\kappa}\,d\mu\right)^{1/(2\kappa)}
 &\le C\rho\left(\frac1{\mu(\mathbb B)}
       \int_{\mathbb B}|D\psi|^2\,d\mu\right)^{1/2},
       \qquad \psi\in C_c^\infty(\mathbb B),\\
 \int_{\mathbb B}|\psi-\psi_{\mathbb B,\mu}|^2\,d\mu
 &\le C\rho^2\int_{\mathbb B}|D\psi|^2\,d\mu,
       \qquad \psi\in C^1(\overline{\mathbb B}),
 \end{split}
\]
where $\rho$ is the ball radius and $\psi_{\mathbb B,\mu}$ is the
weighted average. The constants depend only on dimension and
$[\omega]_{A_2}$; the measure is also doubling
\cite[equation~(1.2)]{FJK}.
For a measurable symmetric matrix $\mathcal A$ with
$c\omega I\le\mathcal A\le C\omega I$, a weak solution
$h\in H^1_{\mathrm{loc}}(d\mu)$ of
$-\operatorname{div}(\mathcal A Dh)=0$ in $\mathbb B$ satisfies
\begin{equation}\label{eq:direct-fks-local-bound}
 \|h\|_{L^\infty(\mathbb B/2)}
 \le C\left(\frac1{\mu(\mathbb B)}
                 \int_{\mathbb B}|h|^2\,d\mu\right)^{1/2}.
\end{equation}
This is \cite[Theorem~2.3.1]{FKS} applied to $h$ and $-h$;
see also \cite[equations~(1.5)--(1.7)]{FJK}.
Its constant depends only on the same data and the weighted
ellipticity constants. We use this local boundedness result;
higher derivatives will follow by differentiating the SME and
applying ordinary interior estimates on the local scale $s$.

For the present weight $\omega_\beta=(\beta^2+|x|^2)^{-1/2}$,
the $A_2$ bound is independent of $\beta>0$. Indeed, for a ball
$\mathbb B=B_\rho(X)$ put $\ell=(\beta^2+|x|^2)^{1/2}$.
If $\rho\le\ell/2$, the weight is comparable to $\ell^{-1}$
throughout the ball. If $\rho>\ell/2$, then $|x|<2\rho$,
$\omega_\beta^{-1}\le3\rho$ on $\mathbb B$, and
\[
 \int_{\mathbb B}\omega_\beta\,dX
 \le2\rho\int_{|\xi|<3\rho,\ \xi\in\R^4}|\xi|^{-1}\,d\xi
 \le C\rho^4.
\]
Since $|\mathbb B|\approx\rho^5$, the product in
\eqref{eq:direct-a2-definition} is uniformly bounded in both cases.
Direct integration also gives, now with $d\mu=\omega_\beta\,dX$,
\begin{equation}\label{eq:direct-weighted-volume}
 \mu(B_\rho(X))\approx\frac{\rho^5}{\beta+|x|+\rho}.
\end{equation}
Thus all the preceding constants are uniform as $\beta/R\to0$.
Lemma~\ref{lem:direct-leaf-comparison} puts the actual matrix
$\mathcal A=u^3d^2B$ in this weighted ellipticity class, using the
positive quadratic phase band to obtain $u\approx s$ and
$d\approx s^{-2}$. Multiplication of the weight by a function
bounded above and below by fixed positive constants preserves
its uniform $A_2$ bound.

Equation $\mathcal Lq=F$ still contains $\mathcal Vq$ and $F$.
To use \eqref{eq:direct-fks-local-bound}, we remove the potential
by a bounded positive solution of $\mathcal L\vartheta=0$ and
then subtract the Green potential of the source.
The Green-function bounds and representation are supplied by
\cite[Theorem~3.3 and Lemma~2.5]{FJK}; the next proof checks the
uniform bounds needed for these two reductions.

\begin{lemma}[Uniform weighted local estimate]
\label{lem:direct-weighted-estimate}
Let $d\mu=\omega_\beta\,dX$, and let $\mathcal L$ be as in
Lemma~\ref{lem:direct-leaf-comparison}. Suppose on an interior
ball $B\subset B_{2R}(X_0)$ that
\[
 \mathcal Lq=F,\qquad
 |F|\le M\omega_\beta s^{-2}(\beta/s)^3.
\]
If $B'\Subset B$ have radii and separation comparable to $R$,
then
\begin{equation}\label{eq:direct-linear-bound}
 \|q\|_{L^\infty(B')}
 \le C\left[\left(\frac1{\mu(B)}\int_Bq^2\,d\mu\right)^{1/2}+M\right].
\end{equation}
The constant may depend on the fixed ball ratios, but is
independent of $\beta/R$.
\end{lemma}

\begin{proof}
The $A_2$ theory applies to the principal part of $\mathcal L$.
We first estimate the Green potentials of the zeroth-order term
and the source, then reduce to the homogeneous equation in
\eqref{eq:direct-fks-local-bound}.

Let $D$ be a ball of radius $\rho$ in the coefficient domain.
The Dirichlet Green kernel $G_D$ of
$-\operatorname{div}(\mathcal A D)$ is normalized with respect
to Lebesgue measure. Extend $\mathcal A$ by $\omega_\beta I$
outside $D$ and consider its Green function on the concentric ball
of radius $8\rho$. Points of $D$ lie in the inner quarter of that
ball, so \cite[Theorem~3.3]{FJK} applies. Domain monotonicity and
\eqref{eq:direct-weighted-volume} give
\begin{equation}\label{eq:direct-green-kernel}
 G_D(X,Z)\le C\int_{|X-Z|}^{8\rho}
                  \frac{r}{\mu(B_r(X))}\,dr
 \le C\frac{s(X)+|X-Z|}{|X-Z|^3}.
\end{equation}
The same conclusion holds for any coefficients uniformly
comparable to $\omega_\beta I$.

Two different integrability gains will be used. For the potential,
\eqref{eq:direct-coefficients} and the kernel bound imply
\begin{equation}\label{eq:direct-potential-smallness}
 \sup_{X\in D}\int_DG_D(X,Z)|\mathcal V(Z)|\,dZ
 \le C(\rho/R)^{1/2}.
\end{equation}
Indeed, at distances $r\le s(X)/2$ the contribution is bounded
by $CR^{-1/2}s(X)^{-3/2}r^2\le C(r/R)^{1/2}$.
On a dyadic annulus with radius $r\ge s(X)/2$, the kernel is
at most $Cr^{-2}$, while
\[
 \int_{B_r(X)}\omega_\beta s^{-3/2}\,dZ
 \le Cr\int_0^{Cr}t^3(\beta+t)^{-5/2}\,dt\le Cr^{5/2}.
\]
Summing gives \eqref{eq:direct-potential-smallness}.

For the source, the separate radial factor gives
\begin{equation}\label{eq:direct-source-potential}
 \sup_{X\in D}\int_DG_D(X,Z)
       \omega_\beta(Z)s(Z)^{-2}(\beta/s(Z))^3\,dZ\le C.
\end{equation}
For $r\le s(X)/2$, its contribution is bounded by
$C\beta^3s(X)^{-5}r^2\le C$.
For $r\ge s(X)/2$, the source mass obeys
\[
 \int_{B_r(X)}\omega_\beta s^{-2}(\beta/s)^3\,dZ
 \le Cr\beta^3\int_0^{Cr}\frac{t^3}{(\beta+t)^6}\,dt
 \le C\beta r.
\]
The dyadic contribution is consequently at most $C\beta/r$,
whose sum is uniformly bounded because $s(X)\ge\beta$.
This source need not be small at the axis.

Work first on a ball $D$ of radius $\varepsilon R$.
Choose a fixed $\varepsilon>0$ so that
\eqref{eq:direct-potential-smallness} is at most $1/4$.
The convergent Green-operator series solves
\[
 \vartheta=1-G_D(\mathcal V\vartheta),\qquad
 \tfrac12\le\vartheta\le2,\qquad\mathcal L\vartheta=0.
\]
Writing $z=q/\vartheta$ gives
\[
 -\operatorname{div}(\vartheta^2\mathcal A Dz)=\vartheta F.
\]
The matrix $\vartheta^2\mathcal A$ is still symmetric and
uniformly comparable to $\omega_\beta I$, so its Green kernel
satisfies \eqref{eq:direct-green-kernel} with the same uniform
constants. Subtract its Dirichlet Green potential of $\vartheta F$.
The supremum of this potential is at most $CM$ by
\eqref{eq:direct-source-potential}.
For each fixed $\beta>0$, the density $\vartheta F/\omega_\beta$
is bounded; thus the representation follows from
\cite[Lemma~2.5]{FJK}. The remaining function solves a homogeneous
weighted divergence equation, and
\eqref{eq:direct-fks-local-bound} yields
\[
 \|q\|_{L^\infty(D/2)}
 \le C\left[\left(\frac1{\mu(D)}\int_Dq^2\,d\mu\right)^{1/2}+M\right].
\]
A finite covering of $B'$ by such balls, and doubling to compare
their averages with the average on $B$, proves
\eqref{eq:direct-linear-bound}.
\end{proof}

\begin{proposition}[Direct axial estimates for the minimal graph]
\label{prop:direct-axial-regularity}
In the band of Lemma~\ref{lem:direct-leaf-comparison}, for every
fixed $N\ge1$ and every fixed $k\ge0$,
\begin{equation}\label{eq:direct-axial-bounds}
 |D^k\partial_y^ju(X)|
 \le C_{N,k}\beta^3R^{-j}(\beta+t)^{-2-k},
 \qquad 1\le j\le N,\quad X\in B_{R/4}(X_0).
\end{equation}
In particular,
$|a_\infty^{(j)}(y)|\le C_j a_\infty(y)R^{-j}$ on the
smaller axial interval. The band and the lower threshold for
$|y_0|$ are independent of $N$.
\end{proposition}

\begin{proof}
Fix $N$ and use finitely many nested concentric balls between
$B_R(X_0)$ and $B_{R/4}(X_0)$. Their separations may depend on
$N$. Let
\[
 u_j=\partial_y^ju,\qquad q_j=\frac{u_j}{\beta^3d},\qquad j\ge1,
\]
where $\beta$ is fixed. These are normalized derivatives of the
original graph. We prove
\begin{equation}\label{eq:direct-induction}
 \|q_j\|_\infty\le C_jR^{-j},\qquad
 |D^k u_j|\le C_{j,k}\beta^3R^{-j}s^{-2-k},\quad k\ge0,
\end{equation}
on successively smaller balls. No axial derivative estimate
for $a_\infty$ is assumed.

For $j=1$, \eqref{eq:direct-jacobi-identities} gives
$\mathcal Lq_1=0$. The local comparison
\eqref{eq:direct-frozen-comparison}, whose comparison leaf
has zero axial derivative, gives $|q_1|\le Cs^{-1}$.
Since $\mu(B_R)\approx R^4$,
\[
 \int_{B_R}q_1^2\,d\mu
 \le CR\int_0^R\frac{t^3}{(\beta+t)^3}\,dt
 \le CR^2\le CR^{-2}\mu(B_R).
\]
Lemma~\ref{lem:direct-weighted-estimate}, with zero forcing,
therefore gives $\|q_1\|_\infty\le C/R$ on a smaller ball.
The local elliptic estimate described below supplies the mixed
derivatives and completes this first step.

We record the sources at higher orders. In Cartesian coordinates,
\[
 Q(u)=A^{ij}(Du)u_{ij}-3/u,\qquad
 A^{ij}(\xi)=\delta_{ij}-\frac{\xi_i\xi_j}{1+|\xi|^2}.
\]
Differentiating $j$ times and separating the complete linearization
gives
\begin{equation}\label{eq:direct-differentiated-sme}
 DQ_u[u_j]=-T_j,\qquad T_1=0.
\end{equation}
Every term of $T_j$ contains at least two factors with positive
axial order, all strictly less than $j$. Apart from fixed
contractions and constants, its principal terms are
\[
 (\partial_\xi^b A^{ij})(Du)
       \prod_{\nu=1}^b Du_{r_\nu}\,D_{ij}u_q,
 \qquad r_\nu\ge1,\quad \sum r_\nu+q=j,
\]
with the linear terms $(b,q)=(0,j)$ and $(b,r_1,q)=(1,j,0)$
removed. Its remaining terms come from differentiating $-3/u$
and have the form
\[
 C u^{-1-b}\prod_{\nu=1}^b u_{r_\nu},
 \qquad b\ge2,\quad r_\nu\ge1,\quad \sum r_\nu=j.
\]
The local bounds for the undifferentiated solution are
$|D^ku|\le C_ks^{1-k}$, and $|Du|\le\Lambda$.
Assuming \eqref{eq:direct-induction} at lower axial orders,
each differentiated gradient costs
$\beta^3R^{-r}s^{-3}$, and each differentiated Hessian costs
$\beta^3R^{-q}s^{-4}$. In a principal term with $q=0$
there are at least two gradient factors and $|D^2u|\le C/s$.
The reciprocal terms have the same bound. Since $\beta/s\le1$,
the product rule therefore gives
\begin{equation}\label{eq:direct-raw-sources}
 |D^kT_j|\le C_{j,k}\beta^6R^{-j}s^{-7-k},\qquad j\ge2.
\end{equation}
All fixed local Cartesian derivative orders are retained in the
induction, so
the extra derivatives in these products are already available.

By \eqref{eq:direct-jacobi-identities}, the normalized equation is
\begin{equation}\label{eq:direct-normalized-sources}
 \mathcal Lq_j=F_j,\qquad
 F_j=\frac{d u^3}{\beta^3W}T_j,\qquad
 |D^kF_j|\le C_{j,k}\omega_\beta R^{-j}s^{-2-k}(\beta/s)^3.
\end{equation}
The factor $(\beta/s)^3$ is the gain that permits direct
differentiation of $u$ at every order. It gives a uniformly
bounded source potential by \eqref{eq:direct-source-potential}.

To obtain the needed $L^2$ bound at order $j\ge2$, test the
equation for $q_{j-1}$ with $\zeta^2q_{j-1}$, where
$|D\zeta|\le C_N/R$. Ellipticity and Young's inequality give
\[
 \begin{split}
 \int\zeta^2|Dq_{j-1}|^2\,d\mu
 &\le CR^{-2}\int_Bq_{j-1}^2\,d\mu
 +C\|q_{j-1}\|_\infty^2\int_B|\mathcal V|\,dX\\
 &\quad+C\|q_{j-1}\|_\infty\int_B|F_{j-1}|\,dX.
 \end{split}
\]
The ball $B$ is centered at $X_0$ and has radius comparable to
$R$. The radial integrations in the preceding lemma give
\[
 \int_B|\mathcal V|\,dX\le CR^{-2}\mu(B),\qquad
 \int_B\omega_\beta s^{-2}(\beta/s)^3\,dX
 \le C\beta R\le CR^{-2}\mu(B).
\]
Together with the induction hypothesis and
\eqref{eq:direct-normalized-sources}, this proves
\[
 \int\zeta^2|Dq_{j-1}|^2\,d\mu\le CR^{-2j}\mu(B).
\]
The normalization contributes one additional, controlled term:
\[
 q_j=\partial_yq_{j-1}+\frac{d_y}{d}q_{j-1},\qquad
 \frac{d_y}{d}=\beta^3\frac{H_{\lambda\lambda}}{H_\lambda}q_1.
\]
Here the fixed-profile derivatives are evaluated at the leaf
through $u(X)$. By \eqref{eq:direct-leaf-parameter-bounds} and
the first step, $|d_y/d|\le C/R$. Thus, on the next smaller ball $B_1$,
\[
 \left(\frac1{\mu(B_1)}\int_{B_1}q_j^2\,d\mu\right)^{1/2}
 \le CR^{-j}.
\]
Applying Lemma~\ref{lem:direct-weighted-estimate} to
\eqref{eq:direct-normalized-sources}, with $M=CR^{-j}$, gives
$\|q_j\|_\infty\le CR^{-j}$.

At each step, recover the mixed derivatives on local balls of
radius comparable to $s=\beta+t$. Divide the normalized equation
by $\omega_\beta$ and rescale by $s$. Its principal part is
uniformly elliptic, and Lemma~\ref{lem:direct-leaf-comparison}
bounds all coefficient derivatives on that scale. Equation
\eqref{eq:direct-normalized-sources} bounds the rescaled source
and its derivatives by $C_{j,k}R^{-j}$.
Standard interior estimates therefore give
\[
 |D^kq_j|\le C_{j,k}R^{-j}s^{-k}.
\]
Multiplying by $\beta^3d$ proves the second estimate in
\eqref{eq:direct-induction}. Near the boundary of an intermediate
ball, its radius can be reduced by a factor depending on $N$;
this changes only the constants. The induction is now complete.

Restricting to $B_{R/4}(X_0)$ proves
\eqref{eq:direct-axial-bounds}. At $x=0$,
$a_\infty^{(j)}(y)=u_j(0,y)$, so the smoothing-radius estimates follow
from $\beta\approx a_\infty(y)$. The geometric threshold was used
only in the initial comparison lemma. Higher orders change
constants and intermediate ball separations, not that threshold.
\end{proof}

\subsection{Expansion around Hardt--Simon leaves}
The direct axial estimates now control both the smoothing radius
$a_\infty(y)=u_\infty(0,y)$ and the axial terms in the SME.
We compare each slice with the leaf having that smoothing radius:
\[
 u_\infty(t,y)=a_\infty(y)h(t/a_\infty(y))+e_\infty(t,y),
 \qquad e_\infty(0,y)=\partial_te_\infty(0,y)=0.
\]
Moving the axial terms of the SME to the right-hand side gives a
forced radial equation. Each term contains two axial derivatives
or a product of two first axial derivatives, hence a factor
$R^{-2}$. The following proposition converts that forcing estimate
into the remainder bounds required for gluing.

\begin{proposition}[Expansion about Hardt-Simon leaves]
\label{prop:direct-leaf-expansion}
Let $N\ge1$ be an arbitrary integer.
Fix $|y_0|$ sufficiently large and set
\[
 \beta:=u_\infty(0,y_0),\qquad R:=\theta_0|y_0|,
\]
where $\theta_0>0$ is the fixed constant chosen in
Lemma~\ref{lem:direct-leaf-comparison}.
Thus $\beta$ is the smoothing radius at the reference height $y_0$
and is held fixed throughout the band.
Assume the direct estimates of Proposition~\ref{prop:direct-axial-regularity}
through axial order $N+2$. There is a fixed $c>0$, independent
of $N$, such that, for $|y-y_0|\le cR$,
the varying smoothing radius
\[
 a_\infty(y):=u_\infty(0,y)
\]
satisfies
\[
 |a_\infty^{(j)}(y)|
 \le C_j a_\infty(y)R^{-j},
 \qquad 1\le j\le N+2.
\]
For $0\le t\le cR$ and $|y-y_0|\le cR$, one has
\begin{equation}\label{eq:leaf-expansion}
 u_\infty(t,y)
 =a_\infty(y)h\bigl(t/a_\infty(y)\bigr)+e_\infty(t,y),
\end{equation}
where
\begin{equation}\label{eq:leaf-remainder}
 |\partial_t^m\partial_y^ne_\infty(t,y)|
 \le C_{mn}a_\infty(y)^3R^{-2-n}
                    (a_\infty(y)+t)^{-m},
 \qquad m+n\le N.
\end{equation}
The constants may depend on $N$ and the indicated derivative
orders, but are independent of $\beta/R$. The geometric constant $c$ and the
lower threshold for $|y_0|$ can be chosen independently of $N$.
\end{proposition}

\begin{proof}
Choose $c>0$ so that the enlarged rectangle
\[
 0\le t\le2cR,\qquad |y-y_0|\le2cR
\]
lies inside $B_{R/4}(X_0)$, where the direct derivative estimates
hold. Abbreviate $u=u_\infty$, $a=a_\infty$, and $e=e_\infty$.

\medskip
\noindent\emph{1. Smoothing radius and the forced radial equation.}
Evaluating Proposition~\ref{prop:direct-axial-regularity}
at $t=0$, where $a(y)=u(0,y)$, gives
\[
 a\approx\beta,\qquad |a^{(j)}|\le C_j\beta R^{-j},
 \qquad 1\le j\le N+2.
\]
The local fixed-leaf comparison of
Lemma~\ref{lem:direct-leaf-comparison} and the direct derivative
estimates of Proposition~\ref{prop:direct-axial-regularity} give
\begin{equation}\label{eq:leaf-u-derivative-bounds}
 \begin{aligned}
 |\partial_t^m(u-t)|
 &\le C_m\beta^3(\beta+t)^{-2-m},\\
 |\partial_t^m\partial_y^ju|
 &\le C_{mj}\beta^3R^{-j}(\beta+t)^{-2-m},
 \qquad 1\le j\le N+2,
 \end{aligned}
\end{equation}
for every fixed transverse derivative order $m\ge0$.
The same estimates hold for $u_a(t,y)=a(y)h(t/a(y))$. Write $D_u=1+u_t^2+u_y^2$. Rearranging the SME
\eqref{cap:eq:Q} gives
\begin{equation}\label{eq:leaf-forced-radial}
 \frac{u_{tt}}{1+u_t^2}+\frac3t u_t-\frac3u=F,
\end{equation}
where
\[
 F=-\frac{u_t^2u_y^2}{(1+u_t^2)D_u}u_{tt}
   +\frac{2u_tu_y}{D_u}u_{ty}
   -\frac{1+u_t^2}{D_u}u_{yy}.
\]
In particular,
\[
 |F|
 \le C\bigl(|u_y|^2|u_{tt}|+|u_yu_{ty}|+|u_{yy}|\bigr)
 \le C\beta^3R^{-2}(\beta+t)^{-2}.
\]
Differentiating this identity and using
\eqref{eq:leaf-u-derivative-bounds} yields
\begin{equation}\label{eq:leaf-radial-forcing-bounds}
 |\partial_t^m\partial_y^nF|
 \le C_{mn}\beta^3R^{-2-n}(\beta+t)^{-2-m},
 \qquad 0\le n\le N,
\end{equation}
for every fixed $m\ge0$.
At most $n+2\le N+2$ axial derivatives of $u$ are needed,
because the forcing contains at most two axial derivatives of $u$.
The direct mixed estimates already control every fixed transverse
order, so increasing $m$ requires no further axial derivatives.

\medskip
\noindent\emph{2. The equation for the remainder.}
Define $e=u-u_a$. Since $u_a$ solves the radial
Hardt-Simon equation, subtraction gives
\begin{equation}\label{eq:leaf-remainder-ode}
 \mathsf a\,e_{tt}+\mathsf b\,e_t+\mathsf c\,e=F,
 \qquad e(0,y)=e_t(0,y)=0.
\end{equation}
More explicitly, with $v_\theta=u_a+\theta e$,
\[
 \begin{aligned}
 \mathsf a
 &=\int_0^1\frac{d\theta}{1+(v_\theta)_t^2},\\
 \mathsf b
 &=\frac3t-\int_0^1
       \frac{2(v_\theta)_t(v_\theta)_{tt}}
            {(1+(v_\theta)_t^2)^2}\,d\theta,\\
 \mathsf c
 &=\int_0^1\frac3{v_\theta^2}\,d\theta.
 \end{aligned}
\]
The bounds already established for $u$ and $u_a$ imply
$\mathsf a\ge c_*>0$ and
\begin{equation}\label{eq:leaf-radial-coefficients}
 \begin{aligned}
 &|\partial_t^m\partial_y^j\mathsf a|
 +(\beta+t)
       |\partial_t^m\partial_y^j(\mathsf b-3/t)|\\
 &\hspace{35mm}
 +(\beta+t)^2|\partial_t^m\partial_y^j\mathsf c|
 \le C_{mj}R^{-j}(\beta+t)^{-m},
 \qquad 0\le j\le N.
 \end{aligned}
\end{equation}
At the axis, $\mathsf a(0,y)=1$,
$\mathsf b-3/t$ is smooth and odd in $t$, and
$\mathsf a,\mathsf c$ are smooth and even.
All equations at $t=0$ are understood through this smooth
radial extension.

For $t\ge T\beta$, with $T$ a sufficiently large fixed constant,
\eqref{eq:leaf-u-derivative-bounds} gives
\begin{equation}\label{eq:leaf-radial-tail-coefficients}
 \mathsf a=\frac12+O(\beta^3t^{-3}),\qquad
 \mathsf b=\frac3t+O(\beta^3t^{-4}),\qquad
 \mathsf c=\frac3{t^2}+O(\beta^3t^{-5}).
\end{equation}
These estimates hold with their differentiated versions.
Since $\beta/R\to0$ as $|y_0|\to\infty$, we may take $|y_0|$
sufficiently large that $2T\beta<cR$.

\medskip
\noindent\emph{3. A uniform estimate for the radial equation.}
For the fixed coefficients above, consider a smooth even solution
\[
 \mathsf a z_{tt}+\mathsf b z_t+\mathsf c z=g,
 \qquad z(0)=z_t(0)=0,
\]
at a fixed value of $y$. We claim that
\begin{equation}\label{eq:leaf-radial-response}
 |\partial_t^mg|\le K_mM(\beta+t)^{-2-m}
 \quad\Longrightarrow\quad
 |\partial_t^mz|\le C_mM(\beta+t)^{-m}
\end{equation}
for every fixed $m\ge0$. The constants depend on finitely many
$K_m$ and coefficient bounds, but not on $\beta/R$.

First consider the smoothing region. Set $r=t/\beta$ and
$Z(r)=z(\beta r)$. After division by the leading coefficient,
the equation becomes
\[
 Z''+\left(\frac3r+b_1(r)\right)Z'+c_1(r)Z=g_1(r),
 \qquad Z(0)=Z'(0)=0.
\]
On $0\le r\le2T$, the coefficients have uniform smooth bounds,
$b_1$ is odd with $|b_1(r)|\le Cr$, and $c_1,g_1$ are even,
with $|g_1|\le CM$.
Writing
\[
 J(r)=r^3\exp\left(\int_0^r b_1(s)\,ds\right),
\]
we obtain
\[
 Z'(r)=J(r)^{-1}\int_0^r
           J(s)\bigl(g_1(s)-c_1(s)Z(s)\bigr)\,ds.
\]
Since $J(r)\approx r^3$ on this fixed interval,
integration and Gronwall's inequality give
\[
 |Z|+|Z'|\le CM.
\]
The equation is a smooth radial elliptic equation in $\R^4$:
the coefficient $b_1(r)/r$ extends smoothly across the origin.
Interior elliptic estimates therefore give
$|\partial_r^m Z|\le C_mM$ on $0\le r\le T$.
This proves \eqref{eq:leaf-radial-response} in the smoothing region.

For $t\ge T\beta$, put $\tau=\log(t/\beta)$ and
$\mathcal Z=(z,\partial_\tau z)$.
By \eqref{eq:leaf-radial-tail-coefficients}, the equation becomes
\[
 \partial_\tau\mathcal Z
 =\left[
   \begin{pmatrix}0&1\\-6&-5\end{pmatrix}
   +O(e^{-3\tau})\right]\mathcal Z+\mathcal G,
 \qquad |\mathcal G|\le CM.
\]
The constant matrix has eigenvalues $-2$ and $-3$.
Since the perturbation is integrable in $\tau$, its evolution
operator satisfies
\[
 |\Phi(\tau,\sigma)|\le Ce^{-2(\tau-\sigma)},
 \qquad \tau\ge\sigma\ge\log T.
\]
Variation of constants, together with the estimate in the smoothing region at
$t=T\beta$, gives
\[
 |\mathcal Z(\tau)|
 \le Ce^{-2(\tau-\log T)}|\mathcal Z(\log T)|
      +CM\int_{\log T}^{\tau}e^{-2(\tau-\sigma)}\,d\sigma
 \le CM.
\]
Differentiating the system in $\tau$ gives bounds for every
fixed $\tau$-derivative, because the differentiated coefficients
and forcing are uniformly controlled.
Converting back to $t$ proves
\eqref{eq:leaf-radial-response} in the remaining region.

Applying \eqref{eq:leaf-radial-response} to
\eqref{eq:leaf-remainder-ode}, with $M=\beta^3R^{-2}$, yields
\[
 |\partial_t^me|\le C_m\beta^3R^{-2}(\beta+t)^{-m}
\]
for every fixed $m\ge0$.

\medskip
\noindent\emph{4. Mixed derivatives.}
We induct on $n=0,\ldots,N$, proving the stronger intermediate
statement
\begin{equation}\label{eq:leaf-remainder-induction}
 |\partial_t^m\partial_y^ne|
 \le C_{mn}\beta^3R^{-2-n}(\beta+t)^{-m}
 \quad\text{for every fixed }m\ge0.
\end{equation}
Keeping all transverse orders at each step ensures that the
second transverse derivatives in the commutators are already
controlled.

Put $e_n=\partial_y^ne$. Differentiating
\eqref{eq:leaf-remainder-ode} gives
\[
 \begin{aligned}
 (\mathsf a\partial_t^2+\mathsf b\partial_t+\mathsf c)e_n
 &=\partial_y^nF
   -\sum_{j=1}^n\binom nj
      \bigl[
       (\partial_y^j\mathsf a)\partial_t^2
       +(\partial_y^j\mathsf b)\partial_t
       +\partial_y^j\mathsf c
      \bigr]e_{n-j},\\
 e_n(0,y)&=\partial_t e_n(0,y)=0.
 \end{aligned}
\]
For $j\ge1$, the singular term $3/t$ has zero $y$-derivative.
Thus \eqref{eq:leaf-radial-coefficients}, the induction
hypothesis, and \eqref{eq:leaf-radial-forcing-bounds} show
that the complete right-hand side, denoted by $g_n$, satisfies
\[
 |\partial_t^mg_n|
 \le C_{mn}\beta^3R^{-2-n}(\beta+t)^{-2-m}
\]
for every fixed $m\ge0$.
It is smooth and even at the axis.
Applying \eqref{eq:leaf-radial-response} with
$M=\beta^3R^{-2-n}$ proves
\eqref{eq:leaf-remainder-induction} and closes the induction.

Finally, $a_\infty(y)\approx\beta$ throughout the band.
Replacing $\beta$ by this comparable quantity and restricting
to $m+n\le N$ gives \eqref{eq:leaf-remainder}.
The choices of the smoothing region and the tail transition $T$
use only fixed low-order coefficient bounds.
Consequently the requirement $2T\beta<cR$ is independent of $N$;
higher derivatives change only the constants in the estimates.
Together with the smoothing-radius estimates, this proves the proposition.
Since $R=\theta_0|y_0|$, taking $y_0=y$ also gives the corresponding
bounds with $R$ replaced by $|y|$, on $0\le t\le c|y|$ after
adjusting the fixed constant $c$.
\end{proof}

\section{Gluing estimates}\label{sec:corrector-estimates}
The gluing argument uses two analytic operations, established here.
First, we invert the radial linearized Hardt--Simon operator with
zero data at the axis, obtaining the growth and derivative bounds for each
corrector profile. We then prove the full corrector estimates and
the connector matching statement used in
Section~\ref{sec:finite-correctors}. Finally, we extend a prescribed mean-curvature
correction off the hypersurface to an ambient conformal factor.
The weighted bounds in this extension record the loss from taking
derivatives at the shrinking smoothing scale. They explain why the
high-order cancellations in Section~\ref{sec:accumulation} are
needed for a smooth metric at the accumulation point.

\subsection{Linearized Hardt-Simon operator}
In the next lemma we prove estimates for solutions of the linearized Hard-Simon ODE.
\begin{lemma}[Radial inverse for the linearized Hardt-Simon operator]
\label{lem:linearized-hs-inverse}
Let $h$ be the normalized Hardt--Simon profile in
\eqref{cap:eq:HS}, and let
\[
 L_hw
 =\frac{w''}{1+h'^2}
  +\left(\frac3\rho
         -\frac{2h'h''}{(1+h'^2)^2}\right)w'
  +\frac3{h^2}w.
\]
Suppose $b\ge0$ and $f$ is smooth and even, with
\[
 |f^{(m)}(\rho)|
 \le C_m(1+\rho)^{b-2-m},
 \qquad \rho\ge0,\quad m\ge0.
\]
Then there is a unique smooth even solution of
\[
 L_hw=f,\qquad w(0)=w'(0)=0.
\]
It is defined on $[0,\infty)$ and satisfies
\[
 |w^{(m)}(\rho)|
 \le C_{b,m}\left(\sum_{j=0}^m C_j\right)
             (1+\rho)^{b-m},
 \qquad \rho\ge0,\quad m\ge0,
\]
where $C_{b,m}$ depends only on $b$, $m$, and the fixed
profile $h$.
\end{lemma}

\begin{proof}
The dilation Jacobi field
\[
 d(\rho)=h(\rho)-\rho h'(\rho)
\]
satisfies $L_hd=0$, as follows by differentiating the
family $\lambda h(\rho/\lambda)$ at $\lambda=1$.
Moreover, $d$ is strictly positive. Indeed, $h(\rho)>\rho$
because the leaf lies on the $|z|>|x|$ side. Since
$h'(0)=0$, a first point at which $h'=1$ would satisfy
\[
 h''=6\left(\frac1h-\frac1\rho\right)<0,
\]
contradicting a first crossing from below. Hence $h'<1$
and $d>0$.

The differentiated asymptotics of $h$ give
\[
 d(\rho)=3b_H\rho^{-2}+O(\rho^{-3})
 \quad\text{as }\rho\to\infty.
\]
Together with positivity and smoothness on compact
intervals, this implies
\[
 d(\rho)\approx(1+\rho)^{-2},\qquad
 |d^{(m)}(\rho)|\le C_m(1+\rho)^{-2-m}.
\]

Introduce the integrating factor
\[
 \mu(\rho)
 =\frac{\rho^3h(\rho)^3}{(1+h'(\rho)^2)^{3/2}}.
\]
Using the equation for $h$, one checks that
\[
 L_hw=f
 \quad\Longleftrightarrow\quad
 (\mu w')'
 +\frac{3\mu(1+h'^2)}{h^2}w
 =\mu(1+h'^2)f.
\]
Since $L_hd=0$, reduction of order gives
\[
 \left(\mu d^2\left(\frac wd\right)'\right)'
 =\mu(1+h'^2)d f.
\]
The zero data at the axis therefore lead to the formula
\[
 w(\rho)
 =d(\rho)\int_0^\rho
   \frac{1}{\mu(r)d(r)^2}
   \left(
     \int_0^r
       \mu(s)(1+h'(s)^2)d(s)f(s)\,ds
   \right)\,dr.
\]

This formula defines a solution for every $\rho\ge0$.
Near zero, $\mu(\rho)/\rho^3$ and $d(\rho)$ are smooth,
positive, and even. The inner integral is consequently
$r^4$ times a smooth even function, and the outer
integrand is $r$ times a smooth even function.
Thus $w$ is smooth and even, with $w(0)=w'(0)=0$.
Conversely, every regular solution with zero data at the axis
must satisfy this formula, proving uniqueness.

We next estimate the solution. The leaf bounds imply
\[
 \mu(\rho)\approx\rho^3(1+\rho)^3,
 \qquad
 \mu(\rho)d(\rho)^2
 \approx\frac{\rho^3}{1+\rho}.
\]
Hence, for $b\ge0$,
\[
 \left|
 \int_0^r
   \mu(s)(1+h'(s)^2)d(s)f(s)\,ds
 \right|
 \le C_b C_0\,r^4(1+r)^{b-1}.
\]
Substituting this into the solution formula yields
\[
 |w(\rho)|
 \le C_b C_0(1+\rho)^{-2}
       \int_0^\rho r(1+r)^b\,dr
 \le C_b C_0(1+\rho)^b.
\]
Differentiating the formula and using
$|d'(\rho)|\le C(1+\rho)^{-3}$ likewise gives
\[
 |w'(\rho)|\le C_b C_0(1+\rho)^{b-1}.
\]

On $0\le\rho\le1$, rescaling both integration variables
to fixed intervals and differentiating the formula
bounds $w^{(m)}$ by a constant times
$\sum_{j=0}^m C_j$.
On $\rho\ge1$, divide the equation by its leading
coefficient. The coefficient of $w'$ and its $k$th
derivative are $O(\rho^{-1-k})$, while those of $w$
are $O(\rho^{-2-k})$.
The bounds already proved for $w,w'$, followed by
successive differentiation of the equation, therefore
give
\[
 |w^{(m)}(\rho)|
 \le C_{b,m}\left(\sum_{j=0}^m C_j\right)
             (1+\rho)^{b-m}.
\]
Combining the two regions proves the lemma.
\end{proof}
\subsection{Corrector estimates and connector matching}\label{app:corrector-proofs}
We give the complete statement underlying
Proposition~\ref{prop:correctors-informal} and its proof, followed by
the proof of Corollary~\ref{cor:arbitrary-order-matching}.
The profiles are those defined in
Section~\ref{sec:finite-correctors}.

\begin{proposition}[Correctors for a prescribed radius function]
\label{prop:correctors-estimates}
The recursively defined profiles $P_{2\ell,\gamma}$ depend only
on the normalized Hardt--Simon profile $h$, and are independent
of the prescribed radius function $a$.
For every $\ell\ge1$ and $\gamma\in\mathcal I_{2\ell}$,
they are smooth and even, satisfy
\[
 P_{2\ell,\gamma}(0)=P'_{2\ell,\gamma}(0)=0,
\]
and obey
\[
 |\partial_\rho^mP_{2\ell,\gamma}(\rho)|
 \le C_{\ell,\gamma,m}(1+\rho)^{2\ell-2-m},
 \qquad \rho\ge0,\quad m\ge0.
\]

Fix $C_1>0$ and constants $M_n\ge1$, $n\ge1$.
For each fixed $J\ge1$, there is an $\varepsilon_J>0$,
depending only on $J$, $C_1$, and
$M_1,\ldots,M_{2J+2}$, with the following property.

Let $y_0\in\R$, $R>0$, and let
\[
 a\in C^\infty(I_{2R};(0,\infty)),
 \qquad I_{2R}=(y_0-2R,y_0+2R),
\]
be any prescribed function satisfying
\[
 |a^{(n)}(y)|\le M_n a(y)R^{-n}
 \quad(n\ge1),\qquad
 \sup_{y\in I_{2R}}\frac{a(y)}R\le\varepsilon_J.
\]

Then $V_J[a]$, defined in \eqref{eq:corrector-inductive-def}, is positive on
\[
 0\le t<(C_1+1)R,\qquad y\in I_{2R},
\]
and $V_J[a](|x|,y)$ is smooth in Cartesian coordinates.
Throughout the interior region
\[
 0\le t\le C_1R,\qquad |y-y_0|\le R,
\]
one has
\[
 V_J[a](t,y)\approx a(y)+t,\qquad
 |\partial_tV_J[a](t,y)|\le C_J,\qquad
 |\partial_yV_J[a](t,y)|
 \le C_J\frac{a(y)}R.
\]
Moreover, for every $m,n\ge0$,
\begin{align}
 |\partial_t^m\partial_y^nW_\ell[a](t,y)|
 &\le C_{\ell,m,n}\,a(y)^3R^{-2\ell-n}
                  (a(y)+t)^{2\ell-2-m},
 \qquad 1\le\ell\le J,
 \label{glue:eq:correction}\\
 |\partial_t^m\partial_y^nQ(V_J[a])(t,y)|
 &\le C_{J,m,n}\,a(y)^3R^{-2J-2-n}
                  (a(y)+t)^{2J-2-m}.
 \label{glue:eq:defect}
\end{align}

Define
\[
 \Sigma_J=
 \left\{
 (x,z,y)\in\R^4\times\R^4\times\R:
 |x|<(C_1+1)R,\quad y\in I_{2R},\quad
 |z|=V_J[a](|x|,y)
 \right\}.
\]
Its Euclidean scalar mean curvature satisfies, for every
$j\ge0$,
\[
 |\nabla_{\Sigma_J}^{\,j}H_{\Sigma_J}(X)|
 \le C_{J,j}\,a(y)^3R^{-2J-2}
                 (a(y)+|x|)^{2J-2-j}
\]
whenever $X=(x,z,y)\in\Sigma_J$,
$|x|\le C_1R$, and $|y-y_0|\le R$.

All constants are independent of $a$, $y_0$, and $R$,
apart from the prescribed modulation bounds $M_n$.
Each constant depends on only finitely many of these bounds.
The smallness threshold $\varepsilon_J$ is independent of
the derivative orders $m,n,j$.
\end{proposition}
\begin{proof}[Proof of Proposition~\ref{prop:correctors-estimates}]
Throughout the proof, $a$ is an arbitrary prescribed function
satisfying the hypotheses. In particular,
\[
 |(\log a)'|\le M_1R^{-1},
\]
so
\[
 e^{-2M_1}a(y_0)\le a(y)\le e^{2M_1}a(y_0),
 \qquad y\in I_{2R}.
\]
We work on the enlarged region
\[
 0\le t<(C_1+1)R,\qquad y\in I_{2R},
\]
and write
\[
 \rho=\frac{t}{a(y)},\qquad
 q_j=a^{j-1}a^{(j)},\qquad s=a(y)+t.
\]
We may decrease $\varepsilon_J$ below $1$ when necessary.
In particular, $s\le (C_1+2)R$ on this region.

We first establish the universal profile estimates.
The identities
\[
 a\partial_y\rho=-q_1\rho,\qquad
 a\partial_yq_j=(j-1)q_1q_j+q_{j+1}
\]
show that an axial differentiation, multiplied by $a$,
raises the weight by one. These are identities for every
positive smooth function $a$.

The differentiated leaf estimates give
\[
 |\partial_\rho^m(h-\rho h')|
 \le C_m(1+\rho)^{-2-m}.
\]
Thus the weight-one coefficient of
\[
 \partial_y\bigl(ah(t/a)\bigr)=q_1(h-\rho h')
\]
has the required decay.

Suppose the profiles of weights below $2\ell$ have already
been constructed with the asserted bounds.
A coefficient of weight $k$ in $V_{\ell-1}/a-h$
has radial growth $(1+\rho)^{k-2}$, with one power lost
for each radial derivative.
The differentiation identities therefore show that a
coefficient of weight $k$ in $(V_{\ell-1})_y$ has growth
$(1+\rho)^{k-3}$, and one in
$a(V_{\ell-1})_{yy}$ has growth $(1+\rho)^{k-4}$.
These statements include all differentiated bounds.

Substituting into $aQ(V_{\ell-1})$ and expanding the
denominators about their weight-zero terms $h$ and
$1+h'^2$ gives
\[
 |\partial_\rho^mS_{2\ell,\gamma}(\rho)|
 \le C_{\ell,\gamma,m}(1+\rho)^{2\ell-4-m}.
\]
Indeed, the second-derivative terms have this order,
while additional products of positive-weight corrections
have further radial decay.
The coefficients are smooth and even.
At the axis, terms involving $\rho^{-1}$ have smooth
even extensions because their numerators are odd.

Lemma~\ref{lem:linearized-hs-inverse}, applied with
$b=2\ell-2$ and $f=-S_{2\ell,\gamma}$, now gives the
unique smooth even solution of
\[
 L_hP_{2\ell,\gamma}=-S_{2\ell,\gamma},
 \qquad
 P_{2\ell,\gamma}(0)=P'_{2\ell,\gamma}(0)=0,
\]
with
\[
 |\partial_\rho^mP_{2\ell,\gamma}(\rho)|
 \le C_{\ell,\gamma,m}(1+\rho)^{2\ell-2-m}.
\]
This closes the induction.
In particular, the construction involves only the formal
axial jets and the fixed radial operator $L_h$; the resulting
profiles are independent of the choice of $a$.

We next evaluate these profiles on the prescribed function.
The modulation hypotheses imply
\[
 |\partial_y^n q_j|
 \le C_{j,n}\left(\frac{a(y)}R\right)^jR^{-n},
\]
and hence, for $\gamma\in\mathcal I_{2\ell}$,
\[
 |\partial_y^nq^\gamma|
 \le C_{\ell,n}
      \left(\frac{a(y)}R\right)^{2\ell}R^{-n}.
\]
Using the formula for $W_\ell$, the profile estimates, and
$\partial_y\rho=-(a'/a)\rho$, we obtain
\[
 \begin{aligned}
 |\partial_t^m\partial_y^nW_\ell[a]|
 &\le C_{\ell,m,n}\,
       a(y)^{1-m}
       \left(\frac{a(y)}R\right)^{2\ell}
       R^{-n}(1+\rho)^{2\ell-2-m}\\
 &=C_{\ell,m,n}\,
       a(y)^3R^{-2\ell-n}
       (a(y)+t)^{2\ell-2-m}.
 \end{aligned}
\]
This proves \eqref{glue:eq:correction}.

Since $s\le (C_1+2)R$, these estimates also yield
\[
 \sum_{\ell=1}^J
 \left(
 \frac{|W_\ell|}{s}+|\partial_tW_\ell|
 \right)
 \le C_J\left(\frac aR\right)^2,
 \qquad
 \sum_{\ell=1}^J|\partial_yW_\ell|
 \le C_J\left(\frac aR\right)^3.
\]
On the other hand,
\[
 ah(t/a)\approx s,\qquad
 \partial_t(ah(t/a))=h'(\rho),\qquad
 \partial_y(ah(t/a))=a'(h-\rho h').
\]
Choosing $\varepsilon_J$ sufficiently small therefore gives
positivity, comparability, and the asserted first-derivative
bounds for $V_J[a]$ on the enlarged region.
This choice uses only finitely many modulation bounds and
does not depend on the derivative orders appearing later.
The axis identities follow from $h(0)=1$, $h'(0)=0$,
and the zero data of the correctors.
Evenness gives smoothness in Cartesian $x$ coordinates.

We now estimate the residual.
The recursive construction cancels all terms of weight at
most $2J$ in $aQ(V_J[a])$.
Only even weights occur, so the first possible remaining
weight is $2J+2$.
To turn this formal cancellation into a uniform estimate,
we use a finite Taylor expansion.

At a fixed $y$, put $\varepsilon=a(y)/R$.
Express $aQ(V_J[a])$ in
$(\rho,q_1,\ldots,q_{2J+2})$, and let $E_\delta$ be
the resulting expression after the substitution
$q_j\mapsto\delta^jq_j$, $0\le\delta\le1$.
Then
\[
 E_1=aQ(V_J[a]),\qquad
 \partial_\delta^kE_\delta\big|_{\delta=0}=0
 \quad(0\le k\le2J+1).
\]

A weight-$k$ contribution to $V_J/a-h$ and its first
two radial derivatives is bounded, respectively, by
\[
 (1+\rho)^{-2}\bigl(\varepsilon(1+\rho)\bigr)^k,\qquad
 (1+\rho)^{-3}\bigl(\varepsilon(1+\rho)\bigr)^k,\qquad
 (1+\rho)^{-4}\bigl(\varepsilon(1+\rho)\bigr)^k.
\]
The corresponding contributions to $(V_J)_y$ and
$a(V_J)_{yy}$ are bounded by
\[
 (1+\rho)^{-3}\bigl(\varepsilon(1+\rho)\bigr)^k,
 \qquad
 (1+\rho)^{-4}\bigl(\varepsilon(1+\rho)\bigr)^k.
\]
All these estimates hold with differentiated radial bounds.

After subtracting the equation for $h$, the residual is
therefore $(1+\rho)^{-4}$ times a rational expression
in the weighted variables, with uniformly controlled
coefficients. The relevant normalized quantities satisfy
\[
 \varepsilon(1+\rho)=\frac{s}{R}\le C_1+2,
 \qquad
 \left|\varepsilon^{-j}q_j\right|
 =\left|\frac{R^ja^{(j)}}a\right|\le M_j.
\]
Along the entire interpolation $0\le\delta\le1$, the
normalized profile divided by $h$ equals
$1+O_J(\varepsilon^2)$.
The slope denominator is bounded below by $1$.
Thus all inverse denominators and their differentiated
expressions are uniformly controlled.

Consequently,
\[
 |\partial_\rho^m\partial_\delta^{2J+2}E_\delta|
 \le C_{J,m}\,
       \varepsilon^{2J+2}(1+\rho)^{2J-2-m},
 \qquad 0\le\delta\le1.
\]
Taylor's formula with integral remainder gives
\[
 aQ(V_J[a])
 =\frac1{(2J+1)!}
   \int_0^1(1-\delta)^{2J+1}
       \partial_\delta^{2J+2}E_\delta\,d\delta.
\]

The same argument applies after axial differentiation.
Indeed, the modulation hypotheses give
\[
 |\partial_y^n\varepsilon|
 \le C_n\varepsilon R^{-n},\qquad
 |\partial_y^n\rho|
 \le C_n\rho R^{-n},
\]
and
\[
 \left|
 \partial_y^n\left(\frac{R^ja^{(j)}}a\right)
 \right|
 \le C_{j,n}R^{-n}.
\]
Thus each axial differentiation contributes a factor
$R^{-1}$ to the preceding bounds.
Differentiating the integral remainder and dividing by $a$
therefore yields
\[
 \begin{aligned}
 |\partial_t^m\partial_y^nQ(V_J[a])|
 &\le C_{J,m,n}\,
       a^{-1-m}R^{-n}
       \left(\frac aR\right)^{2J+2}
       (1+\rho)^{2J-2-m}\\
 &=C_{J,m,n}\,
       a(y)^3R^{-2J-2-n}
       (a(y)+t)^{2J-2-m}.
 \end{aligned}
\]
This proves \eqref{glue:eq:defect}.
Higher derivative orders require higher modulation bounds,
but do not require a smaller $\varepsilon_J$.

Finally, the correction estimates and the differentiated
leaf bounds imply
\[
 |D_{x,y}^kV_J[a](|x|,y)|
 \le C_{J,k}(a(y)+|x|)^{1-k},
 \qquad k\ge1.
\]
Near the axis, these follow from the smooth even radial
extension; away from it, they follow from the usual radial
differentiation formulas.

Since $V_J[a]\approx a+|x|$, the parametrization
\[
 (x,y,\omega)\longmapsto
 \bigl(x,V_J[a](|x|,y)\omega,y\bigr),
 \qquad \omega\in S^3,
\]
gives uniformly controlled local charts after rescaling
by $a(y)+|x|$.
On patches of a sufficiently small fixed multiple of this
radius, the rescaled induced metric, its inverse, and
their derivatives of every fixed order are uniformly
bounded.
The enlarged domain provides these patches throughout
the stated interior region.

In these coordinates,
\[
 H_{\Sigma_J}
 =-\frac{Q(V_J[a])}{\sqrt{1+|DV_J[a]|^2}}.
\]
The differentiated bounds for $V_J[a]$ control the
denominator. Since $a(y)+|x|\le CR$, each axial factor
$R^{-1}$ is bounded by $C(a(y)+|x|)^{-1}$.
The defect estimate therefore gives the coordinate
derivative bounds
\[
 |D^jH_{\Sigma_J}|
 \le C_{J,j}\,a(y)^3R^{-2J-2}
                 (a(y)+|x|)^{2J-2-j}.
\]
The uniformly controlled rescaled metric and connection
give the same bounds for covariant derivatives, proving
the final assertion.
\end{proof}

\begin{proof}[Proof of Corollary~\ref{cor:arbitrary-order-matching}]
Write $u=u_\infty$ throughout this proof.

Choose $\theta_0$ sufficiently small that the connector
estimates hold on the enlarged rectangle
\[
 0\le t\le2R,\qquad |y-y_0|\le2R.
\]
We prove the estimate
on $0\le t\le R$, $|y-y_0|\le R$, which contains the region
in the statement.

We induct on $J$, retaining every fixed mixed derivative
order at each stage. The case $J=0$ is
Proposition~\ref{prop:precise-estimates-for-connector}.
Suppose the assertion holds at order $J-1$, and set
\[
 E=u-V_J[a].
\]
The induction hypothesis and \eqref{glue:eq:correction}
first give the weaker estimate
\[
 |\partial_t^m\partial_y^nE|
 \le C_{J,m,n}a(y)^3R^{-2J-n}
                  (a(y)+t)^{2J-2-m}.
\]
Moreover, because $a(y)=u(0,y)$ and all corrector profiles
have zero data at the axis,
\[
 E(0,y)=E_t(0,y)=0.
\]

Separate the axial terms of $Q$ by writing
\[
 Q(v)=\frac{v_{tt}}{1+v_t^2}+\frac3t v_t-\frac3v
       +\mathcal T(v),
\]
where
\[
 \mathcal T(v)
 =
 \frac{v_t^2v_y^2v_{tt}}
      {(1+v_t^2)(1+v_t^2+v_y^2)}
 -\frac{2v_tv_yv_{ty}}{1+v_t^2+v_y^2}
 +\frac{(1+v_t^2)v_{yy}}{1+v_t^2+v_y^2}.
\]
The established estimates for $u$ and $V_J$, and hence for
their interpolations, imply
\[
 |v_t|\le C_J,\qquad
 |v_{tt}|\le \frac{C_J}{a+t},\qquad
 |v_y|\le C_J\frac{a+t}{R},\qquad
 |v_{ty}|\le \frac{C_J}{R},\qquad
 |v_{yy}|\le C_J\frac{a+t}{R^2},
\]
together with their differentiated versions.
Linearizing the displayed expression for $\mathcal T$
along the segment between $V_J$ and $u$ therefore gives
\[
 \begin{aligned}
 |\mathcal T(u)-\mathcal T(V_J)|
 \le C_J\bigg(
 &\frac{(a+t)^2}{R^2}|E_{tt}|
 +\frac{a+t}{R}|E_{ty}|+|E_{yy}|\\
 &+\frac{a+t}{R^2}|E_t|
 +\frac1R|E_y|
 \bigg).
 \end{aligned}
\]
The same coefficient bounds hold after differentiation,
with an additional factor $(a+t)^{-1}$ for each transverse
derivative and $R^{-1}$ for each axial derivative.
Consequently, the weaker estimate for $E$ yields
\[
 \left|
 \partial_t^m\partial_y^n
 \bigl(\mathcal T(u)-\mathcal T(V_J)\bigr)
 \right|
 \le C_{J,m,n}a(y)^3R^{-2J-2-n}
                    (a(y)+t)^{2J-2-m}.
\]
Here the induction hypothesis is used with up to two
additional axial derivatives.

Since $Q(u)=0$, subtraction of the radial parts gives
\[
 (\mathsf a\partial_t^2+\mathsf b\partial_t+\mathsf c)E
 =-Q(V_J)-\bigl(\mathcal T(u)-\mathcal T(V_J)\bigr),
 \qquad E(0,y)=E_t(0,y)=0.
\]
The coefficients are obtained by averaging the radial
linearization between $V_J$ and $u$, exactly as in
\eqref{eq:leaf-remainder-ode}. By \eqref{glue:eq:defect}
and the preceding estimate, the right-hand side satisfies
\[
 \left|
 \partial_t^m\partial_y^n
 \bigl[-Q(V_J)-\mathcal T(u)+\mathcal T(V_J)\bigr]
 \right|
 \le C_{J,m,n}a(y)^3R^{-2J-2-n}
                    (a(y)+t)^{2J-2-m}.
\]

We verify that the radial response argument used in
Proposition~\ref{prop:direct-leaf-expansion} applies to these
averaged coefficients. The estimates for $u$ give
\[
 |\partial_t^m\partial_y^n(u-t)|
 \le C_{m,n}a(y)^3R^{-n}(a(y)+t)^{-2-m}.
\]
The same bound, with constants depending on $J$, holds
for $V_J-t$: indeed, each corrector satisfies
\[
 |\partial_t^m\partial_y^nW_\ell|
 \le C_{\ell,m,n}a(y)^3R^{-n}(a(y)+t)^{-2-m}
       \left(\frac{a(y)+t}{R}\right)^{2\ell},
\]
and the last factor is bounded on the region under
consideration. Thus the averaged coefficients satisfy
the bounds \eqref{eq:leaf-radial-coefficients}, with
$\beta$ replaced by $a(y)$ and constants depending on $J$.
They also have the tail expansions
\[
 \mathsf a=\frac12+O_J(a^3t^{-3}),\qquad
 \mathsf b=\frac3t+O_J(a^3t^{-4}),\qquad
 \mathsf c=\frac3{t^2}+O_J(a^3t^{-5}),
\]
with their differentiated versions. At the axis,
$\mathsf a(0,y)=1$, $\mathsf b-3/t$ is smooth and odd,
and $\mathsf a,\mathsf c$ are smooth and even.

The radial response estimate
\eqref{eq:leaf-radial-response} extends to the polynomial
forcing occurring here. To see this, fix $y$. Rescaling
$t$ by $a(y)$ in the smoothing region gives a bound of size
\[
 a(y)^{2J+3}R^{-2J-2}
\]
for $E$ on a fixed interval in the smoothing region. On the tail, the
logarithmic variable $\tau=\log(t/a(y))$ gives the same
stable system with limiting eigenvalues $-2$ and $-3$
as in that proof. The forcing now grows like
$e^{2J\tau}$, and
\[
 \int_{\log T_J}^{\tau}
 e^{-2(\tau-\sigma)}e^{2J\sigma}\,d\sigma
 \le C_J e^{2J\tau},
\]
where $T_J$ is a fixed constant separating the inner and outer regions.
Since $a(y)/R\to0$, this transition lies inside the
available region once $|y_0|$ is sufficiently large.
Variation of constants and differentiation of the radial
equation therefore give
\[
 |\partial_t^mE|
 \le C_{J,m}a(y)^3R^{-2J-2}(a(y)+t)^{2J-m}
\]
for every fixed $m\ge0$.

Finally, differentiate the radial equation in $y$ and
induct on the axial derivative order. The commutators
are sums of terms of the form
\[
 (\partial_y^r\mathsf a)
       \partial_t^2\partial_y^{n-r}E,\qquad
 (\partial_y^r\mathsf b)
       \partial_t\partial_y^{n-r}E,\qquad
 (\partial_y^r\mathsf c)
       \partial_y^{n-r}E,
 \qquad 1\le r\le n.
\]
The coefficient bounds give
\[
 \begin{aligned}
 &|\partial_t^m\partial_y^r\mathsf a|
 +(a+t)|\partial_t^m\partial_y^r(\mathsf b-3/t)|\\
 &\hspace{20mm}
 +(a+t)^2|\partial_t^m\partial_y^r\mathsf c|
 \le C_{J,m,r}R^{-r}(a+t)^{-m}.
 \end{aligned}
\]
The singular term $3/t$ has zero axial derivative.
Hence the induction hypothesis in $n$, retaining every
fixed transverse order, bounds the differentiated forcing
and all commutators by
\[
 C_{J,m,n}a(y)^3R^{-2J-2-n}
                 (a(y)+t)^{2J-2-m}.
\]
The differentiated data at the axis still vanish. Applying the
same radial response estimate proves
\[
 |\partial_t^m\partial_y^nE|
 \le C_{J,m,n}a(y)^3R^{-2J-2-n}
                  (a(y)+t)^{2J-m}.
\]
This closes both inductions and proves
\eqref{glue:eq:matching}.
\end{proof}

\subsection{Metric corrections} In the next lemma we collect the estimates needed for the conformal metric correction.
\begin{lemma}[Conformal correction with weighted estimates]
\label{lem:conformal-correction}
Use the axial band, the function $a$, and the scale $R$ from
Proposition~\ref{prop:correctors-estimates}.
Let $V$ be a positive smooth even profile on
\[
 0\le t<(C_1+1)R,\qquad |y-y_0|<2R,
\]
and set
\[
 \Sigma=\{(x,z,y): |z|=V(|x|,y)\}.
\]
Suppose that throughout this enlarged region
\[
 V(t,y)\approx a(y)+t\le CR,\qquad
 |D_{x,y}^{\,k}V(|x|,y)|
 \le C_k(a(y)+|x|)^{1-k}\quad(k\ge1),
\]
and, for some fixed $J\ge1$,
\[
 |\nabla_\Sigma^{\,j}H_\Sigma|
 \le C_{J,j}a(y)^3R^{-2J-2}
                (a(y)+t)^{2J-2-j}
 \qquad(j\ge0).
\]
Here $H_\Sigma=\operatorname{div}_\Sigma\nu$, with $\nu$
pointing towards increasing $|z|$.

There is a smooth compactly supported function $F$ such that
$g=e^{2F}g_{\rm euc}$ makes $\Sigma$ minimal on
\[
 t\le C_1R,\qquad |y-y_0|\le R.
\]
Moreover, $F|_\Sigma=0$, its support lies in the enlarged band
and in
\[
 \frac34V(t,y)<|z|<\frac54V(t,y),
\]
and, for every $m\ge0$,
\begin{equation}
\label{eq:conformal-correction-scale}
 |D^mF(x,z,y)|+
 |D^m(g-g_{\rm euc})(x,z,y)|
 \le C_{J,m}a(y)^3R^{-2J-2}
                (a(y)+t)^{2J-1-m},
 \qquad t=|x|.
\end{equation}
The derivatives in this estimate are ambient Cartesian
derivatives. The constants are independent of $a$ and $R$,
apart from the constants in the assumed estimates.
\end{lemma}

\begin{proof}
Choose a smooth cutoff $\eta(t,y)$, supported in the enlarged
band and equal to one on the smaller region, such that
$\eta(|x|,y)$ is smooth and
\[
 |D_{x,y}^{\,k}\eta|\le C_kR^{-k}.
\]
Choose also $\chi\in C_c^\infty((-1,1))$ with $\chi=1$ near
zero. By rotational symmetry, write the mean curvature as
$H_\Sigma(t,y)$, and define
\[
 F(x,z,y)=
 -\frac{\eta(t,y)H_\Sigma(t,y)}
        {8\sqrt{1+V_t(t,y)^2+V_y(t,y)^2}}\,
 \bigl(|z|-V(t,y)\bigr)
 \chi\!\left(\frac{4(|z|-V(t,y))}{V(t,y)}\right).
\]
Extend this function by zero outside the indicated region.
The cutoff vanishes near $z=0$, while smoothness at $x=0$
follows from the smooth even profiles. The cutoff $\eta$
ensures smooth extension across the boundaries of the band.

On $\Sigma$,
\[
 F=0,\qquad
 \nu\bigl(|z|-V(t,y)\bigr)
 =\sqrt{1+V_t^2+V_y^2}.
\]
Consequently,
\[
 \nu F=-\frac18\eta H_\Sigma.
\]
Since $\dim\Sigma=8$, the conformal transformation formula gives
\[
 H_{\Sigma,g}
 =e^{-F}(H_\Sigma+8\nu F)
 =(1-\eta)H_\Sigma.
\]
Thus $\Sigma$ is minimal wherever $\eta=1$.

We next justify the ambient derivative estimates.
In the coordinates $(x,y,\omega)$, where
$z=V(|x|,y)\omega$ and $\omega\in S^3$, the induced metric is
\[
 \bigl(I+D_{x,y}V\otimes D_{x,y}V\bigr)
 \oplus V^2g_{S^3}.
\]
The bounds on $V$ imply that the connection coefficients
in the $(x,y)$ directions have $k$th derivatives bounded by
$C_k(a(y)+t)^{-1-k}$.
Since $H_\Sigma$ is independent of $\omega$, conversion of
its covariant derivatives to Cartesian $(x,y)$ derivatives
therefore gives
\[
 |D_{x,y}^{\,k}H_\Sigma(t,y)|
 \le C_{J,k}a(y)^3R^{-2J-2}
                (a(y)+t)^{2J-2-k}.
\]

On the support of the normal cutoff,
\[
 |z|\approx V(t,y)\approx a(y)+t.
\]
Hence, for every $k\ge0$,
\begin{align*}
 |D^k(|z|-V)|&\le C_k(a(y)+t)^{1-k},\\
 \left|D^k(1+V_t^2+V_y^2)^{-1/2}\right|
 &\le C_k(a(y)+t)^{-k},\\
 \left|D^k\chi\!\left(\frac{4(|z|-V)}{V}\right)\right|
 &\le C_k(a(y)+t)^{-k}.
\end{align*}
Also, $R^{-k}\le C_k(a(y)+t)^{-k}$ on the enlarged band.
Applying the product rule to the formula for $F$ yields
\[
 |D^mF|
 \le C_{J,m}a(y)^3R^{-2J-2}
                (a(y)+t)^{2J-1-m}.
\]

Finally, the undifferentiated factor satisfies
\[
 a(y)^3R^{-2J-2}(a(y)+t)^{2J-1}
 =
 \left(\frac{a(y)}R\right)^3
 \left(\frac{a(y)+t}R\right)^{2J-1}
 \le C_J.
\]
In differentiating $e^{2F}$, each product of derivatives
of $F$ is therefore bounded by the same weight as a single
derivative of the corresponding total order. Together with
$|e^{2F}-1|\le C_J|F|$, this proves
\eqref{eq:conformal-correction-scale} for the metric.
\end{proof}

\section{General quadratic cones}\label{sec:general-cones}

The proposed construction follows the proof of
Theorem~\ref{thm:main-accumulation}. We will insert successively
smaller copies of the conjectured connector and increase the
correction order as the insertion points approach zero.
We will then remove the remaining mean curvature by a conformal
change supported away from the exact connector cores.
The increasing correction order should ensure that this metric
change is smooth and flat at the accumulation point.
Carrying this out requires the corresponding arbitrary-order
end and matching estimates in addition to the existence of
the connectors.

The expected behavior at infinity is determined by the
same Jacobi calculation as in the present example. Put
\[
 m=p+q,\qquad
 \gamma=\frac{m-1-\sqrt{m^2-6m+1}}2,\qquad
 V=C_{p,q}\times\R,
\]
and write $r=(|x|^2+|z|^2)^{1/2}$. For the minimizing
quadratic cones, $1<\gamma\le2$. On rotationally invariant
functions, the Jacobi operator of $V$ is
\[
 L_V=\partial_{rr}+\frac mr\partial_r
       +\partial_{yy}+\frac m{r^2}.
\]
A direct calculation gives
\[
 L_V\left[
 r^{-\gamma}
 \left(y^2-\frac{r^2}{m+1-2\gamma}\right)
 \right]=0.
\]
This field has homogeneity $2-\gamma<1$ and therefore
describes a perturbation compatible with the cylindrical
blow-down.

More precisely, orient the normal of $V$ towards
$p|z|^2-q|x|^2>0$, and let $f_R$ be the normal height of
$R^{-1}M$ over the regular part of $V$. For the attachment
with tip $C_{p+1,q}$, we expect
\[
 R^{\gamma-1}f_R
 \longrightarrow
 A r^{-\gamma}
 \left(y^2-\frac{r^2}{m+1-2\gamma}\right),
 \qquad A>0,
\]
locally smoothly on $\operatorname{reg}V$.
The opposite attachment has the opposite sign.
When $p=q=3$, this reduces to
\[
 \gamma=2,\qquad
 A\left(\frac{y^2}{r^2}-\frac13\right).
\]

We expect to use the following replacements for the
ingredients of the present construction.

\begin{enumerate}[leftmargin=*,itemsep=0.7em]
\item
The transverse models will be the Hardt--Simon leaves
associated with $C_{p,q}$. Their dilations provide the
foliations on the two sides of the cone. For a radius graph
$|z|=v(t,y)$, with $t=|x|$, the equation becomes
\[
 \left(\delta_{ij}
       -\frac{v_i v_j}{1+|Dv|^2}\right)v_{ij}
 +\frac pt v_t-\frac qv=0,
 \qquad i,j\in\{t,y\}.
\]
The cylindrical profile is $\sqrt{q/p}\,t$.
The two sides must be treated with their respective leaf
models; they need not have identical expansions.

\item
The barriers and non-concentration estimates used here
have general versions in
\cite[Propositions~2.9 and~4.3]{SzCyl}.
These apply to strictly stable, strictly minimizing cones,
including the minimizing quadratic cones considered here.
They provide the replacements for the corresponding
Simons-cone inputs in our barrier and end arguments.

The phase analysis must also be adjusted. If $f$ is an
invariant Jacobi field and $P=r^\gamma f$, then
\[
 P_{rr}+\frac{m-2\gamma}{r}P_r+P_{yy}=0.
\]
For $m=6$, this is the radial harmonic equation used in
the present paper. In general, the harmonic replacement
must be carried out for this weighted equation.
The quadratic growth estimate and the nonvanishing of
its quadratic coefficient still need to be proved.

\item
The existence argument for the Dirichlet problem on a bounded domain and the classification
of singularities must be adapted to the new equation.
The relevant density inequalities can be checked directly.

The link of $C_{p,q}$ is
\[
 S^p\!\left(\sqrt{\frac{p}{p+q}}\right)
 \times
 S^q\!\left(\sqrt{\frac{q}{p+q}}\right).
\]
Consequently, its normalized density is
\[
 \Theta(C_{p,q})
 =
 \frac{|S^p|\,|S^q|}{|S^{p+q}|}
 \left(\frac{p}{p+q}\right)^{p/2}
 \left(\frac{q}{p+q}\right)^{q/2}.
\]
Taking a product with $\R$ preserves the normalized density,
so
\[
 \Theta(C_{p,q}\times\R)=\Theta(C_{p,q}).
\]

We claim that the displayed density strictly decreases
when either index increases. To see this, set
\[
 \Phi(s)=\log\Gamma\!\left(\frac{s+1}{2}\right)
          -\frac{s}{2}\log s,\qquad s>0.
\]
The sphere-area formula gives
\[
 \log\Theta(C_{p,q})
 =\log(2\sqrt{\pi})
   +\Phi(p+q)-\Phi(p)-\Phi(q).
\]
Moreover, the integral representation of the derivatives
of $\log\Gamma$ yields
\[
 \Phi''(s)
 =\frac12\int_0^\infty e^{-su}
       \left(\frac{u}{\sinh u}-1\right)\,du<0.
\]
Thus $\Phi'$ is strictly decreasing, and
\[
 \log\frac{\Theta(C_{p+1,q})}{\Theta(C_{p,q})}
 =
 \int_0^1
 \bigl[\Phi'(p+q+s)-\Phi'(p+s)\bigr]\,ds
 <0.
\]
Interchanging $p$ and $q$ proves the other inequality.
Therefore both proposed attachments satisfy
\[
 \frac{\Theta(C_{p+1,q})}
      {\Theta(C_{p,q}\times\R)}<1,
 \qquad
 \frac{\Theta(C_{p,q+1})}
      {\Theta(C_{p,q}\times\R)}<1.
\]

These are precisely the density comparisons required by
the monotonicity formula.

\item
Once the connector and its quantitative end estimates are
available, we expect to repeat the corrector and matching
construction using the linearized operator of the
corresponding Hardt--Simon leaf.

The exponents change with $\gamma$. A leaf dilated by $a$
has leading normal coefficient proportional to
$a^{1+\gamma}$, so the natural phase has homogeneity
$1+\gamma$. A quadratic axial phase therefore predicts
\[
 a_\infty(Y)\sim\kappa |Y|^{2/(1+\gamma)},
 \qquad \kappa>0.
\]
Thus the cubic phase and the exponent $2/3$ in the present
paper are the special case $\gamma=2$.

The corresponding conformal estimate to establish is
\[
 |D^jF|
 \le C_{J,j}a(y)^{1+\gamma}R^{-2J-2}
                  (a(y)+t)^{2J+1-\gamma-j}.
\]
The conformal construction itself is unchanged, with the
dimension factor $8$ replaced by $m+2$:
\[
 H_{e^{2F}g_{\rm euc}}
 =e^{-F}\bigl(H+(m+2)\nu F\bigr).
\]
With these weighted estimates, the same increasing-order
choice would give a smooth metric flat at the accumulation
point.
\end{enumerate}

The general cylindrical estimates therefore supply several
of the analytic tools. The remaining work is to construct
the entire connectors, identify their isolated tangent
cones, and establish the weighted end and matching
estimates needed for the accumulating construction.

\end{document}